\documentclass[pdflatex,sn-mathphys-num]{sn-jnl}
\usepackage{graphicx}
\usepackage{amsmath,amssymb,amsfonts}
\usepackage{amsthm}
\usepackage{enumitem}
\usepackage{mathtools}
\usepackage{float}
\usepackage{xcolor}
\usepackage{placeins}
\usepackage{booktabs}
\usepackage{tabularx}
\usepackage{array}
\usepackage{ragged2e}
\usepackage{pifont}

\definecolor{GoodGreen}{RGB}{20,120,65}
\definecolor{BadRed}{RGB}{180,30,45}

\newcommand{\cmark}{\textcolor{GoodGreen}{\ding{51}}}
\newcommand{\xmark}{\textcolor{BadRed}{\ding{55}}}

\newcolumntype{Y}{>{\RaggedRight\arraybackslash}X}
\newcolumntype{C}[1]{>{\centering\arraybackslash}p{#1}}
\newcolumntype{L}[1]{>{\RaggedRight\arraybackslash}p{#1}}

\newsavebox{\amoSubfigureBox}
\newcounter{amosubfigure}
\renewcommand{\theamosubfigure}{\alph{amosubfigure}}
\AtBeginEnvironment{figure}{\setcounter{amosubfigure}{0}}
\newcommand{\subfloat}[2][]{%
  \stepcounter{amosubfigure}%
  \sbox{\amoSubfigureBox}{#2}%
  \begin{minipage}[b]{\wd\amoSubfigureBox}
    \centering
    \usebox{\amoSubfigureBox}\par
    \smallskip
    {\footnotesize(\theamosubfigure)~#1}
  \end{minipage}%
}

\definecolor{refblue}{rgb}{0.0, 0.4, 1.0}
\definecolor{citeblue}{rgb}{0.0, 0.1, 1.0}
\hypersetup{
    colorlinks=true,
    hypertexnames=false,
    linkcolor=refblue,
    citecolor=citeblue,
    urlcolor=refblue,
    filecolor=refblue
}

\newcommand{\prox}{\operatorname{Prox}}
\newcommand{\dom}{\operatorname{dom}}

\floatstyle{ruled}
\newfloat{algorithm}{tbp}{loa}
\floatname{algorithm}{Algorithm}
\newcommand{\KwIn}[1]{\noindent\textbf{Input.} #1\par\smallskip}
\newcommand{\For}[2]{\noindent\textbf{for} #1 \textbf{do}\par
  \begingroup\leftskip=1em #2\par\endgroup
  \noindent\textbf{end for}}

\theoremstyle{thmstyleone}
\newtheorem{theorem}{Theorem}[section]
\newtheorem{lemma}{Lemma}[section]
\newtheorem{assumption}{Assumption}[section]

\theoremstyle{thmstyletwo}
\newtheorem{remark}{Remark}[section]

\makeatletter
\patchcmd{\@maketitle}{\removelastskip\vskip24pt}{\removelastskip\vskip16pt}{}{}
\patchcmd{\@maketitle}{\removelastskip\vskip24pt}{\removelastskip\vskip16pt}{}{}
\patchcmd{\@maketitle}{\removelastskip\vskip36pt\vskip0pt}{\removelastskip\vskip18pt\vskip0pt}{}{}
\makeatother

\begin{document}

\title[Accelerated Golden Ratio Primal--Dual Algorithm]{Two Adaptive Accelerated Golden Ratio Primal--Dual Algorithms With an Application to Poisson Imaging Problem}

\author*[1,2]{\fnm{Santanu} \sur{Soe}}
\email{ma22d002@smail.iitm.ac.in}
\email{santanu.soe@student.unimelb.edu.au}

\author[1]{\fnm{V.} \sur{Vetrivel}}
\email{vetri@iitm.ac.in}

\affil*[1]{\orgdiv{Department of Mathematics}, \orgname{Indian Institute of Technology Madras}, \orgaddress{\city{Chennai}, \postcode{600036}, \country{India}}}

\affil[2]{\orgdiv{School of Mathematics and Statistics}, \orgname{The University of Melbourne}, \orgaddress{\city{Parkville}, \postcode{3010}, \state{VIC}, \country{Australia}}}

\abstract{This paper revisits the adaptive extended golden-ratio primal--dual algorithm (aEGRPDA) proposed by Soe et al. (2026) for structured convex optimisation problems involving a differentiable term that is only locally smooth. We prove that the artificial upper bound imposed on the primal step-size in aEGRPDA is redundant, since the adaptive rule itself keeps the step-sizes bounded above. As a consequence, the ergodic $\mathcal O(1/N)$ estimates for the objective residual and feasibility violation, where $N\ge1$ denotes the number of iterations, are independent of this hyperparameter. Consequently, the resulting adaptive golden-ratio primal--dual method, therefore, requires neither a step-size cap, nor a linesearch procedure, nor a known global Lipschitz constant. We establish linear convergence of the algorithm when both the primal and dual functions are strongly convex. Furthermore, we develop two accelerated variants, in addition to the local smoothness assumption: one for the case where the nonsmooth primal component is strongly convex, and another for the case where the differentiable term is globally strongly convex. For these accelerated methods, we prove an ergodic $\mathcal O(1/N^2)$ convergence rate. Preliminary numerical experiments on a Poisson imaging problem illustrate the efficiency and robustness of the proposed approaches.}


\maketitle


\section{Introduction}\label{sec:introduction}
Let $\mathbb{H}_1$ and $\mathbb{H}_2$ be finite-dimensional real Hilbert spaces with inner product
$\langle\cdot,\cdot\rangle$ and induced norm $\|\cdot\| = \sqrt{\langle \cdot, \cdot\rangle}$. We study the composite convex optimisation model
\begin{equation}\label{main_prob}
    \min_{x\in\mathbb{H}_1} \; f(x) + g(Kx) + h(x),
\end{equation}
where $f:\mathbb{H}_1\to(-\infty,+\infty]$ and $g:\mathbb{H}_2\to(-\infty,+\infty]$ are proper, convex, and lower semicontinuous (lsc), $K:\mathbb{H}_1\to\mathbb{H}_2$ is a linear operator, and $h:\mathbb{H}_1\to\mathbb{R}$ is convex and differentiable.
In contrast to the standard global smoothness assumption, we work with the assumption that $h$ is \emph{locally smooth}, meaning that
$\nabla h$ is locally Lipschitz continuous, i.e., for every compact set $B\subseteq \mathbb{H}_1$ there exists $L_B>0$ such that
\begin{equation*}
    \|\nabla h(x)-\nabla h(y)\| \le L_B \|x-y\|,
    \quad \forall x,y\in B.
\end{equation*}
The optimisation problem in the form \eqref{main_prob} arises in many applications, including signal processing, image denoising, and machine learning; see \cite{chambolle2011first,chen2023first,chen2016primal,combettes2012primal,esser2010general,pan2024compressive,tang2017splitting,vladarean2021first} and the references therein.

\section{Notation and preliminaries}\label{sec2}

In this section, we summarise notation and a few standard tools that will be used throughout. We denote the golden ratio by
$\varphi := \frac{1+\sqrt5}{2}$. Given a linear operator $K:\mathbb{H}_1\to\mathbb{H}_2$, its operator norm is defined as
$\|K\|:=\sup\{\|Kx\|:\ \|x\|=1\}$. Given a nonempty set $D\subseteq\mathbb{H}_1$, we write $\iota_D$ for the \emph{indicator function} of $D$, i.e.,
$\iota_D(x)=0$ if $x\in D$ and $\iota_D(x)=+\infty$ otherwise.

\medskip
Let $f:\mathbb{H}_1\to(-\infty,+\infty]$ be a proper, convex, and lower semicontinuous (lsc) function, then its \emph{effective domain} is $\dom f := \{x\in\mathbb{H}_1:\ f(x)<+\infty\}$. For $x\in\dom f$, the \emph{subdifferential} of $f$ at $x$ is given by
\[
\partial f(x):=\Big\{v\in\mathbb{H}_1:\ f(z)\ge f(x)+\langle v, z-x\rangle\ \ \forall z\in\mathbb{H}_1\Big\}.
\]
Given $\lambda>0$, the \emph{proximal operator} of $f$ is defined by
\[
\prox_{\lambda f}(x)
:=\operatorname*{argmin}_{z\in\mathbb{H}_1}
\left\{ f(z)+\frac{1}{2\lambda}\|z-x\|^2 \right\}.
\]
We now have the following useful Lemma of the proximal operator.

\medskip
\begin{lemma}\label{usefullemma3}
\cite[Theorem~6.3]{beck2017first}
Let $f:\mathbb{H}_1\to(-\infty,+\infty]$ be proper, convex, and lsc. Given $x\in\mathbb{H}_1$ and $\lambda>0$, a point
$u\in\mathbb{H}_1$ satisfies $u=\prox_{\lambda f}(x)$ if and only if
\[
\lambda\big(f(u)-f(z)\big)\le \langle u-x,\ z-u\rangle,
\qquad \forall z\in\mathbb{H}_1.
\]
\end{lemma}

\subsection{Constrained form and Lagrangian}\label{subsec:lagrangian}
By introducing an auxiliary variable $w\in\mathbb{H}_2$, the composite problem \eqref{main_prob} can be equivalently reformulated as
\begin{equation}\label{eq:constrained_form}
    \min_{x\in\mathbb{H}_1,\; w\in\mathbb{H}_2}
    \Bigl\{ f(x)+h(x)+g(w)\;\;\text{subject to}\;\; Kx-w=0 \Bigr\}.
\end{equation}
For later use, we define the objective function
\begin{equation*}
    \Phi(x,w):=f(x)+h(x)+g(w),
\end{equation*}
and the associated Lagrangian function
\begin{equation*}
    \mathbb{L}(x,w,y)
    := \Phi(x,w) + \langle y,\, Kx-w\rangle,
    \qquad (x,w,y)\in \mathbb{H}_1\times\mathbb{H}_2\times\mathbb{H}_2,
\end{equation*}
where $y\in\mathbb{H}_2$ denotes the Lagrange multiplier corresponding to the constraint $Kx=w$. By noting that $\sup_{y\in\mathbb{H}_2}\langle y,\,Kx-w\rangle=\iota_{\{0\}}(Kx-w)$, the reformulated problem~\eqref{eq:constrained_form} can be written as
\begin{equation*}
\min_{(x,w)\in\mathbb{H}_1\times\mathbb{H}_2} \ \sup_{y\in\mathbb{H}_2} \ \mathbb{L}(x,w,y).
\end{equation*}
Furthermore, using the Legendre-Fenchel conjugate $g(w)=\sup_{y\in\mathbb{H}_2}\big\{\langle w,y\rangle-g^*(y)\big\}$, we eliminate $w$ in the Lagrangian, by following $\inf_{w\in\mathbb{H}_2}\big(g(w)-\langle y,w\rangle\big)=-g^*(y)$. Consequently, the saddle-point problem of \eqref{eq:constrained_form} is
\begin{equation}\label{eq:saddle_fenchel}
    \min_{x\in\mathbb{H}_1} \ \max_{y\in\mathbb{H}_2}
    \Big\{ f(x)+h(x)+\langle Kx,y\rangle-g^*(y)\Big\}.
\end{equation}
Again, minimizing $\mathbb{L}(x,w,y)$ over $(x,w)$ yields the dual function
\[
\inf_{x\in\mathbb{H}_1,\,w\in\mathbb{H}_2} \mathbb{L}(x,w,y)
=-(f+h)^*(-K^*y)-g^*(y),
\]
and hence the Fenchel--Rockafellar dual of \eqref{eq:constrained_form} (equivalently of \eqref{main_prob}) is
\begin{equation}\label{eq:dual}
    \max_{y\in\mathbb{H}_2} \ \Big\{-(f+h)^*(-K^*y)-g^*(y)\Big\}.
\end{equation}
We are now ready to impose the following standing assumptions.

\medskip
\begin{assumption}\label{assumption1}
The saddle-point problem \eqref{eq:saddle_fenchel} admits at least one solution. Furthermore, either both the proximal mappings of $f$ and $g$ have closed forms, or they can be computed efficiently to high accuracy.
\end{assumption}

A standard constraint qualification condition based on relative interiors guarantees strong duality and, consequently, the existence of
saddle points; see, e.g., \cite[Corollaries~28.2.2 and~28.3.1]{Rockafellar+1970}. Under Assumption~\ref{assumption1},
$(\bar x,\bar w,\bar y)$ is a saddle point of $\mathbb{L}$ if and only if $(\bar x,\bar w)$ solves \eqref{eq:constrained_form},
$\bar y$ solves the dual problem \eqref{eq:dual}, and
\begin{equation}\label{eq:saddle_ineq}
    \mathbb{L}(\bar x,\bar w,y)\le \mathbb{L}(\bar x,\bar w,\bar y)\le \mathbb{L}(x,w,\bar y),
    \qquad \forall (x,w,y)\in\mathbb{H}_1\times\mathbb{H}_2\times\mathbb{H}_2.
\end{equation}
Equivalently, saddle points are described by the KKT relations
\[
-K^*\bar y \in \partial f(\bar x)+\nabla h(\bar x),\quad
\bar y\in \partial g(\bar w),\quad
K\bar x=\bar w.
\]
Let 
\[
\mathbf{\Pi}
:=\Big\{(\bar x,\bar w,\bar y)\in\mathbb{H}_1\times\mathbb{H}_2\times\mathbb{H}_2:\ 
-K^*\bar y \in \partial f(\bar x)+\nabla h(\bar x),\ \bar y\in\partial g(\bar w),\ K\bar x=\bar w\Big\}
\]
be the set of all saddle points.  Fix any $(\bar x,\bar w,\bar y)\in\mathbf{\Pi}$. For an arbitrary $(x,w,y)\in\mathbb{H}_1\times\mathbb{H}_2\times\mathbb{H}_2$, we introduce the following residual as a measure,

\begin{equation}\label{primal_dual_gap}
    \mathbb{J}(x,w,y)
    := \mathbb{L}(x,w,y)-\mathbb{L}(\bar x,\bar w,y)
    = \Phi(x,w)-\Phi(\bar x,\bar w) + \langle y,\ Kx-w\rangle,
\end{equation}
which will be required in our subsequent convergence analysis. Moreover, by the saddle property~\eqref{eq:saddle_ineq}, we have
\begin{equation}\label{Eq:measure}
    \mathbb{J}(x,w,\bar y)=\mathbb{L}(x,w,\bar y)-\mathbb{L}(\bar x,\bar w,\bar y)\ge 0,
\qquad \forall (x,w)\in\mathbb{H}_1\times\mathbb{H}_2.
\end{equation}
This measure will be used to control both the objective residual and the constraint violation.
We emphasise that we work with such residual-type measures since the classical ``primal--dual gap function'' may be
uninformative and can even vanish at non-stationary points when the dual domain is unbounded; see
\cite{chang2022grpdarevisited,zhou2022new,chambolle2011first} for related remarks.

\medskip
The following elementary facts will be frequently used in our analysis.

\medskip
\begin{lemma}\label{UsefulLemma1}
\cite{bauschke2017correction}
Let $u,v,w\in\mathbb{H}_1$ and $\lambda\in\mathbb{R}$. Then
\begin{subequations}\label{eq:basic_identities_alt}
\begin{align}
    2\langle u-v,\ u-w\rangle
    &= \|u-v\|^2 + \|u-w\|^2 - \|v-w\|^2, \label{eq:basic_identities_alt_a}\\
    \|\lambda u + (1-\lambda)v\|^2
    &= \lambda\|u\|^2 + (1-\lambda)\|v\|^2 - \lambda(1-\lambda)\|u-v\|^2. \label{eq:basic_identities_alt_b}
\end{align}
\end{subequations}
\end{lemma}

\medskip
\begin{lemma}\label{usefullemma2}
\cite{chang2022grpdarevisited}
Let $(p_n)$ and $(q_n)$ be nonnegative sequences such that
$p_{n+1}\le p_n-q_n$ for all $n\ge 1$, then $\sum_{n=1}^\infty q_n<\infty$ and $\lim_{n\to\infty}p_n$ exists.
\end{lemma}

\medskip
\begin{lemma}\label{lemm: sum un and vn}
Let $0<q<1$ and let $(u_n)$, $(v_n)$ be nonnegative sequences such that $u_n\le q\,u_{n-1}+v_n$ for all $n\ge 1$.
If $\sum_{n=1}^\infty v_n<\infty$, then $\sum_{n=1}^\infty u_n<\infty$.
\end{lemma}


\medskip
\begin{lemma}\label{min a and b}
Let $a,b>0$. Then $\min\{a,b\}\le \sqrt{ab}$.
\end{lemma}
\begin{proof}
For brevity, we give a proof of this lemma. If $a\le b$, then $\min\{a,b\}=a$. Since $a^2\le ab$, we obtain $a\le\sqrt{ab}$. The case $b\le a$ is analogous.
\end{proof}

\section{Background and motivation}

Over the past decade, primal--dual algorithms have become increasingly popular \cite{jezierska2012primal,latafat2018plug,chang2021goldengrpda,malitsky2018first,chambolle2011first,komodakis2015playing},
largely because of their ability to exploit the composite structure of problems in which a nonsmooth term is coupled with a linear operator.
This advantage is typically realised through a saddle-point formulation \cite{chambolle2011first,condat2013primal,chambolle2016ergodic,vu2013splitting}, where the linear operator and the nonsmooth function are separated by introducing the Fenchel--Rockafellar conjugate \cite[Definition 13.1]{bauschke2017correction}; see Section~\ref{sec2} for further details. A number of well-known primal--dual algorithms have been developed for solving~\eqref{main_prob} when $h$ is globally smooth; see, for example~\cite{condat2013primal,vu2013splitting,zhou2022new,chen2016primal,yan2018new,salim2022dualize,malitsky2026first}. In these methods, the step-sizes typically depend on the Lipschitz constant of $\nabla h$ and the operator norm $\|K\|$ through an inequality. In practice, for problems like \eqref{main_prob}, the Lipschitz constant of $\nabla h$ is often unavailable (either hard to estimate or $h$ is locally smooth, whereas $\|K\|$ may be easily available), and even when it can be estimated, it may be extremely large, leading to conservative step-size choices and potentially slow convergence. Among the primal--dual algorithms in the literature, the method most closely related to the present work is the E-GRPDA of Zhou et al.~\cite{zhou2022new}. This algorithm uses fixed primal and dual step-sizes and is based on the golden-ratio convex combination idea introduced by Malitsky~\cite{malitsky2020golden} for \emph{mixed variational inequality problems} (MVIP). Motivated by the derivation of E-GRPDA in the fixed step-size case, Soe et al.~\cite[Algorithm~3]{Soe2026} proposed an adaptive variant of E-GRPDA~\cite{zhou2022new,chang2021goldengrpda}, called aEGRPDA. It is worth noting that although one may rewrite~\eqref{main_prob} as an MVIP problem, and apply the adaptive Golden Ratio algorithm (aGRAAL)~\cite[Algorithm~1]{malitsky2020golden} (an improved version of GRAAL), the resulting approach leads to worse error bounds and is not well suited to the composite structure like in~\eqref{main_prob}; see~\cite[Algorithm 1]{Soe2026} for further details. The aEGRPDA algorithm~\cite{Soe2026} is tailored to locally smooth $h$, incorporates an adaptive primal step-size, and exhibits favourable convergence properties; see~\cite[Section~6]{Soe2026}. We emphasise the word "adaptive", by which we mean that the primal step-size leverages the local information of the inverse of the Lipschitz constant of $\nabla h$, while still allowing the steps to remain non-monotone. We now recall the aEGRPDA algorithm. 

\begin{algorithm}[H]
\caption{The aEGRPDA for \eqref{main_prob}}\label{alg:aEGRPDA}
\KwIn{Choose $x_0\in\mathbb{H}_1$, $y_0\in\mathbb{H}_2$ and set $z_0=x_0$. Let $\tau_0>0$, $\beta>0$, $\psi\in(1,\varphi]$, $\rho=\psi^{-1}+\psi^{-2}$, $\theta_0=1$ and $\tau_{\max}>0$, where $\varphi=(1+\sqrt5)/2$ is the golden ratio.}

\For{$n=1,2,\ldots$}{
  \textbf{Step 1} (Compute)
  \begin{align*}
    z_n
    &=\frac{\psi-1}{\psi}x_{n-1}+\frac{1}{\psi}z_{n-1},\\
    x_n
    &=
    \prox_{\tau_{n-1}f}
    \left(
        z_n-\tau_{n-1}K^*y_{n-1}
        -\tau_{n-1}\nabla h(x_{n-1})
    \right). 
  \end{align*}

  \textbf{Step 2} (Update)
  \begin{align*}
    \tau_n
    =
    \min\left\{
        \rho\tau_{n-1},
        \frac{\psi\theta_{n-1}}
        {9\bigl(\bar L_n^2+\beta\psi\|K\|^2\bigr)\tau_{n-1}},
        \tau_{\max}
    \right\},
    \qquad
    \sigma_n=\beta\tau_n.
  \end{align*}

  \textbf{Step 3} (Compute)
  \begin{align*}
    w_n
    &=
    \prox_{\frac{1}{\sigma_n}g}
    \left(
        \frac{y_{n-1}}{\sigma_n}+Kx_n
    \right), \\
    y_n
    &=
    y_{n-1}
    +
    \sigma_n(Kx_n-w_n).
  \end{align*}

  \textbf{Step 4} (Update)
  \begin{equation*}
    \theta_n=\frac{\psi\tau_n}{\tau_{n-1}}.
  \end{equation*}
}
\end{algorithm}

\medskip
There are two aspects of Algorithm \ref{alg:aEGRPDA}. On one hand, the primal step-size $(\tau_n)$ can adapt to the local smoothness of $h$, recover from poor initial choices, and even increase along the iterations. On the other hand, the convergence proof~\cite{Soe2026} relies on an additional hyperparameter $\tau_{\max}$, and the resulting ergodic rate estimates take the form
\begin{equation}\label{eq:gap_function}
\left | \Phi(\tilde{x}_N,\, \tilde{w}_N)- \Phi(x^*, w^*)\right|
\leq \frac{C_1}{N},
\qquad
\left \| K\tilde{x}_N- \tilde{w}_N \right\|
\leq \frac{2C_1}{\delta N},
\end{equation}
where $(x^*,w^*,y^*)\in\mathbf{\Pi}$, $N\ge 1$, $\tilde{x}_N = \frac{1}{N} \sum_{n=1}^{N} x_n$, $\tilde{w}_N = \frac{1}{N} \sum_{n=1}^{N} w_n$, and $\delta\ge 2\|y^*\|$. Moreover, the constant $C_1$ depends linearly on $\tau_{\max}$
\begin{equation*}
    C_1 := \frac{{9C\tau_{\max}}(L^2+\beta\psi \|K\|^2)}{2\psi},
\end{equation*}
where $L$ is an (unknown) Lipschitz constant of $\nabla h$ over the compact region $B=\overline{\operatorname{conv}}\{x^*, x_0, x_1,\ldots\}$, and the constant $C>0$ arises from telescoping Fej\'er-type terms; see~\cite[Theorem~5.1]{Soe2026} for a detailed discussion. This dependence creates a practical dilemma. If $\tau_{\max}$ is chosen very large, then the theoretical bound~\eqref{eq:gap_function} becomes uninformative. Conversely, if $\tau_{\max}$ is chosen too small, the algorithm may be forced to take unnecessarily conservative steps, which can noticeably slow down convergence. The next example illustrates this phenomenon. Consider the quadratic problem
\begin{equation}\label{eq:toy_primal}
\min_{x\in\mathbb{R}^n}\; \frac{1}{2}\|Kx-b\|^2 + \frac{L}{2}\|x\|^2,
\end{equation}
where $b\in\mathbb{R}^m$ is fixed, $K\in\mathbb{R}^{m\times n}$, and $L>0$ is the global Lipschitz constant of $\nabla h$. We can rewrite~\eqref{eq:toy_primal} equivalently as
\begin{equation*}
\min_{x\in\mathbb{R}^n,\; w\in\mathbb{R}^m}\;
\Phi(x,w):= g(w)+h(x)
\quad \text{subject to}\quad Kx=w,
\end{equation*}
with $g(w)=\tfrac12\|w-b\|^2$, $h(x)=\tfrac{L}{2}\|x\|^2$ and $f\equiv 0$.
The unique minimiser of this example is $x^\ast=(K^\top K + L I)^{-1}K^\top b,$ and $w^\ast=Kx^\ast$. Take $(m,n)=(50,100)$, $L=10^{-3}$, and construct $K$ with small spectral norm. Specifically, we draw a random Gaussian matrix $A\in\mathbb{R}^{m\times n}$ and scale it as
\[
K := \frac{\kappa}{\|A\|}\,A;~~\kappa=10^{-3},
\]
so that $\|K\|\approx 10^{-3}$. Figure~\ref{fig:feas_gap_both_sizes_lasso_n_1000} compares several values of $\tau_{\max}$ and reports the feasibility and objective residuals, and the corresponding primal step-sizes. The theory suggests that $\tau_{\max}$ should be chosen on the order of $10^3$ for this instance, and from Figure~\ref{fig:feas_gap_both_sizes_lasso_n_1000} we can observe that smaller choices of $\tau_{\max}$ than the expected value, prematurely limit the step-sizes and slow down convergence. This motivates further analysis of aEGRPDA, without the need for an artificial cap.

\begin{figure}[htbp]
\centering

\subfloat[Feasibility residual]{
  \includegraphics[width=0.27\linewidth]{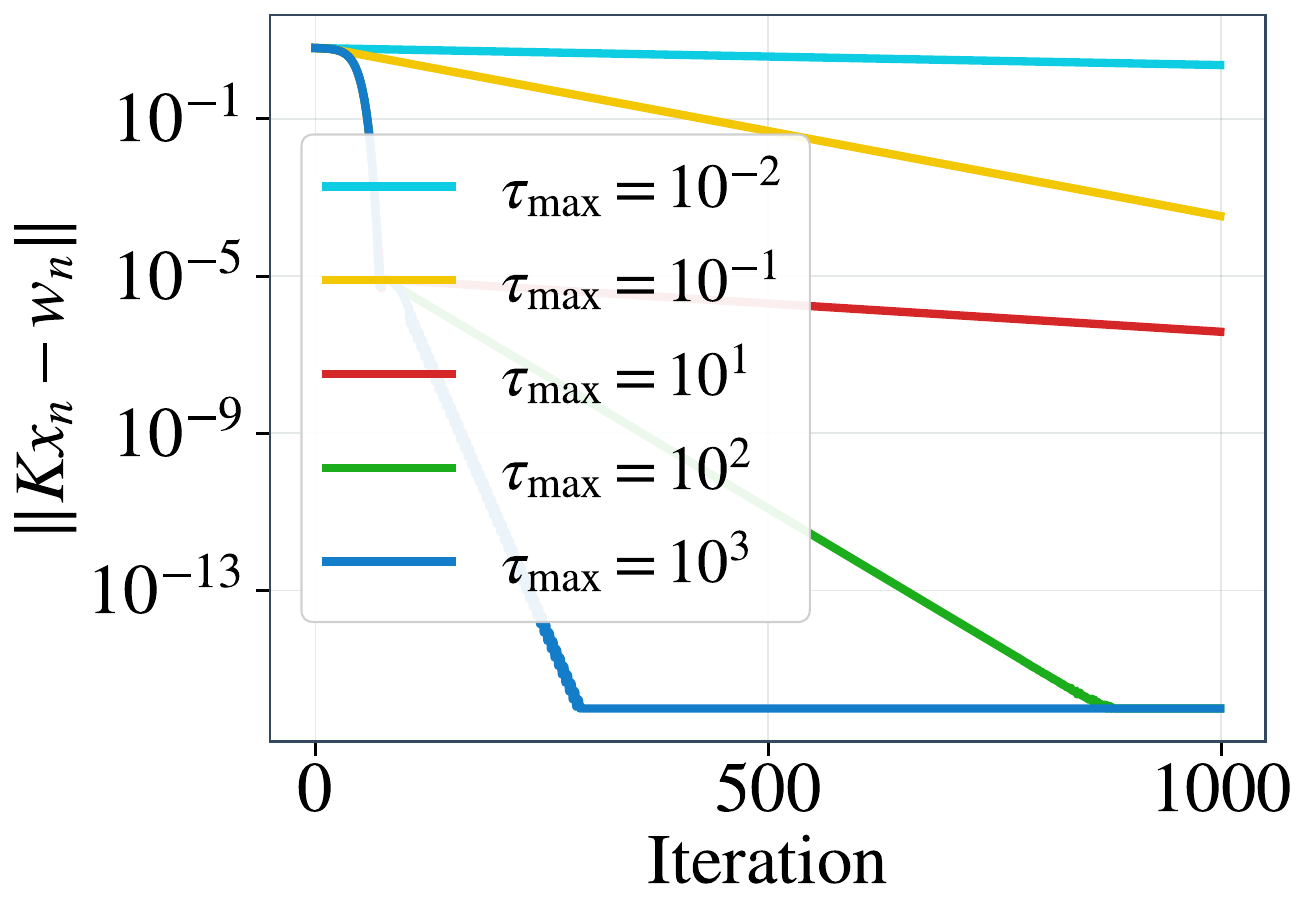}
}\hfill
\subfloat[Objective residual]{
  \includegraphics[width=0.27\linewidth]{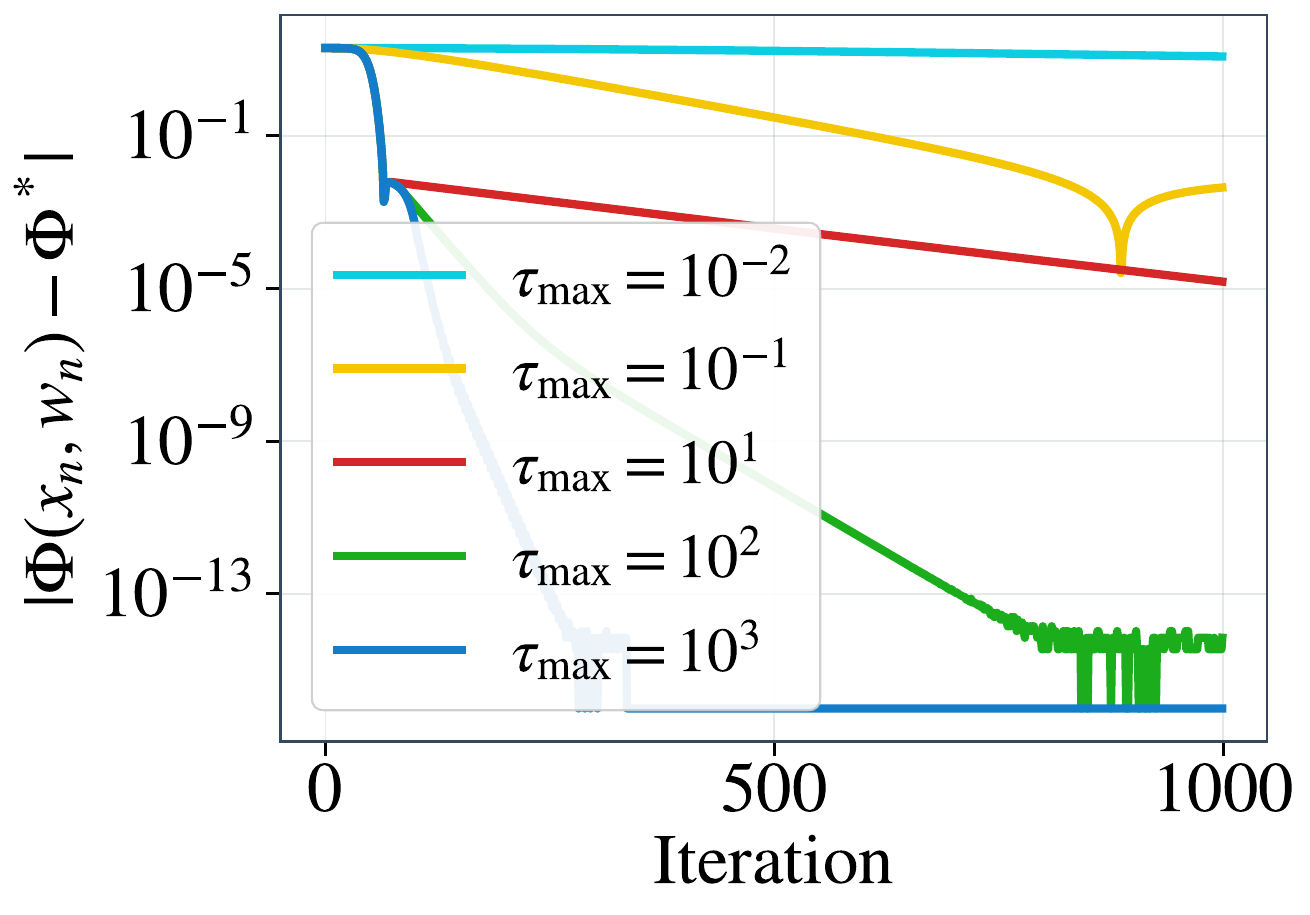}
}\hfill
\subfloat[Primal step-size]{
  \includegraphics[width=0.27\linewidth]{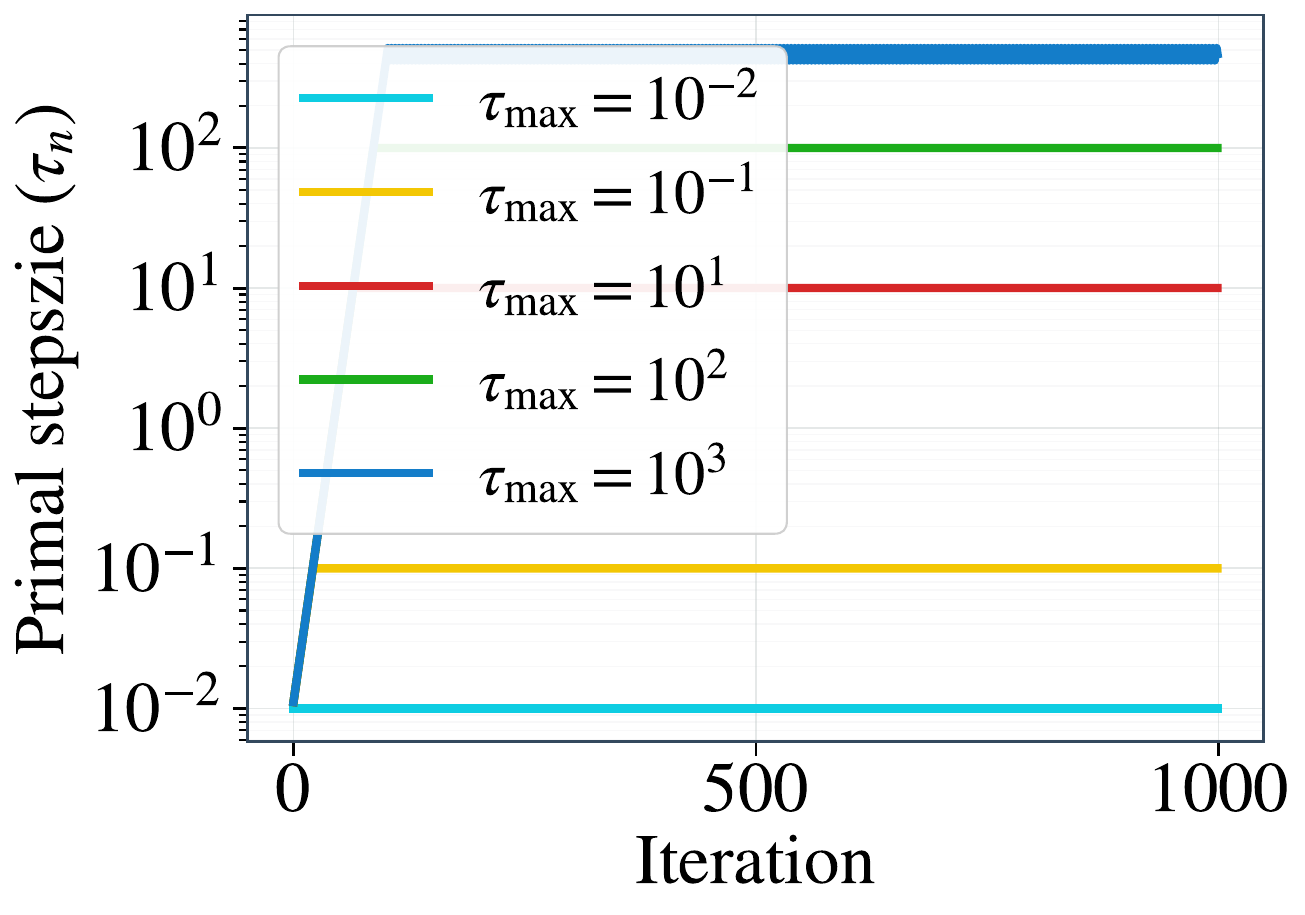}
}
\caption{Influence of $\tau_{\max}$ on feasibility and objective residuals.}
\label{fig:feas_gap_both_sizes_lasso_n_1000}
\end{figure}
Beyond the setting considered above, in many applications where $h$ is globally smooth, the Lipschitz constant $L$ of $\nabla h$ is often much larger than the operator norm $\|K\|$. Consequently, it becomes far more important to estimate $L$ locally, whereas the norm $\|K\|$ can be computed in advance. Examples include logistic regression, fused lasso, elastic net problem, and so on; see \cite{latafat2023adaptive,Soe2026,vladarean2021first} and the references therein. An interesting example where $h$ is locally smooth, arises in the Poisson imaging problems~\cite{di2020acquire}:
\begin{equation*}
    \min_{x}\; \iota_{x\ge 0}(x)+ \lambda\|\nabla x\|_{2,1}+\mathrm{KL}(Ax,y),
\end{equation*}
where $K=\nabla$, $h(\cdot)=\mathrm{KL}(A(\cdot),y)$ and $\mathrm{KL}$ denotes the Kullback--Leibler divergence \cite{bauschke2017descent}; also see Section~\ref{Num_section} for further details. It is well known that $\|K\|\approx \sqrt{8}$~\cite{chambolle2011first}, while $h$ is only locally smooth~\cite{bauschke2017descent}. Such problems underscore the need for algorithms that robustly handle local smoothness, require minimal tuning, and perform well in practice.

\medskip
Another objective of this paper is not just to revisit the analysis of aEGRPDA~\cite{Soe2026} under the local smoothness of $h$, but also to establish accelerated convergence rates when additional curvature is available. In the literature, accelerated primal--dual methods with ergodic $\mathcal{O}(1/N^2)$ rates for~\eqref{main_prob} are typically developed under the assumption that $\nabla h$ is globally Lipschitz continuous; see, for example,~\cite{chambolle2016ergodic,driggs2024practical,malitsky2017chambolle}. Related accelerated schemes for deterministic and stochastic saddle-point problems can also be found in~\cite{chen2014optimal}. More recently, Nesterov-type acceleration~\cite{nesterov2013gradient,nesterov2013introductory} has been used in~\cite{condat2026nesterov}. These results provide important acceleration mechanisms, but they are essentially confined to the globally smooth framework of $h$. This limitation is significant when a globally smooth structure for $h$ is unavailable. Therefore, a natural question is whether one can design accelerated primal--dual algorithms that use the local behaviour of $h$ rather than relying on a global smoothness constant. In particular, when strong convexity is present in one of the component functions, it is reasonable to expect that this additional curvature should be reflected in the convergence rate. We compare several primal--dual algorithms that achieve acceleration under local or global smoothness assumptions on $h$; see the Table~\ref{tab:literature-comparison}.

\begin{table}[htbp]
\centering
\caption{Comparison of related primal-dual methods under local smoothness and acceleration requirements.}
\label{tab:literature-comparison}
\renewcommand{\arraystretch}{1.18}
\setlength{\tabcolsep}{5pt}
\begin{tabularx}{\textwidth}{
>{\RaggedRight\arraybackslash}p{3.2cm}
>{\centering\arraybackslash}p{2.0cm}
>{\centering\arraybackslash}p{2.2cm}
>{\RaggedRight\arraybackslash}X
}
\toprule
\textbf{Method}
&
\textbf{Local $h$}
&
\textbf{Acceleration}
&
\textbf{Main features/limitations}
\\
\midrule

PDA-L~\cite{malitsky2018first}
&
\xmark
&
\cmark
&
Requires a line-search procedure.
\\

adaPDM~\cite{latafat2023adaptive}
&
\cmark
&
\xmark
&
Adaptive, but not accelerated in the present setting.
\\

aPDAc-L~\cite{chang2026convex}
&
\cmark
&
\cmark
&
Requires a line-search procedure.
\\

aGRAAL~\cite{malitsky2020golden}
&
\cmark
&
\xmark
&
Not suitable after a primal--dual reformulation, and gives worse convergence results.
\\

Soe--Tam--Vetrivel~\cite{Soe2026}
&
\cmark
&
\xmark
&
Allows local smoothness and non-monotone step-sizes, but gives an $\mathcal O(1/N)$ rate and uses an artificial step-size cap.
\\

\textbf{This work}
&
\cmark
&
\cmark
&
\textbf{Line-search-free accelerated rate under local smoothness.}
\\

\bottomrule
\end{tabularx}
\end{table}
A closely related contribution is the aPDAc-L algorithm, recently developed in ~\cite{chang2026convex}, which addresses~\eqref{main_prob} under local smoothness assumptions and achieves an accelerated ergodic $\mathcal{O}(1/N^2)$ rate of convergence. However, this method uses a backtracking linesearch to control the local smoothness of $h$. While linesearch is theoretically powerful, it can be less attractive computationally as it often requires careful initialisation, repeated matrix--vector evaluations, and an inner loop that continues until a suitable stopping criterion is satisfied. Therefore, our goal is to obtain acceleration without using backtracking. There are other works in the literature that are in line with adaptive primal--dual methods that avoid linesearch, such as APDA~\cite{vladarean2021first} in the case $f\equiv 0$, and adaPDM~\cite{latafat2023adaptive}. These methods reduce the dependence on global Lipschitz constants and share a similar adaptive philosophy. However, their convergence analyses rely on primal--dual gap arguments and do not yield accelerated $\mathcal{O}(1/N^2)$ guarantees under their adaptive step-size strategies. Therefore, one of the main motivations of this paper is to bridge this gap by developing a backtracking-free adaptive golden-ratio primal--dual framework that uses only local smoothness information and still achieves accelerated rates when suitable strong curvature is available.

\medskip
The structure of the paper and its main contributions are summarised as follows.

\begin{itemize}
    \item In Section~\ref{sec2}, we collect the notation and standing assumptions used throughout the paper. We also introduce the residual measures that are used later in the convergence analysis.
    
    \item In Section~\ref{sec5}, we revisit the adaptive extended golden-ratio
    primal--dual algorithm of~\cite{Soe2026} and remove the artificial upper
    bound $\tau_{\max}$ on the primal step-size. Indeed, using a simple inequality
    (Lemma~\ref{min a and b}), we prove that the adaptive step-size rule itself keeps
    the primal step-sizes bounded. We further establish global convergence and ergodic $\mathcal O(1/N)$ rates for both the objective residual and the feasibility violation. Unlike the bounds in \cite[Algorithm~3]{Soe2026}, our estimates do not depend on the artificial parameter $\tau_{\max}$.

    \item In Section~\ref{subsec:linear_convergence_results}, we establish linear convergence of Algorithm~\ref{algorithm 2} under the additional assumption that both $f$ and $g^*$ are strongly convex.
    
    \item In Section~\ref{sec:acc}, we develop two separate accelerated variants of Algorithm \ref{algorithm 2} under the
    strong convexity of $f$ and $h$. These schemes retain the
    non-monotone nature of the step-size rule and use local smoothness
    information of $h$. For the proposed accelerated algorithms (Algorithm \ref{algorithm 2 acclerated} and \ref{algorithm:hstrong_acc}), we prove improved ergodic $\mathcal O(1/N^2)$ convergence rates.

    \item Finally, in Section \ref{Num_section}, we present preliminary numerical experiments on
    the Poisson imaging problem. The results illustrate the robustness of the proposed step-size rules and the benefit of the proposed accelerated variants.
\end{itemize}

\section{Parameter free adaptive golden ratio primal-dual algorithm}\label{sec5}

In this section, we present a \emph{parameter-free} counterpart of the adaptive golden ratio algorithm~\cite{Soe2026}. We establish the global convergence of the generated iterates and the ergodic $\mathcal{O}(1/N)$ rates for the objective gap and the feasibility violation. Unless stated otherwise, we work under Assumption~\ref{assumption1} and the following one.

\begin{assumption}\label{assumption_5.1}
Suppose that $f$ and $g$ are proper, convex and lower semicontinuous, and that $h$ is convex and locally smooth.
\end{assumption}
For convenience, we introduce the local estimate
\begin{equation*}
L_n :=
\frac{\|\nabla h(x_n)-\nabla h(x_{n-1})\|}{\|x_n-x_{n-1}\|}
\quad\text{if }x_n\neq x_{n-1},\qquad n\ge 1.
\end{equation*}

\begin{algorithm}[H]
\caption{Parameter Free GRPDA (PF-GRPDA) for \eqref{main_prob}}\label{algorithm 2}
\KwIn{Choose $x_0\in\mathbb{H}_1$, $y_0\in\mathbb{H}_2$ and set $z_0=x_0$. Fix $\beta>0$, $\psi\in(1,\varphi]$, $\rho=\psi^{-1}+\psi^{-2}$. Let $(\alpha_n)_{n\ge0}\subset(\epsilon,1/3]$, where $\epsilon>0$. Choose $\theta_0=1$ and $\tau_0>0$.}

\For{$n=1,2,\ldots$}{
  \textbf{Step 1} (Compute)
  \begin{align}
    z_n
    &=\frac{\psi-1}{\psi}x_{n-1}+\frac{1}{\psi}z_{n-1}, \label{eqn:39a}\\
    x_n
    &=
    \prox_{\tau_{n-1} f}
    \bigl(
        z_n-\tau_{n-1}K^*y_{n-1}
        -\tau_{n-1}\nabla h(x_{n-1})
    \bigr). \label{eqn:39b}
  \end{align}

  \textbf{Step 2} (Update)
  \begin{align}
    \tau_n
    =
    \min\left\{
        \rho\tau_{n-1},
        \frac{\alpha_n\alpha_{n-1}\psi\theta_{n-1}}
        {(L_n^2+\beta\psi\|K\|^2)\tau_{n-1}}
    \right\},
    \qquad
    \sigma_n=\beta\tau_n . \label{eq:39c}
  \end{align}

  \textbf{Step 3} (Compute)
  \begin{align}
    w_n
    &=
    \prox_{\frac{1}{\sigma_n}g}
    \left(
        \frac{y_{n-1}}{\sigma_n}+Kx_n
    \right), \label{eq:w_n update}\\
    y_n
    &=
    y_{n-1}
    +
    \sigma_n(Kx_n-w_n). \label{eq:y_n update}
  \end{align}

  \textbf{Step 4} (Update)
  \begin{equation*}
    \theta_n=\frac{\psi\tau_n}{\tau_{n-1}}.
  \end{equation*}
}
\end{algorithm}

Before proving the convergence analysis, we record a few remarks that will be used repeatedly. 

\begin{remark}
In Algorithm~\ref{algorithm 2} we use the convention $\frac00=+\infty$. 
With this convention, if $x_n=x_{n-1}$, then the update~\eqref{eq:39c} reduces to $\tau_n=\rho\tau_{n-1}$. For a rigorous analysis, in the second term of \eqref{eq:39c}, we replace the coefficient by $\alpha_n\alpha_{n-1}$. In the original algorithm the step-size $(\tau_n)$ was computed with $\alpha=\frac{1}{9}$, see \cite[Algorithm 3]{Soe2026}.
\end{remark}

\begin{remark}\label{tau_n_bdd_above}
Unlike~\cite{Soe2026}, we do not impose any \emph{a priori} upper bound on the step-sizes. In fact, the rule~\eqref{eq:39c} already yields a uniform bound.
Let
\[
D_n \coloneqq L_n^2+\beta\psi\|K\|^2,
\qquad
c_n \coloneqq \frac{\psi\alpha_n\alpha_{n-1}\theta_{n-1}}{D_n}.
\]
Since $D_n\ge \beta\psi\|K\|^2$ for all $n$, the update~\eqref{eq:39c} can be written as
\begin{equation}\label{eq:tau_min_form}
\tau_n=\min\left\{\rho\tau_{n-1},~\frac{c_n}{\tau_{n-1}}\right\}.
\end{equation}
By applying Lemma~\ref{min a and b} to~\eqref{eq:tau_min_form} gives
\[
\tau_n \le \sqrt{\Big(\rho\tau_{n-1}\Big)\Big(\frac{c_n}{\tau_{n-1}}\Big)}
      =  \sqrt{\rho\,c_n}.
\]
Moreover, since $\theta_n\le 1+\frac{1}{\psi}$ for all $n$ and $D_n\ge \beta\psi\|K\|^2$, we have
\[
c_n=\frac{\psi\alpha_n\alpha_{n-1}\theta_{n-1}}{D_n}
\le \frac{\psi+1}{9\beta\psi\|K\|^2}
=:\,c_{\max}.
\]
Consequently,
\[
\tau_n \le U:=\sqrt{\rho\,c_{\max}}
\qquad \forall n\ge 1.
\]
\end{remark}

\begin{remark}\label{remark_5.2}
From the step-size rule~\eqref{eq:39c}, one always has
\[
\tau_n \le \frac{\psi\alpha_n\alpha_{n-1}\theta_{n-1}}{L_n^2+\beta\psi\|K\|^2}\,\frac{1}{\tau_{n-1}}.
\]
Using $\theta_n=\frac{\psi\tau_n}{\tau_{n-1}}$ and $\frac{\psi\theta_{n-1}}{\tau_{n-1}}=\frac{\theta_n\theta_{n-1}}{\tau_n}$, the above inequality can be rewritten as
\[
\tau_n \le \frac{\alpha_n\alpha_{n-1}\theta_n\theta_{n-1}}{L_n^2+\beta\psi\|K\|^2}\,\frac{1}{\tau_n}.
\]
In particular, we obtain the useful bounds
\[
\tau_nL_n \le\sqrt{\alpha_n\alpha_{n-1}\theta_n\theta_{n-1}}
\quad\text{and}\quad
\tau_n\|K\|\le\sqrt{\frac{\alpha_n\alpha_{n-1}\theta_n\theta_{n-1}}{\beta\psi}}.
\]
\end{remark}

\begin{lemma}\label{lemma_5.0.1}
Under Assumptions~\ref{assumption1} and \ref{assumption_5.1}, let $\{(z_n,x_n,w_n,y_n,\tau_n)\}$ be generated by Algorithm~\ref{algorithm 2}. Let $(x^*,w^*,y^*)\in\mathbf{\mathbf{\Pi}}$ be a saddle point of $\mathbb{L}$.
Then, for any $y\in\mathbb{H}_2$ and $n\ge 1$, we have
\begin{multline}\label{adap_4}
2\tau_n\mathbb{J}(x_n,w_n,y)+ \frac{\psi}{\psi-1}\|x^*-z_{n+2}\|^2 + \frac{1}{\beta}\|y-y_{n}\|^2+\alpha_n\theta_n\|x_n-x_{n+1}\|^2 \\
\le
\frac{\psi}{\psi-1}\|x^*-z_{n+1}\|^2 +\frac{1}{\beta}\|y-y_{n-1}\|^2 + \alpha_{n-1}\theta_{n-1}\|x_n-x_{n-1}\|^2- {\theta_n}\|x_n-z_{n+1}\|^2.
\end{multline}
\end{lemma}

\begin{proof}
Starting from~\eqref{eqn:39b}, \eqref{eq:w_n update}, \eqref{eq:y_n update} and Lemma~\ref{usefullemma3}, we obtain
\begin{equation}\label{3.2.1}
\begin{aligned}
\tau_{n}\big(f(x_{n+1})-f(x^*)\big)
&\le \left\langle x_{n+1}-z_{n+1}+\tau_n K^*y_n+\tau_n\nabla h(x_n),\,x^*-x_{n+1}\right\rangle,\\
g(w_n)-g(w^*) &\le \langle y_n, \,w_n-\,w^*\rangle.
\end{aligned}
\end{equation}
Next, using the identity $x_n-z_n=\psi(x_n-z_{n+1})$, we similarly have
\begin{equation}\label{3.2.3}
\begin{aligned}
\tau_{n-1}\big(f(x_n)-f(x_{n+1})\big)
&\le \left\langle \psi(x_n-z_{n+1})+\tau_{n-1}K^*y_{n-1}+\tau_{n-1}\nabla h(x_{n-1}),\,x_{n+1}-x_n\right\rangle .
\end{aligned}
\end{equation}
Multiplying the second inequality in~\eqref{3.2.1} by $\tau_n$, multiplying~\eqref{3.2.3} by $\frac{\tau_n}{\tau_{n-1}}$, and adding the result to the first inequality in~\eqref{3.2.1} yields
\begin{multline}\label{3.2.5}
\tau_n\left(f(x_n)-f(x^*)+g(w_n)-g(w^*)\right)
\le \langle x_{n+1}-z_{n+1}, \,x^*-x_{n+1}\rangle\\
+ \theta_n\langle x_n-z_{n+1},\,x_{n+1}-x_n\rangle
+ \tau_n\langle K^*y_{n-1}, \,x_{n+1}-x_n\rangle
+ \tau_n\langle K^*y_n,\,x^*-x_{n+1}\rangle\\
+ \tau_n\langle y_n,\,w_n-w^*\rangle
+ \tau_n\langle \nabla h(x_{n-1}),\,x_{n+1}-x_n\rangle
+ \tau_n\langle \nabla h(x_n),\,x^*-x_{n+1}\rangle.
\end{multline}
Using the feasibility condition $Kx^*=w^*$, we have
\begin{align}\label{eq:Ky-collect}
&\langle K^*y_{n-1},\,x_{n+1}-x_n\rangle
  + \langle K^*y_n,\,x^*-x_{n+1}\rangle
  + \langle y_n,\,w_n-w^*\rangle \nonumber\\
&= \langle y_{n-1}-y_n,\,K(x_{n+1}-x_n)\rangle
   + \langle y_n,\,-Kx_n+w_n\rangle .
\end{align}
Likewise, the gradient terms satisfy
\begin{align}\label{eq:grad-collect}
&\langle \nabla h(x_{n-1}),\,x_{n+1}-x_n\rangle
  + \langle \nabla h(x_n),\,x^*-x_{n+1}\rangle \nonumber\\
&= \langle \nabla h(x_n)-\nabla h(x_{n-1}),\,x_n-x_{n+1}\rangle
   + \langle \nabla h(x_n),\,x^*-x_n\rangle .
\end{align}
Substituting~\eqref{eq:Ky-collect}--\eqref{eq:grad-collect} into~\eqref{3.2.5} gives
\begin{multline}\label{3.2.5_new}
\tau_n\!\left(f(x_n)-f(x^*)+g(w_n)-g(w^*)\right)
\le \langle x_{n+1}-z_{n+1},\,x^*-x_{n+1}\rangle\\
+ \theta_n\langle x_n-z_{n+1},\,x_{n+1}-x_n\rangle
+ \tau_n\langle y_n,\,-Kx_n+w_n\rangle
+ \tau_n\langle y_{n-1}-y_n,\,K(x_{n+1}-x_n)\rangle\\
+ \tau_n\langle \nabla h(x_n)-\nabla h(x_{n-1}),\,x_n-x_{n+1}\rangle
+ \tau_n\langle \nabla h(x_n),\,x^*-x_n\rangle.
\end{multline}
From the $y$-update~\eqref{eq:y_n update} with $\sigma_n=\beta\tau_n$, we also have the identity
\begin{equation}\label{eq:y-id}
\tau_n\langle y,\,Kx_n-w_n\rangle=\frac{1}{\beta}\langle y-y_n,\,y_n-y_{n-1}\rangle +\tau_n\langle y_n ,\, Kx_n-w_n\rangle.
\end{equation}
Adding~\eqref{eq:y-id} to~\eqref{3.2.5_new} yields
\begin{multline}\label{lemma4.1.0}
\tau_n\!\left(f(x_n)-f(x^*)+\langle y, Kx_n-w_n\rangle+g(w_n)-g(w^*)\right)
\le \langle x_{n+1}-z_{n+1},\,x^*-x_{n+1}\rangle\\
+ \theta_n\langle x_n-z_{n+1},\,x_{n+1}-x_n\rangle
+ \frac{1}{\beta}\langle y-y_n,\,y_n-y_{n-1}\rangle
+ \tau_n\langle y_n-y_{n-1},\, K(x_{n+1}-x_n)\rangle\\
+ \tau_n\langle \nabla h(x_n)-\nabla h(x_{n-1}),\,x_n-x_{n+1}\rangle
+ \tau_n\langle \nabla h(x_n),\,x^*-x_n\rangle.
\end{multline}
By the convexity of $h$, we have
\[
\tau_n\langle \nabla h(x_n),\,x^*-x_n\rangle\le \tau_n\bigl(h(x^*)-h(x_n)\bigr).
\]
Combining this estimate with~\eqref{lemma4.1.0} and recalling the definition of $\mathbb{J}$ in~\eqref{primal_dual_gap}, we obtain
\begin{multline}\label{lemma4.1_.1}
\tau_n\,\mathbb{J}(x_n,w_n,y)
\le \langle x_{n+1}-z_{n+1},\,x^*-x_{n+1}\rangle
+ \theta_n\langle x_n-z_{n+1},\,x_{n+1}-x_n\rangle\\
+ \frac{1}{\beta}\langle y-y_n, \, y_n-y_{n-1}\rangle
+ \tau_n\langle y_n-y_{n-1},\, K(x_{n+1}-x_n)\rangle\\
+ \tau_n\langle \nabla h(x_n)-\nabla h(x_{n-1}),\,x_n-x_{n+1}\rangle .
\end{multline}
Applying~Lemma \ref{UsefulLemma1}\eqref{eq:basic_identities_alt_a} to the first three terms on the right-hand side of~\eqref{lemma4.1_.1} gives
\begin{multline}\label{lemma4.1_2}
\tau_n\mathbb{J}(x_n,w_n, y)+  \frac{1}{2}\|x^*-x_{n+1}\|^2 + \frac{1}{2\beta}\|y-y_{n}\|^2
\le \frac{1}{2}\|x^*-z_{n+1}\|^2+ \frac{1}{2\beta}\|y-y_{n-1}\|^2 \\ 
-\frac{1}{2}\|x_{n+1}-z_{n+1}\|^2
-\frac{\theta_n}{2}\|x_n-z_{n+1}\|^2-\frac{\theta_n}{2}\|x_n-x_{n+1}\|^2
+ \frac{\theta_n}{2}\|x_{n+1}-z_{n+1}\|^2
-\frac{1}{2\beta}\|y_n-y_{n-1}\|^2 \\ 
+ \tau_n\langle y_n-y_{n-1},\, K(x_{n+1}-x_n)\rangle
+ \tau_n\langle \nabla h(x_n)-\nabla h(x_{n-1}),\, x_n-x_{n+1}\rangle.
\end{multline}
We now control the last two inner products. By Cauchy--Schwarz and Remark~\ref{remark_5.2}, we have
\begin{equation}\label{adap_2}
\begin{aligned}
2\tau_n\langle \nabla h(x_n)-\nabla h(x_{n-1}), x_n-x_{n+1}\rangle
&\le 2\tau_n\|\nabla h(x_n)-\nabla h(x_{n-1})\|\,\|x_n-x_{n+1}\| \\
&\le 2\tau_nL_n\|x_n-x_{n-1}\|\,\|x_n-x_{n+1}\|\\
&\le 2\sqrt{\alpha_n\alpha_{n-1}\theta_n\theta_{n-1}}\|x_n-x_{n-1}\|\,\|x_n-x_{n+1}\|\\
&\le \alpha_n\theta_n\|x_n-x_{n+1}\|^2+\alpha_{n-1}\theta_{n-1}\|x_n-x_{n-1}\|^2.
\end{aligned}
\end{equation}
Similarly, using Remark~\ref{remark_5.2} again gives
\begin{equation}\label{adap_3}
\begin{aligned}
2\tau_n\langle y_n-y_{n-1},\, K(x_{n+1}-x_n)\rangle
&\le 2\tau_n\|Kx_n-Kx_{n+1}\|\,\|y_n-y_{n-1}\| \\
&\le 2\tau_n\|K\|\,\|x_n-x_{n+1}\|\,\|y_n-y_{n-1}\|\\
&\le2 \sqrt{\frac{\alpha_n\alpha_{n-1}\theta_n\theta_{n-1}}{\beta\psi}}\|x_n-x_{n+1}\|\,\|y_n-y_{n-1}\|\\
&\le \alpha_n\theta_n\|x_n-x_{n+1}\|^2+\frac{\alpha_{n-1}\theta_{n-1}}{\beta\psi}\|y_n-y_{n-1}\|^2.
\end{aligned}
\end{equation}
From~\eqref{eqn:39a}, we have
\begin{equation}\label{derived_equality2}
\|x^*-x_{n+1}\|^2= \frac{\psi}{\psi-1}\|x^*-z_{n+2}\|^2 - \frac{1}{\psi-1}\|x^*-z_{n+1}\|^2+ \frac{1}{\psi}\|x_{n+1}-z_{n+1}\|^2.
\end{equation}
Substituting~\eqref{adap_2} and~\eqref{adap_3} into~\eqref{lemma4.1_2}, and then using~\eqref{derived_equality2}, we obtain
\begin{multline}\label{adap_5}
2\tau_n\mathbb{J}(x_n,w_n,y)+ \frac{\psi}{\psi-1}\|x^*-z_{n+2}\|^2 + \frac{1}{\beta}\|y-y_{n}\|^2
\le \frac{\psi}{\psi-1}\|x^*-z_{n+1}\|^2\\
+\frac{1}{\beta}\|y-y_{n-1}\|^2 
+ \left(\theta_n-1-\frac{1}{\psi}\right)\|x_{n+1}-z_{n+1}\|^2 
- {\theta_n}\|x_n-z_{n+1}\|^2 
+ \alpha_{n-1}\theta_{n-1}\|x_n-x_{n-1}\|^2\\
-\theta_n(1-2\alpha_n)\|x_n-x_{n+1}\|^2
-\frac{1}{\beta}\left(1-\frac{\alpha_{n-1}\theta_{n-1}}{\psi}\right)\|y_n-y_{n-1}\|^2.
\end{multline}
Finally, since $\tau_n\le \rho\tau_{n-1}$ for all $n$, we have $\theta_n=\dfrac{\psi\tau_n}{\tau_{n-1}}\le \psi\rho\le 1+\frac1\psi$.
Moreover, following $\epsilon<\alpha_n\le \frac13$~$\forall n$, we have
\begin{equation*}
1-\frac{\alpha_{n-1}\theta_{n-1}}{\psi}
\ge 1-\alpha_{n-1}\left(\frac{1}{\psi}+\frac{1}{\psi^2}\right)
\ge 1-\frac{1}{3}\left(\frac{1}{\psi}+\frac{1}{\psi^2}\right)
>0,
\quad \forall\psi\in(1,\varphi].
\end{equation*}
Together with $1-2\alpha_n\ge\alpha_n$, these observations reduce~\eqref{adap_5} to~\eqref{adap_4}, completing the proof of the lemma.
\end{proof}

\begin{lemma}\label{prop_1}
Let $(x^*,w^*,y^*)\in\mathbf{\Pi}$, and let
${(x_n,z_n,\tau_n,\theta_n)}$ be the sequence generated by
Algorithm~\ref{algorithm 2}. Suppose that $h$ is locally smooth. Then the
sequences $(x_n)$ and $(z_n)$ are bounded. Moreover, $(\tau_n)$ and $(\theta_n)$ are bounded away from zero.
\end{lemma}
\begin{proof}
Since $(x^*, w^*, y^*)\in\mathbf{\Pi}$, we have $2\tau_n\mathbb{J}(x_n,w_n,y^*)\ge 0$ for all $n$. 
Therefore, applying~\eqref{adap_4} with $y=y^*$ and using Lemma~\ref{usefullemma2}, we deduce that $(z_n)$ is bounded. 
Then~\eqref{eqn:39a} implies that $(x_n)$ is also bounded. Let $B=\overline{\operatorname{Conv}}\{x^*,x_0,x_1,\ldots\}$ denote the closed convex hull of the iterates together with $x^*$. 
Then $B$ is closed and bounded, hence compact. Therefore, by local Lipschitz continuity of $\nabla h$, there exists $L>0$ such that $h$ is $L$-smooth on $B$, and in particular $L\ge L_n$ for all $n$. Hence combining Remark~\ref{tau_n_bdd_above} with~\cite[Lemma~4.2]{tam2023bregman}, we obtain the explicit lower bounds
\[
\tau_n\ge \underline \tau := \frac{\epsilon^2\psi^2}{U(L^2+\beta\psi\|K\|^2)}
\quad\text{and}\quad
\theta_n\ge \frac{\epsilon^2\psi^3}{U^2(L^2+\beta\psi \|K\|^2)},
\qquad \forall n.
\]
\end{proof}

\begin{theorem}\label{theorem:1}
Under Assumptions~\ref{assumption1} and \ref{assumption_5.1}, let $\{(z_n,x_n,w_n,y_n,\tau_n)\}$ be generated by Algorithm~\ref{algorithm 2}, and let $(x^*, w^*, y^*)\in\mathbf{\Pi}$. Then both sequences $\{(x_n,y_n)\}$ and $\{(z_n,y_n)\}$ converge to a solution of~\eqref{eq:saddle_fenchel}.
\end{theorem}

\begin{proof}
Since $(x^*,w^*,y^*)$ is a saddle point of $\mathbb{L}$, we have $2\tau_n\mathbb{J}(x_n,w_n,y^*)\ge 0$ for all $n\ge 1$.
Thus~\eqref{adap_4} can be rewritten in the form
\[
d_{n+1}(y)\le d_n(y)-q_n,\qquad \forall n\ge 1,
\]
where
\begin{equation}\label{Theo4.1_1}
\begin{aligned}
d_n(y) &\coloneqq \frac{\psi}{\psi-1}\|x^*-z_{n+1}\|^2+\frac{1}{\beta}\|y-y_{n-1}\|^2+\alpha_{n-1}\theta_{n-1}\|x_n-x_{n-1}\|^2,\\
q_n &\coloneqq \theta_n\|x_n-z_{n+1}\|^2.
\end{aligned}
\end{equation}
By Lemma~\ref{usefullemma2}, $\lim_{n\to\infty}d_n(y)$ exists and is finite, and moreover $\lim_{n\to\infty}q_n= 0$.
Using Lemma~\ref{prop_1}, we therefore obtain
\begin{equation}\label{Theorem4.1_2}
\lim_{n\to\infty}\|x_n-z_{n+1}\|^2
=\lim_{n\to\infty}\frac{1}{\psi^2}\|x_n-z_n\|^2
=0.
\end{equation}
Combining~\eqref{Theorem4.1_2} with Lemma~\ref{prop_1} and the triangle inequality yields
\[
\lim_{n\to\infty}\|x_n-x_{n-1}\|^2=0.
\]
Since $\lim_{n\to\infty}d_n(y)$ exists and is finite, and~\eqref{Theorem4.1_2} holds, it follows that the sequences $\{x_n\}$, $\{y_n\}$, and $\{z_n\}$ are bounded. Let $(\widetilde x,\widetilde y)$ be any cluster point of $\{(x_n,y_n)\}$, and let $\{(x_{n_k},y_{n_k})\}$ be a subsequence converging to $(\widetilde x,\widetilde y)$, that is $\lim_{k\to\infty}x_{n_k}= \widetilde x$ and $\lim_{k\to\infty}y_{n_k}= \widetilde y$. From~\eqref{Theorem4.1_2} we also have $z_{n_k}\to \widetilde x$.
Using~\eqref{eqn:39b}, \eqref{eq:w_n update}, \eqref{eq:y_n update}, \eqref{3.2.1}, \eqref{3.2.3}, Lemma~\ref{usefullemma3} and Lemma~\ref{prop_1}, we obtain that for all $(x,y)\in\mathbb{H}_1\times\mathbb{H}_2$,
\begin{equation*}
\begin{aligned}
\langle x_{n_k+1}-z_{n_k+1}+\tau_{n_k}K^*y_{n_k}+\tau_{n_k}\nabla h(x_{n_k}),\,x-x_{n_k+1}\rangle 
&\ge \tau_{n_k-1}\bigl(f(x_{n_k+1})-f(x)\bigr),\\
\left\langle \frac{1}{\beta}(y_{n_k}-y_{n_k-1})-\tau_{n_k}Kx_{n_k},\,y-y_{n_k}\right\rangle
&\ge \tau_{n_k}\bigl(g^*(y_{n_k})-g^*(y)\bigr).
\end{aligned}
\end{equation*}
Since $f$ and $g^*$ are lower semicontinuous, letting $k\to\infty$ gives
\begin{equation*}
\langle K^*\widetilde y+\nabla h(\widetilde x),\, x-\widetilde x \rangle\ge f(\widetilde x)-f(x),
\qquad
-\langle K\widetilde x, \,y-\widetilde y \rangle\ge g^*(\widetilde y)-g^*(y).
\end{equation*}
These inequalities show that $(\widetilde x,\widetilde y)$ solves~\eqref{eq:saddle_fenchel}. Finally, since Lemma~\ref{lemma_5.0.1} holds for any saddle point in $\mathbf{\Pi}$, we may set $(\bar x,\bar y)=(\widetilde x,\widetilde y)$ in~\eqref{Theo4.1_1} to conclude that $\lim_{k\to\infty}d_{n_k}(\widetilde y)=0$.
Because $\lim_{n\to\infty}d_n(\widetilde y)$ exists, this implies $\lim_{n\to\infty}d_n(\widetilde y)=0$, hence $z_n\to \widetilde x$ and $y_n\to \widetilde y$. Now using~\eqref{Theorem4.1_2} once more yields $x_n\to \widetilde x$.
This completes the proof.
\end{proof}

\subsection{Sublinear rate of convergence}

We next demonstrate the ergodic sublinear rate of Algorithm~\ref{algorithm 2} in terms of the objective residual and feasibility violation. For $N\ge 1$, define the ergodic averages
\begin{equation}\label{x_ntilde and w_ntilde}
\widetilde{x}_N\coloneqq \frac{1}{N}\sum_{n=1}^{N} x_n,
\qquad
\widetilde{w}_N\coloneqq \frac{1}{N}\sum_{n=1}^{N} w_n.
\end{equation}

\begin{theorem}[Sublinear rate of convergence]\label{sublinear rate of convergence 1}\leavevmode\par
Under Assumptions~\ref{assumption1}, \ref{assumption_5.1}, let $\{(z_n,x_n,w_n,y_n,\tau_n)\}$ be the sequence generated by Algorithm~\ref{algorithm 2}, and let $(\bar x,\bar w,\bar y)\in\mathbf{\Pi}$.
Then there exists a constant $P_1>0$ such that
\[
\bigl|\Phi(\widetilde{x}_N,\widetilde{w}_N)-\Phi(\bar x,\bar w)\bigr|\le \frac{P_1}{N},
\qquad
\|K\widetilde{x}_N-\widetilde{w}_N\|\le \frac{2P_1}{bN},
\]
for all $N\ge 1$, where $b>0$ satisfies $b\ge 2\|\bar y\|$.
\end{theorem}

\begin{proof}
Fix any $y\in\mathbb{H}_2$. From~\eqref{Theo4.1_1} and~\eqref{adap_4}, we have
\begin{equation*}
2\tau_n\mathbb{J}(x_n,w_n,y)\le d_n(y)-d_{n+1}(y),
\qquad \forall n.
\end{equation*}
By Lemma~\ref{prop_1}, there exists $\Delta>0$ such that $\tau_n\ge \Delta$ for all $n$, where $\Delta=\frac{\epsilon^2\psi^2}{U(L^2+\beta\psi\|K\|^2)}$. Hence
\[
2\Delta\,\mathbb{J}(x_n,w_n,y)\le d_n(y)-d_{n+1}(y),
\qquad \forall n.
\]
Summing from $n=1$ to $N$ yields
\begin{equation*}
\begin{aligned}
2\Delta\sum_{n=1}^{N}\mathbb{J}(x_n,w_n,y)
&\le d_1(y)-d_{N+1}(y)
\le d_1(y)\\
&= \frac{\psi}{\psi-1}\,\|\bar x-z_{2}\|^2
  + \frac{1}{\beta}\,\|y-y_0\|^2
  + \alpha_0\theta_0\,\|x_1-x_0\|^2 .
\end{aligned}
\end{equation*}
Using the definition of $(\widetilde x_N,\widetilde w_N)$ in~\eqref{x_ntilde and w_ntilde} and the joint convexity of $\mathbb{J}(\cdot,\cdot,y)$ in $(x,w)$ for fixed $y$, we obtain
\begin{equation}\label{Roc_3}
\begin{aligned}
\mathbb{J}(\widetilde{x}_N,\widetilde{w}_N,y)
&= \Phi(\widetilde{x}_N,\, \widetilde{w}_N)+ \langle y, K\widetilde{x}_N-\widetilde{w}_N \rangle - \Phi(\bar x,\, \bar w) \\ 
&\le \frac{1}{N}\sum_{n=1}^{N} \mathbb{J}(x_n,w_n,y)
\le \frac{d_1(y)}{2\Delta N}.
\end{aligned}
\end{equation}
Let $S_3\coloneqq \frac{\psi}{\psi-1}\|\bar x-z_2\|^2+\frac{1}{\beta}(b+\|y_0\|)^2+\alpha_0\theta_0\|x_1-x_0\|^2$, so that $d_1(y)\le S_3$ whenever $\|y\|\le b$. Taking the maximum of~\eqref{Roc_3} over $\|y\|\le b$ gives
\begin{equation}\label{Roc_4_new}
\Phi(\widetilde{x}_N,\, \widetilde{w}_N)+ b\|K\widetilde{x}_N-\widetilde{w}_N \| - \Phi(\bar x,\, \bar w)\le\frac{P_1}{N},
\end{equation}
where $P_1=\frac{S_3}{2\Delta}$. Since $(\bar x,\bar w,\bar y)$ is a saddle point of $\mathbb{L}$, we have $\mathbb{L}(\bar x,\bar w,\bar y)\le \mathbb{L}(\widetilde x_N,\widetilde w_N,\bar y)$. Using $K\bar x=\bar w$ and $\|\bar y\|\le\frac{b}{2}$, we obtain
\begin{equation}\label{Roc_5_new}
\Phi(\bar x,\, \bar w)-\Phi(\widetilde{x}_N,\,\widetilde{w}_N)
\le \langle \bar y,\, K\widetilde{x}_N-\widetilde{w}_N \rangle
\le \frac{b}{2}\|K\widetilde{x}_N-\widetilde{w}_N \|.
\end{equation}
Combining~\eqref{Roc_4_new} and~\eqref{Roc_5_new} yields
\[
b\|K\widetilde{x}_N-\widetilde{w}_N\|
\le \frac{b}{2}\|K\widetilde{x}_N-\widetilde{w}_N\|+\frac{P_1}{N},
\]
hence $\|K\widetilde{x}_N-\widetilde{w}_N\|\le \frac{2P_1}{bN}$. Substituting this estimate into~\eqref{Roc_4_new} together with~\eqref{Roc_5_new} gives
\[
\bigl|\Phi(\widetilde{x}_N,\, \widetilde{w}_N)- \Phi(\bar x,\,\bar w)\bigr|
\le \frac{P_1}{N}.
\]
\end{proof}

\section{Linear convergence results}
\label{subsec:linear_convergence_results}

In this section, we establish linear convergence results for PF-GRPDA
(Algorithm~\ref{algorithm 2}). Throughout, we take $\sigma_n=\beta\tau_n,~\rho=\psi^{-1}+\psi^{-2}.$ Let $\psi_0\approx1.32472$ be the unique real root of
$\psi^3-\psi-1=0$. We assume $\psi\in(\psi_0,\varphi]$, then $\psi>\rho$, and this strict inequality is the key ingredient used below to obtain a contraction.

\medskip
Let $(x^*,w^*,y^*)\in\mathbf{\Pi}$ be a saddle point. Recall that $w^*=Kx^*$, and
that the residual introduced in \eqref{primal_dual_gap} is
\[
    \mathbb J(x,w,y)
    =
    \Phi(x,w)-\Phi(x^*,w^*)
    +
    \langle y,Kx-w\rangle .
\]
By the saddle-point property $\mathbb J(x,w,y^*)\ge0$, $\forall (x,w)\in\mathbb H_1\times\mathbb H_2.$

\medskip
We first consider the case where both nonsmooth component functions in \eqref{main_prob} provide curvature, that is 
$f$ is $\mu_f$-strongly convex and $g^*$ is
$\mu_{g^*}$-strongly convex. Recall that $f$ is called $\mu$- strongly convex if \begin{equation}\label{defn: strng f}
    f(u)\ge f(v)+\langle \xi, u-v\rangle + \frac{\mu}{2}\|u-v\|^{2}
    \quad \forall u,v \in\mathbb{H}_1,\ \forall \xi \in\partial f(v).
\end{equation}

\begin{assumption}\label{ass:linear_sc}
Suppose that 
\begin{enumerate}[label=\textup{(\roman*)}]
    \item $h:\mathbb H_1\to\mathbb R$ is convex and locally smooth;
    \item $f$ is $\mu_f$-strongly convex and $g^*$ is $\mu_{g^*}$-strongly convex.
\end{enumerate}
\end{assumption}

\begin{lemma}\label{lemma:linear_step-size_bounds}
Let Assumptions \ref{assumption1} and \ref{ass:linear_sc} hold, and let
$\{(z_n,x_n,w_n,y_n,\tau_n)\}$ be generated by
Algorithm~\ref{algorithm 2}. Then there exist constants $\underline \tau, \overline{\tau}>0$ such that $\underline\tau\le\tau_n\le\overline\tau$. Consequently, $(\theta_n)$ is also bounded below and above by positive constants. 
\end{lemma}

\begin{proof}
By Remark~\ref{tau_n_bdd_above}, Lemma~\ref{lemma_5.0.1}, and
Lemma~\ref{prop_1}, the step-sizes generated by Algorithm~\ref{algorithm 2}
satisfy $0<\underline\tau\le\tau_n\le\overline\tau<+\infty$. Since $\theta_n=\frac{\psi\tau_n}{\tau_{n-1}}$, it follows that $\frac{\psi\underline\tau}{\overline\tau}\le\theta_n\le\frac{\psi\overline\tau}{\underline\tau}$.
\end{proof}

\begin{lemma}\label{lemma:linear_one_step}
Under Assumptions~\ref{assumption1} and \ref{ass:linear_sc}, let
$\{(z_n,x_n,w_n,y_n,\tau_n)\}$ be generated by Algorithm~\ref{algorithm 2}. Suppose that $(x^*, w^*, y^*)\in\mathbf{\Pi}$. Then
\begin{align}
2\tau_n\mathbb J(x_n,w_n,y^*)&+ \frac{\psi(1+\mu_f\tau_n)}{\psi-1}
    \|z_{n+2}-x^*\|^2
+
\left(
    \frac1\beta+\mu_{g^*}\tau_n
\right)
    \|y_n-y^*\|^2
\nonumber\\
&\quad
+
\bigl(\alpha_n\theta_n+\mu_f\tau_n\bigr)
    \|x_{n+1}-x_n\|^2
+
\theta_n\|x_n-z_{n+1}\|^2
\nonumber\\
&\le
\frac{\psi+\mu_f\tau_n}{\psi-1}
    \|z_{n+1}-x^*\|^2
+
\frac1\beta
    \|y_{n-1}-y^*\|^2
+
\alpha_{n-1}\theta_{n-1}
    \|x_n-x_{n-1}\|^2 .
\label{eq:linear_one_step_J}
\end{align}
\end{lemma}

\begin{proof}
Since $f$ is $\mu_f$-strongly convex, by \eqref{eqn:39b} and \eqref{defn: strng f}, we have
\begin{align*}
&
\left\langle
    x_{n+1}-z_{n+1}
    +
    \tau_nK^*y_n
    +
    \tau_n\nabla h(x_n),
    x^*-x_{n+1}
\right\rangle
\nonumber\\
&\qquad
\ge
\tau_n\bigl(f(x_{n+1})-f(x^*)\bigr)
+
\frac{\mu_f\tau_n}{2}\|x_{n+1}-x^*\|^2.
\end{align*}
Similarly, we obtain
\begin{align*}
&
\left\langle
    \theta_n (x_n-z_{n+1})
    +
    \tau_{n}K^*y_{n-1}
    +
    \tau_{n}\nabla h(x_{n-1}),
    x_{n+1}-x_n
\right\rangle
\nonumber\\
&\qquad
\ge
\tau_{n}\bigl(f(x_n)-f(x_{n+1})\bigr)
+
\frac{\mu_f\tau_{n}}{2}\|x_{n+1}-x_n\|^2 .
\end{align*}
From the $w_n$-update \eqref{eq:w_n update}, we get
\[
    y_{n-1}+\sigma_nKx_n-\sigma_n w_n\in\partial g(w_n).
\]
Further, following \eqref{eq:y_n update}, we have $y_n\in\partial g(w_n).$ Since $g$ is proper, convex and lsc, $w_n\in\partial g^*(y_n).$ Similarly, at the saddle point, we have $y^*\in\partial g(w^*)$, hence $w^*\in\partial g^*(y^*).$ Therefore, using the strong convexity of $g^*$, we obtain
\[
    g^*(y_n)
    \ge
    g^*(y^*)
    +
    \langle w^*,y_n-y^*\rangle
    +
    \frac{\mu_{g^*}}{2}\|y_n-y^*\|^2.
\]
By the Fenchel--Young inequality \cite[Proposition 16.10]{bauschke2017correction}
\[
    g^*(y_n)=\langle y_n,w_n\rangle-g(w_n)~~\text{and}~~
    g^*(y^*)=\langle y^*,w^*\rangle-g(w^*).
\]
Substituting these identities into the previous inequality yields
\[
    \langle y_n,w_n\rangle-g(w_n)
    \ge
    \langle y^*,w^*\rangle-g(w^*)
    +
    \langle w^*,y_n-y^*\rangle
    +
    \frac{\mu_{g^*}}{2}\|y_n-y^*\|^2.
\]
This is equivalent to
\begin{equation*}
    \langle y_n,w_n-w^*\rangle
    \ge
    g(w_n)-g(w^*)
    +
    \frac{\mu_{g^*}}{2}\|y_n-y^*\|^2.
\end{equation*}
 From this point, by an analogous argument as in Lemma \ref{lemma_5.0.1}, we obtain \eqref{eq:linear_one_step_J}.
\end{proof}

\medskip
Define
\begin{equation*}
    \mathcal E_n
    :=
    \frac{\psi+\mu_f\tau_n}{\psi-1}\|z_{n+1}-x^*\|^2
    +
    \frac1\beta\|y_{n-1}-y^*\|^2
    +
    \alpha_{n-1}\theta_{n-1}\|x_n-x_{n-1}\|^2 .
\end{equation*}

\begin{lemma}\label{lemma:linear_contraction}
There exists $\eta>1$ such that
\begin{equation}\label{lemma: descent}
        \eta\mathcal E_{n+1}\le\mathcal E_n,
    \qquad
    \forall n\ge1.
\end{equation}

\end{lemma}

\begin{proof}
Since $\tau_{n+1}\le\rho\tau_n$, we have
\[
    \frac{\psi(1+\mu_f\tau_n)}{\psi+\mu_f\tau_{n+1}}
    \ge
    1+
    \frac{(\psi-\rho)\mu_f\tau_n}{\psi+\rho\mu_f\tau_n}.
\]
By Lemma~\ref{lemma:linear_step-size_bounds}, define the constant
\[
    \eta_1
    :=
    1+
    \frac{(\psi-\rho)\mu_f\underline\tau}
         {\psi+\rho\mu_f\overline\tau}.
\]
Since $\psi>\rho$ for any $\psi\in(\psi_0, \varphi)$, we have $\eta_1>1$ and
\[
    \frac{\psi(1+\mu_f\tau_n)}{\psi-1}
    \ge
    \eta_1\frac{\psi+\mu_f\tau_{n+1}}{\psi-1}.
\]
Also
\[
    \frac1\beta+\mu_{g^*}\tau_n
    \ge
    \eta_2\frac1\beta,
\]
where $\eta_2:=1+\beta\mu_{g^*}\underline\tau$. Finally, recalling that $\epsilon\le\alpha_n\le1/3$ and $\theta_n\le\overline\theta:=\frac{\psi\overline\tau}{\underline\tau}$, we derive
\[
    \alpha_n\theta_n+\mu_f\tau_n
    \ge
    \eta_3\alpha_n\theta_n,
\]
where $\eta_3:=1+\frac{3\mu_f\underline\tau}{\overline\theta}$. Take $\eta:=\min\{\eta_1,\eta_2,\eta_3\}.$ Then $\eta>1$. Combining these coefficient estimates with Lemma~\ref{lemma:linear_one_step}
and dropping the nonnegative terms
$2\tau_n\mathbb J(x_n,w_n,y^*)$ and
$\theta_n\|x_n-z_{n+1}\|^2$, we obtain \eqref{lemma: descent}.
\end{proof}

\begin{theorem}\label{thm:linear_pf_grpda}
Let Assumptions~\ref{assumption1} and \ref{ass:linear_sc} hold, and let $\{(z_n,x_n,w_n,y_n,\tau_n)\}$ be generated by Algorithm~\ref{algorithm 2}. Let $(x^*,w^*,y^*)$ be the unique saddle point of $\mathbb L$. Then there exists a constant
$C_0>0$ and $q\in(0,1)$ such that
\begin{equation}\label{eq:linear_rate_main}
    \|z_{n+1}-x^*\|^2
    +
    \|y_{n-1}-y^*\|^2
    +
    \|x_n-x_{n-1}\|^2
    \le
    C_0q^n \quad \forall n\ge 1.
\end{equation}
Furthermore, there exist constants $C_1,C_2>0$ such that for all $n\ge 1$
\[
    \|x_n-x^*\|^2
    \le
    C_1q^n  ~~\text{and}~~ \mathbb J(x_n,w_n,y^*)\le C_2q^n.
\]
\end{theorem}

\begin{proof}
By Lemma~\ref{lemma:linear_contraction}, there exists $\eta>1$ such that
\[
    \mathcal E_{n+1}\le \frac{1}{\eta}\mathcal E_n .
\]
Set $q:=\eta^{-1}\in(0,1)$. Then
\[
    \mathcal E_n\le q^{n-1}\mathcal E_1
    =
    \frac{\mathcal E_1}{q}q^n,
    \quad \forall n\ge1.
\]
By the definition of $\mathcal E_n$, we have
\[
    \mathcal E_n
    \ge
    p_0
    \left(
        \|z_{n+1}-x^*\|^2
        +
        \|y_{n-1}-y^*\|^2
        +
        \|x_n-x_{n-1}\|^2
    \right),
\]
where
\[
    p_0:=
    \min\left\{
        \frac{\psi}{\psi-1},\,
        \frac1\beta,\,
        \epsilon\underline\theta
    \right\}>0.
\]
Therefore,
\[
    \|z_{n+1}-x^*\|^2
    +
    \|y_{n-1}-y^*\|^2
    +
    \|x_n-x_{n-1}\|^2
    \le
    C_0q^n,
    \qquad \forall n\ge1,
\]
where $C_0:=\frac{\mathcal E_1}{q\,p_0}.$ This proves \eqref{eq:linear_rate_main}. Next, from $z_{n+1}=\frac{\psi-1}{\psi}x_n+\frac1\psi z_n,$ we have $x_n
    =
    \frac{\psi}{\psi-1}z_{n+1}
    -
    \frac1{\psi-1}z_n.$ Hence, for $n\ge1$
\[
\begin{aligned}
    \|x_n-x^*\|^2
    &\le
    \frac{2\psi^2}{(\psi-1)^2}\|z_{n+1}-x^*\|^2
    +
    \frac{2}{(\psi-1)^2}\|z_n-x^*\|^2  \\
    &\le C_1 q^n,
\end{aligned}
\]
where
\[
    C_1:=
        \left(
            \frac{2\psi^2}{(\psi-1)^2}
            +
            \frac{2}{q(\psi-1)^2}
        \right)C_0.
\]
Finally, Lemma~\ref{lemma:linear_one_step} gives
\[
    2\tau_n\mathbb J(x_n,w_n,y^*)\le \mathcal E_n.
\]
Since $\tau_n\ge\underline\tau>0$, we have
\[
    \mathbb J(x_n,w_n,y^*)
    \le
    \frac{\mathcal E_n}{2\underline\tau}
    \le
    C_2q^n,
\]
where $C_2:=\frac{\mathcal E_1}{2q\underline\tau}.$ This completes the proof.
\end{proof}

\section{Accelerated adaptive GRPDA}\label{sec:acc}

In this section, we develop an accelerated variant of Algorithm~\ref{algorithm 2}
under the additional assumption that $f$ is strongly convex. Such acceleration
is standard in primal--dual methods when $h$ is globally smooth, since curvature
in one block of the saddle formulation can be exploited to obtain an improved
ergodic $\mathcal{O}(1/N^{2})$ rate; e.g., see
\cite{chambolle2011first,chang2021goldengrpda,chang2022goldenlinesearch}.
However, we show that the same type of acceleration can be obtained without assuming
a global Lipschitz constant for $\nabla h$. More precisely, we prove an accelerated
ergodic rate when $f$ is strongly convex and $h$ is only locally smooth.

\medskip
\begin{assumption}\label{assump:strong f}
The function $f$ is $\mu$-strongly convex for some $\mu>0$.
\end{assumption}

We now state the proposed algorithm.

\medskip
\begin{algorithm}[H]
\caption{Parameter-Free Accelerated GRPDA (PF-AGRPDA)}
\label{algorithm 2 acclerated}
\KwIn{Choose $x_0\in\mathbb H_1$, $y_0\in\mathbb H_2$, and set $z_0=x_0$. Let $\psi_0\approx1.32472$ be the unique real root of $\psi^3-\psi-1=0$. Choose $\beta_0>0$, $\psi\in(\psi_0,\varphi)$, $0<\alpha\le 1/3$. Let $\rho=\psi^{-1}+\psi^{-2}$, $\theta_0=1$, $C=\psi\|K\|^2$, $\Theta=\max\{\theta_0,\psi\rho\}$, and $U=\sqrt{\frac{\rho\psi\alpha^2\Theta}{C\beta_0}},$ and choose $0<\tau_0\le U$.}

\For{$n=1,2,\ldots$}{
  \textbf{Step 1} (Compute)
  \begin{align}
    z_n
    &=
    \frac{\psi-1}{\psi}x_{n-1}
    +
    \frac{1}{\psi}z_{n-1},\nonumber\\
    x_n
    &=
    \prox_{\tau_{n-1}f}
    \left(
        z_n-\tau_{n-1}K^*y_{n-1}
        -\tau_{n-1}\nabla h(x_{n-1})
    \right).\nonumber
  \end{align}

  \textbf{Step 2} (Update)
  \begin{align}
    \zeta_n
    &=
    \frac{\psi-\rho}{\psi+\rho\mu\tau_{n-1}}, \label{step-size: 1}\\
    \beta_n
    &=
    \beta_{n-1}\bigl(1+\mu\zeta_n\tau_{n-1}\bigr), \label{stepszie: 2}\\
    \tau_n
    &=
    \min\left\{
        \rho\tau_{n-1},
        \frac{\psi\alpha^2\theta_{n-1}}
        {(L_n^2+\beta_n\psi\|K\|^2)\tau_{n-1}}
    \right\},
    \qquad
    \sigma_n=\beta_n\tau_n . \label{step-size: 3}
  \end{align}

  \textbf{Step 3} (Compute)
  \begin{align}
    w_n
    &=
    \prox_{\frac{1}{\sigma_n}g}
    \left(
        \frac{y_{n-1}}{\sigma_n}+Kx_n
    \right), \nonumber\\
    y_n
    &=
    y_{n-1}
    +
    \sigma_n(Kx_n-w_n).\nonumber
  \end{align}

  \textbf{Step 4} (Update)
  \begin{equation*}
    \theta_n=\frac{\psi\tau_n}{\tau_{n-1}}.
  \end{equation*}
}
\end{algorithm}

Some remarks on Algorithm~\ref{algorithm 2 acclerated} are in order.

\begin{remark}\label{rem:psi0_zeta}
As also observed in \cite{chang2021goldengrpda}, $\psi_0\approx 1.32472$ is the unique real root of $\psi^3-\psi-1=~0$.
Given $\psi\in(\psi_0,\varphi)$, one has $\rho=\psi^{-1}+\psi^{-2}$ and $\psi-\rho>0$. Thus $\zeta_n>0$ for all $n$
by~\eqref{step-size: 1}.
\end{remark}

\begin{remark}\label{rem:mu0_reduction}
If $\mu=0$ (so that $f$ is convex), then \eqref{stepszie: 2} gives $\beta_n\equiv \beta_0$ and
\eqref{step-size: 3} simplifies to the step-size rule \eqref{eq:39c}. Consequently, Algorithm~\ref{algorithm 2 acclerated} reduces to Algorithm~\ref{algorithm 2}.
\end{remark}

\begin{lemma}\label{lemma:acc_beta_tau_zeta_1}
Let $\{(\tau_n,\zeta_n,\beta_n)\}$ be generated by
Algorithm~\ref{algorithm 2 acclerated}. Then the following assertions hold.
\begin{enumerate}[label=\textup{(\roman*)}]
    \item\label{it:acc_zeta_beta}
    $(\beta_n)$ is strictly increasing, and $      \underline{\zeta}\le \zeta_n<1
        \quad \forall n\ge1,$ where $        \underline{\zeta}:=
        \frac{\psi-\rho}{\psi+\rho\mu U}.$

    \item\label{it:acc_tau_upper}
    $\tau_n\le U$ for all $n\ge1$. Moreover, $\tau_n^2\beta_n\le
        \widehat B:=
        \frac{\alpha^2\Theta^2}{C}, \forall n\ge 1$.
\end{enumerate}
\end{lemma}

\begin{proof}
\noindent
\textbf{\ref{it:acc_zeta_beta} and \ref{it:acc_tau_upper}.}
From the update of $\beta_n$ and the facts $\mu>0$, $\tau_n>0$, and
$\zeta_{n+1}>0$, it follows that the sequence $(\beta_n)$ is strictly increasing. Also,
$\tau_n\le\rho\tau_{n-1}$ gives
$\theta_n\le\Theta$ for all $n\ge0$. From the step-size rule \eqref{step-size: 3} and Lemma~\ref{min a and b}, we have
\[
    \tau_n
    \le
    \sqrt{
        \frac{\rho\psi\alpha^2\theta_{n-1}}{L_n^2+C\beta_n}
    }
    \le
    \sqrt{
        \frac{\rho\psi\alpha^2\Theta}{C\beta_0}
    }
    =U.
\]
Next, the second term in \eqref{step-size: 3} gives
\[
    \tau_n
    \le
    \frac{\psi\alpha^2\theta_{n-1}}
    {(L_n^2+C\beta_n)\tau_{n-1}}.
\]
Using $\frac{\psi\theta_{n-1}}{\tau_{n-1}} =
\frac{\theta_n\theta_{n-1}}{\tau_n},$ we obtain
\[
    \tau_n^2
    \le
    \frac{\alpha^2\theta_n\theta_{n-1}}{L_n^2+C\beta_n}.
\]
Since $\theta_n,\theta_{n-1}\le\Theta$, it follows that
\[
    \tau_n^2\beta_n
    \le
    \frac{\alpha^2\Theta^2}{C}
    =
    \widehat B
    \quad \forall n\ge1.
\]
Since $\tau_{n-1}\le U$ for every $n\ge1$, we get
\[
    \zeta_n
    =
    \frac{\psi-\rho}{\psi+\rho\mu\tau_{n-1}}
    \ge
    \frac{\psi-\rho}{\psi+\rho\mu U}
    =
    \underline{\zeta}.
\]
Finally, $\zeta_n<1$ is immediate from
$\psi-\rho<\psi+\rho\mu\tau_{n-1}$. This proves
\ref{it:acc_zeta_beta} and \ref{it:acc_tau_upper}.
\end{proof}


\begin{lemma}\label{acclerated: lem:1}
 Under Assumptions \ref{assumption1}, \ref{assumption_5.1} and \ref{assump:strong f}, let
$\{(z_n,x_n,w_n,y_n,\tau_n)\}$ be the sequence generated by
Algorithm \ref{algorithm 2 acclerated}. Let
$(x^*,w^*,y^*)\in\mathbf{\Pi}$ be any saddle point of $\mathbb{L}$.
Then, for any $y\in\mathbb{H}_2$ and for any $n\geq 1$, we have
\begin{multline}\label{acclerated lem eq:1}
     2\tau_n\mathbb{J}(x_n,w_n, y)
     +  \frac{\psi(1+\mu\tau_n)}{\psi-1}
        \lVert x^*-z_{n+2} \rVert^2
     + \dfrac{1}{\beta_n}\lVert y-y_{n} \rVert^2
     +(\alpha\theta_n+\mu\tau_n)\lVert x_n-x_{n+1} \rVert^2       \\
     \leq
     \frac{\psi+\mu\tau_n}{\psi-1}
        \lVert x^*-z_{n+1} \rVert^2
     + \dfrac{1}{\beta_n}\lVert y-y_{n-1} \rVert^2
     +\alpha\theta_{n-1}\|x_n-x_{n-1}\|^2 .
 \end{multline}
\end{lemma}

\begin{proof}
First observe that, by the optimality condition of the $x_{n+1}$-update, we have
\[
    \frac{z_{n+1}-x_{n+1}}{\tau_n}
    -K^*y_n-\nabla h(x_n)
    \in\partial f(x_{n+1}).
\]
Using the definition of strong convexity \eqref{defn: strng f} of $f$, we obtain, for every $u\in\mathbb H_1$,
\begin{equation*}
\begin{aligned}
\tau_n\big(f(x_{n+1})-f(u)\big)
&\le
\left\langle
x_{n+1}-z_{n+1}
+\tau_nK^*y_n
+\tau_n\nabla h(x_n),
u-x_{n+1}
\right\rangle                                      \\
&\quad
-\frac{\mu\tau_n}{2}\|x_{n+1}-u\|^2 .
\end{aligned}
\end{equation*}
Similarly, applying the fact
$x_n-z_n=\psi(x_n-z_{n+1})$, we again have
\begin{equation*}
\begin{aligned}
\tau_n\big(f(x_n)-f(x_{n+1})\big)
&\le
\left\langle
\theta_n(x_n-z_{n+1})
+\tau_nK^*y_{n-1}
+\tau_n\nabla h(x_{n-1}),
x_{n+1}-x_n
\right\rangle                                      \\
&\quad
-\frac{\mu\tau_n}{2}\|x_{n+1}-x_n\|^2 .
\end{aligned}
\end{equation*}
By applying analogous arguments as in Lemma \ref{lemma_5.0.1}, we obtain
\begin{multline}\label{strng f: eq 1}
     \tau_n\mathbb{J}(x_n,w_n, y)
     +  \dfrac{1}{2}\lVert x^*-x_{n+1} \rVert^2
     + \dfrac{1}{2\beta_n}\lVert y-y_{n} \rVert^2  
     \leq
     \dfrac{1}{2}\lVert x^*-z_{n+1} \rVert^2
     + \dfrac{1}{2\beta_n}\lVert y-y_{n-1} \rVert^2
     \\
     -\dfrac{1}{2}\lVert x_{n+1}-z_{n+1} \rVert^2  
     - \dfrac{\theta_n}{2}\lVert x_n-z_{n+1} \rVert^2
     -\dfrac{\theta_n}{2}\lVert x_n-x_{n+1} \rVert^2
     + \dfrac{\theta_n}{2}\lVert x_{n+1}-z_{n+1} \rVert^2          \\
     -\dfrac{1}{2\beta_n}\lVert y_n-y_{n-1} \rVert^2
     -\frac{\mu\tau_n}{2}\|x_{n+1}-x_n\|^2
     -\frac{\mu\tau_n}{2}\|x_{n+1}-x^*\|^2                         \\
     +\tau_n\langle y_n-y_{n-1},\, K(x_{n+1}-x_n)\rangle           
     + \tau_n\langle \nabla h(x_n)-\nabla h(x_{n-1}),
        x_n-x_{n+1}\rangle .
 \end{multline}
Simplifying \eqref{strng f: eq 1}, we get
\begin{multline}\label{strng f: eq 2}
     2\tau_n\mathbb{J}(x_n,w_n, y)
     +  (1+\mu\tau_n)\lVert x^*-x_{n+1} \rVert^2
     + \dfrac{1}{\beta_n}\lVert y-y_{n} \rVert^2                  
     \leq
     \lVert x^*-z_{n+1} \rVert^2\\
     +\dfrac{1}{\beta_n}\lVert y-y_{n-1} \rVert^2
     +(\theta_n-1)\lVert x_{n+1}-z_{n+1} \rVert^2                 
     - \theta_n\lVert x_n-z_{n+1} \rVert^2
     -(\theta_n+\mu\tau_n)\lVert x_n-x_{n+1} \rVert^2\\
     -\dfrac{1}{\beta_n}\lVert y_n-y_{n-1} \rVert^2               
     +2\tau_n\langle y_n-y_{n-1},\, K(x_{n+1}-x_n)\rangle          
     +2\tau_n\langle \nabla h(x_n)-\nabla h(x_{n-1}),
        x_n-x_{n+1}\rangle .
 \end{multline}
By combining \eqref{derived_equality2} with \eqref{strng f: eq 2}, we obtain
\begin{multline}\label{strng f: eq 3}
     2\tau_n\mathbb{J}(x_n,w_n, y)
     +  \frac{\psi(1+\mu\tau_n)}{\psi-1}
        \lVert x^*-z_{n+2} \rVert^2 \\
     {}+ \dfrac{1}{\beta_n}\lVert y-y_{n} \rVert^2                  
     \leq
     \frac{\psi+\mu\tau_n}{\psi-1}
        \lVert x^*-z_{n+1} \rVert^2\\
     + \dfrac{1}{\beta_n}\lVert y-y_{n-1} \rVert^2                
     +\left(\theta_n-1-\frac{1+\mu\tau_n}{\psi}\right)
        \lVert x_{n+1}-z_{n+1} \rVert^2
     - \theta_n\lVert x_n-z_{n+1} \rVert^2                       \\
     -(\theta_n+\mu\tau_n)\lVert x_n-x_{n+1} \rVert^2
     -\dfrac{1}{\beta_n}\lVert y_n-y_{n-1} \rVert^2               
     +2\tau_n\langle y_n-y_{n-1},\, K(x_{n+1}-x_n)\rangle          \\
     +2\tau_n\langle \nabla h(x_n)-\nabla h(x_{n-1}),
        x_n-x_{n+1}\rangle .
 \end{multline}
Following \eqref{adap_2}, the last term in \eqref{strng f: eq 3} can be estimated as
\begin{equation}\label{strng f: eq 4}
\begin{aligned}
2\tau_n\langle \nabla h(x_n)-\nabla h(x_{n-1}),
x_n-x_{n+1}\rangle
&\leq
\alpha\theta_n\lVert x_n-x_{n+1}\rVert^2
+\alpha\theta_{n-1}\lVert x_n-x_{n-1}\rVert^2 .
\end{aligned}
\end{equation}
Similarly, using Remark~\ref{remark_5.2}, \eqref{adap_3}, and
\eqref{step-size: 3}, we have
\begin{equation}\label{strng f: eq 5}
\begin{aligned}
2\tau_n\langle y_n-y_{n-1},\, K(x_{n+1}-x_n)\rangle
&\leq
\alpha\theta_n\lVert x_n-x_{n+1} \rVert^2
+\dfrac{\alpha\theta_{n-1}}{\beta_n\psi}
    \lVert y_n-y_{n-1} \rVert^2 .
\end{aligned}
\end{equation}
Altogether, \eqref{strng f: eq 3}, \eqref{strng f: eq 4}, and
\eqref{strng f: eq 5} yield
\begin{align*}
&2\tau_n\mathbb{J}(x_n,w_n,y)
 +\frac{\psi(1+\mu\tau_n)}{\psi-1}\lVert x^*-z_{n+2}\rVert^2
 +\frac{1}{\beta_n}\lVert y-y_n\rVert^2 \\
&\quad\leq
 \frac{\psi+\mu\tau_n}{\psi-1}\lVert x^*-z_{n+1}\rVert^2
 +\frac{1}{\beta_n}\lVert y-y_{n-1}\rVert^2 \\
&\qquad
 +\left(\theta_n-1-\frac{1+\mu\tau_n}{\psi}\right)
  \lVert x_{n+1}-z_{n+1}\rVert^2
 -\theta_n\lVert x_n-z_{n+1}\rVert^2 \\
&\qquad
 -(\theta_n(1-2\alpha)+\mu\tau_n)\lVert x_n-x_{n+1}\rVert^2
 +\alpha\theta_{n-1}\lVert x_n-x_{n-1}\rVert^2 \\
&\qquad
 -\frac{1}{\beta_n}\left(1-\frac{\alpha\theta_{n-1}}{\psi}\right)
  \lVert y_n-y_{n-1}\rVert^2 .
\end{align*}
From \eqref{step-size: 3}, we have $\theta_n-1-\frac{1+\mu\tau_n}{\psi}\leq \theta_n-1-\frac{1}{\psi}\leq 0$ and $1-\frac{\alpha\theta_{n-1}}{\psi}>0$ for all $n$. Furthermore, note that $1-2\alpha\ge \alpha$. Hence, using these facts yields Lemma \ref{acclerated: lem:1}.
\end{proof}


\begin{lemma}\label{lem:acc_weighted_fejer}
Under Assumptions~\ref{assumption1}, \ref{assumption_5.1}  and~\ref{assump:strong f}, let the sequence $\{(z_n,x_n,w_n,y_n,\tau_n)\}$ be generated by
Algorithm~\ref{algorithm 2 acclerated}, and fix a saddle point
$(x^*,w^*,y^*)\in\mathbf\Pi$. For $y\in\mathbb H_2$, and for all $n\ge 1$, define
\[
    E_n(y):=
    \frac{\psi+\mu\tau_n}{2(\psi-1)}\|x^*-z_{n+1}\|^2
    +\frac1{2\beta_n}\|y-y_{n-1}\|^2
    +\frac{\alpha\theta_{n-1}}2\|x_n-x_{n-1}\|^2.
\]
Then, for all $y\in\mathbb H_2$, we have
\begin{equation}\label{eq:acc_weighted_descent}
    \beta_n\tau_n\mathbb J(x_n,w_n,y)
    +\beta_{n+1}E_{n+1}(y)
    \le
    \beta_nE_n(y), ~~n\ge 1.
\end{equation}
Consequently, the sequences $(z_n)$, $(x_n)$ and $(y_n)$ are bounded.
\end{lemma}

\begin{proof}
To telescope \eqref{acclerated lem eq:1} in Lemma~\ref{acclerated: lem:1}, we proceed as follows. Since $\tau_{n+1}\le\rho\tau_n$ and
\[
    \frac{\beta_{n+1}}{\beta_n}
    =1+\mu\zeta_{n+1}\tau_n
    =1+\frac{\mu\tau_n(\psi-\rho)}{\psi+\rho\mu\tau_n},
\]
we get
\begin{equation}\label{eq:acc_coeff_z}
\begin{aligned}
    \frac{\beta_{n+1}}{\beta_n}(\psi+\mu\tau_{n+1})
    &\le
    \left(1+\frac{\mu\tau_n(\psi-\rho)}{\psi+\rho\mu\tau_n}\right)
    (\psi+\rho\mu\tau_n)                  =\psi(1+\mu\tau_n).
\end{aligned}
\end{equation}
Next, notice that
\begin{equation}\label{eq:acc_coeff_dx}
\begin{aligned}
  \beta_n(\alpha\theta_n+\mu\tau_n)-\beta_{n+1}\alpha\theta_n
    &=\beta_n\mu\tau_n\bigl(1-\alpha\theta_n\zeta_{n+1}\bigr)\ge0,
\end{aligned}
\end{equation}
where the last inequality follows from
$\theta_n\le\psi\rho<2$, $\zeta_{n+1}<1$, and $0<\alpha\le\frac{1}{3}$. Now multiplying~\eqref{acclerated lem eq:1} by $\frac{\beta_n}{2}$ and using
\eqref{eq:acc_coeff_z} and \eqref{eq:acc_coeff_dx} gives
\eqref{eq:acc_weighted_descent}.

\medskip
Take $y=y^*$, then by \eqref{Eq:measure}, 
$\mathbb J(x_n,w_n,y^*)\ge0$. Furthermore, from \eqref{eq:acc_weighted_descent}, we obtain
\[
    \beta_nE_n(y^*)\le \beta_1E_1(y^*)
    \qquad \forall n\ge1.
\]
In particular
\[
    \beta_n\frac{\psi+\mu\tau_n}{2(\psi-1)}\|x^*-z_{n+1}\|^2
    \le \beta_1E_1(y^*).
\]
Since $\beta_n\ge\beta_0>0$ and $\psi+\mu\tau_n\ge\psi$, it implies that the sequence
$(z_n)$ is bounded. Therefore, by the definition of $(z_n)$, the boundedness $(x_n)$ follows. Finally
\[
    \frac12\|y^*-y_{n-1}\|^2
    \le \beta_nE_n(y^*)
    \le \beta_1E_1(y^*)
\]
shows that $(y_n)$ is bounded.
\end{proof}


\begin{remark}\label{eq:acc_bdd_x_n}
Since $(x_n)$ is bounded (by Lemma~\ref{lem:acc_weighted_fejer}) and $h$ is locally smooth, there exist a compact set $B\subset\mathbb H_1$ containing all the iterates $x_n$, and $x^*$, and a constant $L>0$ such that
\[
    \|\nabla h(u)-\nabla h(v)\|
    \le L\|u-v\|
    \quad \forall u,v\in B.
\]
By the definition of $L_n$, this yields $L_n\le L$ for all $n\ge1$.
\end{remark}

\begin{lemma}[Lower step-size bound and quadratic growth of $\beta_n$]\label{lemma:acc_beta_tau_zeta}
Under Assumptions~\ref{assumption1}, \ref{assumption_5.1}  and~\ref{assump:strong f}, let $\{\tau_n\}$ and $\{\beta_n\}$ be generated by
Algorithm~\ref{algorithm 2 acclerated}. Let $L>0$ be the constant from
Remark~\ref{eq:acc_bdd_x_n}. Then the following assertions hold.
\begin{enumerate}[label=\textup{(\roman*)}]
    \item\label{it:acc_tau_lower}
    There exists $\bar M>0$ such that
    \[
        \tau_n\ge \frac{\bar M}{\sqrt{L^2+C\beta_n}}
        \qquad \forall n\ge1.
    \]

    \item\label{it:acc_beta_quad}
    Let $\kappa:=\mu\underline{\zeta} \bar M$. Then there exists
    $N_0\in\mathbb N$ and $c_0>0$ such that
    \[
        \beta_n\ge c_0n^2
        \qquad \forall n\ge N_0,
    \]
    where one may take $c_0:=\frac{\kappa^2}{512C}$.
\end{enumerate}
\end{lemma}

\begin{proof}
We first prove \ref{it:acc_tau_lower}. By Lemma~\ref{lemma:acc_beta_tau_zeta_1},
$\tau_n\le U$, $\tau_n^2\beta_n\le \widehat B$ for $n\ge1$, and
$\zeta_n<1$. Hence
\[
    \beta_{n+1}
    =\beta_n(1+\mu\zeta_{n+1}\tau_n)
    \le (1+\mu U)\beta_n
    \qquad \forall n\ge0.
\]
Thus, for every $n\ge2$,
\[
    \beta_{n+1}\le(1+\mu U)^2\beta_{n-1}.
\]
Combining this with Lemma~\ref{lemma:acc_beta_tau_zeta_1}\ref{it:acc_tau_upper}, we obtain
\[
    \tau_{n-1}^2\beta_{n+1}
    \le
    (1+\mu U)^2\tau_{n-1}^2\beta_{n-1}
    \le
    (1+\mu U)^2\widehat B,
    \quad n\ge2.
\]
Consequently
\begin{equation}\label{eq:acc_aux_bound_for_tau_lower}
    \tau_{n-1}^2(L^2+C\beta_{n+1})
    \le
    U^2L^2+C(1+\mu U)^2\widehat B,
    \quad n\ge2.
\end{equation}
Define
\[
    \bar M:=
    \min\left\{
        \tau_1\sqrt{L^2+C\beta_1},
        \tau_2\sqrt{L^2+C\beta_2},
        \frac{\psi^2\alpha^2}
        {\sqrt{U^2L^2+C(1+\mu U)^2\widehat B}}
    \right\}.
\]
Then $\bar M>0$. We claim that
\[
    \tau_n\ge\frac{\bar M}{\sqrt{L^2+C\beta_n}}
    \quad \forall n\ge1.
\]
This holds for $n=1,2$ by the definition of $\bar M$. Suppose it holds for some
$n\ge2$. We now show that it is true for $n+1$. If the first branch in \eqref{step-size: 3} is the minimum, then
\[
    \tau_{n+1}=\rho\tau_n
    \ge
    \frac{\rho \bar M}{\sqrt{L^2+C\beta_n}}
    \ge
    \frac{\bar M}{\sqrt{L^2+C\beta_{n+1}}},
\]
since $\rho>1$ and $(\beta_n)$ is increasing. And if the second branch is chosen,
then
\[
\begin{aligned}
    \tau_{n+1}
    &=
    \frac{\psi\alpha^2\theta_n}
    {(L_{n+1}^2+C\beta_{n+1})\tau_n}
    =
    \frac{\psi^2\alpha^2}
    {(L_{n+1}^2+C\beta_{n+1})\tau_{n-1}}.
\end{aligned}
\]
Since $L_{n+1}\le L~\forall n$, using \eqref{eq:acc_aux_bound_for_tau_lower} gives
\[
\begin{aligned}
    \tau_{n+1}
    &\ge
    \frac{\psi^2\alpha^2}
    {(L^2+C\beta_{n+1})\tau_{n-1}}                                      \\
    &\ge
    \frac{\psi^2\alpha^2}
    {\sqrt{U^2L^2+C(1+\mu U)^2\widehat B}}
    \frac1{\sqrt{L^2+C\beta_{n+1}}}                                      \\
    &\ge
    \frac{\bar M}{\sqrt{L^2+C\beta_{n+1}}}.
\end{aligned}
\]
Thus \ref{it:acc_tau_lower} follows by induction.

\medskip
We now prove \ref{it:acc_beta_quad}. From Lemma~\ref{lemma:acc_beta_tau_zeta_1} and
Lemma \ref{lemma:acc_beta_tau_zeta}\ref{it:acc_tau_lower}, observe that
\[
    \mu\zeta_{n+1}\tau_n
    \ge
    \frac{\mu\underline{\zeta} \bar M}{\sqrt{L^2+C\beta_n}}
    =
    \frac{\kappa}{\sqrt{L^2+C\beta_n}}.
\]
Therefore,
\begin{equation}\label{eq:acc_beta_recursion}
    \beta_{n+1}
    \ge
    \beta_n\left(1+\frac{\kappa}{\sqrt{L^2+C\beta_n}}\right).
\end{equation}
This recursion implies that $(\beta_n)$ is unbounded. Indeed, if
$\beta_n\le B$ for all $n$, then \eqref{eq:acc_beta_recursion} gives
\[
    \beta_{n+1}
    \ge
    \left(1+\frac{\kappa}{\sqrt{L^2+CB}}\right)\beta_n,
\]
which contradicts boundedness. Let $s_n:=\sqrt{L^2+C\beta_n}.$ Then $(s_n)$ is increasing and unbounded, and \eqref{eq:acc_beta_recursion} gives
\[
\begin{aligned}
    s_{n+1}^2
    &=L^2+C\beta_{n+1}                                      \\
    &\ge
    L^2+C\beta_n
    +\frac{C\kappa\beta_n}{\sqrt{L^2+C\beta_n}}                  \\
    &=
    s_n^2+\kappa\left(s_n-\frac{L^2}{s_n}\right).
\end{aligned}
\]
Since $(s_n)$ is unbounded, we can find some $n_0\in\mathbb N$ such that $\forall n\ge n_0$
\[
    s_n\ge\max\left\{\sqrt2L,\frac{\kappa}{16}\right\}.
\]
Then,
\[
    s_n-\frac{L^2}{s_n}\ge\frac{s_n}{2},~~\forall n\ge n_0
\]
and hence
\[
    s_{n+1}^2\ge s_n^2+\frac{\kappa}{2}s_n~~\forall n\ge n_0.
\]
Since $s_n\ge\frac{\kappa}{16}$
\[
    \left(s_n+\frac{\kappa}{8}\right)^2
    =s_n^2+\frac{\kappa}{4}s_n+\frac{\kappa^2}{64}
    \le
    s_n^2+\frac{\kappa}{2}s_n.
\]
Thus
\[
    s_{n+1}\ge s_n+\frac{\kappa}{8}
    \qquad \forall n\ge n_0.
\]
Consequently,
\[
    s_n\ge s_{n_0}+\frac{\kappa}{8}(n-n_0)
    \qquad \forall n\ge n_0.
\]
Therefore,
\[
    \beta_n
    =\frac{s_n^2-L^2}{C}
    \ge
    \frac1C\left[
        \left(s_{n_0}+\frac{\kappa}{8}(n-n_0)\right)^2-L^2
    \right].
\]
Let
\[
    N_0:=\max\left\{2n_0,\left\lceil\frac{23L}{\kappa}\right\rceil\right\}.
\]
For every $n\ge N_0$, we have $n-n_0\ge \frac{n}{2}$ and
$L^2\le \frac{\kappa^2 n^2}{512}$. Hence
\[
\begin{aligned}
    \left(s_{n_0}+\frac{\kappa}{8}(n-n_0)\right)^2-L^2
    &\ge
    s_{n_0}^2+\frac{\kappa}{8}s_{n_0}n
    +\frac{\kappa^2}{256}n^2-\frac{\kappa^2}{512}n^2                         \\
    &=
    s_{n_0}^2+\frac{\kappa}{8}s_{n_0}n+\frac{\kappa^2}{512}n^2.
\end{aligned}
\]
Thus
\[
    \beta_n
    \ge
    \frac{s_{n_0}^2}{C}
    +\frac{\kappa s_{n_0}}{8C}n
    +\frac{\kappa^2}{512C}n^2
    \qquad \forall n\ge N_0.
\]
In particular, we have
\[
    \beta_n\ge c_0n^2,
    \quad \forall n\ge N_0,
\]
where $ c_0:=\frac{\kappa^2}{512C}$.
Therefore, the proof is complete.
\end{proof}


\begin{lemma}\label{lem:beta_tau_linear}
Let $\{(\beta_n,\zeta_n,\tau_n)\}$ be generated by
Algorithm~\ref{algorithm 2 acclerated}. Then there exist constants
$c_1>0$ and $N_1\in\mathbb N$ such that
\[
    \beta_n\tau_n\ge c_1 n
    \qquad \forall n\ge N_1.
\]
More precisely, one can take
\[
    c_1:=\bar M\sqrt{\frac{c_0}{2C}}
    ~~\text{and}~~
    N_1:=
    \max\left\{
        N_0,
        \left\lceil\frac{L}{\sqrt{c_0C}}\right\rceil
    \right\},
\]
where $N_0$ and $c_0$ are given in
Lemma~\ref{lemma:acc_beta_tau_zeta}.
\end{lemma}

\begin{proof}
By Lemma~\ref{lemma:acc_beta_tau_zeta}\ref{it:acc_tau_lower},
\[
    \tau_n\ge\frac{\bar M}{\sqrt{L^2+C\beta_n}}
    \qquad \forall n\ge1.
\]
Thus
\[
    \beta_n\tau_n
    \ge
    \frac{\bar M\beta_n}{\sqrt{L^2+C\beta_n}}
    \qquad \forall n\ge1.
\]
Let $n\ge N_1$. Since $n\ge N_0$,
Lemma~\ref{lemma:acc_beta_tau_zeta}\ref{it:acc_beta_quad} gives
$\beta_n\ge c_0n^2$. Moreover, by the definition of $N_1$, we have
\[
    L^2\le c_0Cn^2\le C\beta_n.
\]
Therefore $L^2+C\beta_n\le2C\beta_n$, and hence
\[
\begin{aligned}
    \beta_n\tau_n
    &\ge
    \frac{\bar M\beta_n}{\sqrt{2C\beta_n}}
    =\bar M\sqrt{\frac{\beta_n}{2C}}
    \ge
    \bar M\sqrt{\frac{c_0}{2C}}\,n
    =c_1n.
\end{aligned}
\]
\end{proof}

\medskip
Given $N\ge1$, define
\begin{equation*}
    Q_N:=\sum_{n=1}^N\beta_n\tau_n,
    \qquad
    \widetilde x_N:=\frac1{Q_N}\sum_{n=1}^N\beta_n\tau_nx_n,
    \qquad
    \widetilde w_N:=\frac1{Q_N}\sum_{n=1}^N\beta_n\tau_nw_n.
\end{equation*}

\begin{theorem}\label{Accelerated sublinear rate 2}
Under Assumptions~\ref{assumption1}, \ref{assumption_5.1} and~\ref{assump:strong f}, let
$\{(z_n,x_n,w_n,y_n,\tau_n)\}$ be generated by
Algorithm~\ref{algorithm 2 acclerated}, and let
$(\bar x,\bar w,\bar y)\in\mathbf\Pi$. Then there exist constants
$c_1>0$ and $P_4>0$ such that, for every $N\ge1$,
\begin{equation}\label{eq:rate_final_thm}
    \bigl|\Phi(\widetilde x_N,\widetilde w_N)-\Phi(\bar x,\bar w)\bigr|
    \le
    \frac{P_4}{c_1(N^2+N)},
    \qquad
    \|K\widetilde x_N-\widetilde w_N\|
    \le
    \frac{P_4}{b\,c_1(N^2+N)},
\end{equation}
where $b>0$ is any constant satisfying $b\ge2\|\bar y\|$.
\end{theorem}

\begin{proof}
Apply Lemma~\ref{lem:acc_weighted_fejer} with $x^*=\bar x$ and arbitrary
$y\in\mathbb H_2$. Summing \eqref{eq:acc_weighted_descent} from $n=1$ to $N$
gives
\begin{equation*}
    \sum_{n=1}^N\beta_n\tau_n\mathbb J(x_n,w_n,y)
    \le
    \beta_1E_1(y)-\beta_{N+1}E_{N+1}(y)
    \le
    \beta_1E_1(y).
\end{equation*}
Since $\mathbb J(\cdot,\cdot,y)$ is convex in $(x,w)$, the definition of
$(\widetilde x_N,\widetilde w_N)$ gives
\begin{equation}\label{eq:acc_ergodic_J}
    \mathbb J(\widetilde x_N,\widetilde w_N,y)
    \le
    \frac1{Q_N}\sum_{n=1}^N\beta_n\tau_n\mathbb J(x_n,w_n,y)
    \le
    \frac{\beta_1E_1(y)}{Q_N}.
\end{equation}
Fix $b>0$ with $b\ge2\|\bar y\|$, and set
\[
    P_3:=
    \frac{\beta_1(\psi+\mu\tau_1)}{2(\psi-1)}\|\bar x-z_2\|^2
    +\frac12(b+\|y_0\|)^2
    +\frac{\alpha\beta_1\theta_0}{2}\|x_1-x_0\|^2 .
\]
For every $y$ with $\|y\|\le b$, we have $\beta_1E_1(y)\le P_3$. Taking the
supremum of \eqref{eq:acc_ergodic_J} over $\|y\|\le b$ yields
\begin{equation}\label{eq:acc_gap_upper_ball}
    \Phi(\widetilde x_N,\widetilde w_N)-\Phi(\bar x,\bar w)
    +b\|K\widetilde x_N-\widetilde w_N\|
    \le
    \frac{P_3}{Q_N}.
\end{equation}
Since $(\bar x,\bar w,\bar y)$ is a saddle point and $K\bar x=\bar w$
\begin{equation}\label{eq:acc_gap_lower_saddle}
    \Phi(\bar x,\bar w)-\Phi(\widetilde x_N,\widetilde w_N)
    \le
    \langle \bar y,K\widetilde x_N-\widetilde w_N\rangle
    \le
    \frac b2\|K\widetilde x_N-\widetilde w_N\|.
\end{equation}
Combining \eqref{eq:acc_gap_upper_ball} and \eqref{eq:acc_gap_lower_saddle} gives
\begin{equation}\label{eq:acc_gap_feas_over_Q}
    \|K\widetilde x_N-\widetilde w_N\|
    \le
    \frac{2P_3}{bQ_N},
    \qquad
    \bigl|\Phi(\widetilde x_N,\widetilde w_N)-\Phi(\bar x,\bar w)\bigr|
    \le
    \frac{P_3}{Q_N}.
\end{equation}
By Lemma~\ref{lem:beta_tau_linear}, there exist $c_1>0$ and
$N_1\in\mathbb N$ such that $\beta_n\tau_n\ge c_1n$ for all $n\ge N_1$.
Hence, for $N\ge N_1$,
\[
    Q_N
    \ge
    c_1\sum_{n=N_1}^N n
    =
    \frac{c_1}{2}\bigl(N(N+1)-(N_1-1)N_1\bigr).
\]
In particular, for every $N\ge2N_1$
\begin{equation}\label{eq:acc_Q_lower}
    Q_N\ge\frac{c_1}{4}(N^2+N).
\end{equation}
Combining \eqref{eq:acc_gap_feas_over_Q} and \eqref{eq:acc_Q_lower} gives, for
all $N\ge2N_1$
\[
    \bigl|\Phi(\widetilde x_N,\widetilde w_N)-\Phi(\bar x,\bar w)\bigr|
    \le
    \frac{4P_3}{c_1(N^2+N)},
\]
and
\[
    \|K\widetilde x_N-\widetilde w_N\|
    \le
    \frac{8P_3}{b\,c_1(N^2+N)}.
\]
Finally, enlarge the constant to handle the finite set
$\{1,2,\ldots,2N_1-1\}$. Take
\[
\begin{aligned}
    P_4:=\max\Bigl\{&
    8P_3,
    \max_{1\le N\le2N_1-1}
    c_1(N^2+N)
    \bigl|\Phi(\widetilde x_N,\widetilde w_N)-\Phi(\bar x,\bar w)\bigr|,
    \\
    &\max_{1\le N\le2N_1-1}
    b\,c_1(N^2+N)\|K\widetilde x_N-\widetilde w_N\|
    \Bigr\}.
\end{aligned}
\]
Then \eqref{eq:rate_final_thm} holds for every $N\ge1$.
\end{proof}

\subsection[Strongly convex differentiable term]{When $h$ is strongly convex}\label{subsec:hstrong_acc}

In this subsection, we show that PF-GRPDA (Algorithm \ref{algorithm 2}) also
admits an accelerated ergodic rate when the locally smooth term is strongly convex.  This case is slightly different from the acceleration obtained under the strong convexity of $f$.  Indeed, the strong convexity of $f$ produces
terms involving $x_{n+1}$, while the strong convexity of $h$ gives a curvature
term at $x_n$.  Thus, the proof has to use the golden-ratio relation between
$x_n$, $z_{n+1}$ and $z_{n+2}$ in a more effective way.  The result below is
stated for a globally strongly convex and locally smooth $h$. However, when $h$ is globally smooth, the local Lipschitz constant appearing in the proof can simply be replaced by the
global Lipschitz constant of $\nabla h$.

\begin{assumption}\label{assump:hstrong_acc}
Suppose that $h$ is globally $\mu_h$-strongly convex for some $\mu_h>0$, that is,
\begin{equation}\label{eq:hstrong_global}
    h(v)\ge h(u)+\langle \nabla h(u),v-u\rangle
    +\frac{\mu_h}{2}\|v-u\|^2,
    \qquad \forall u,v\in\mathbb H_1 .
\end{equation}
\end{assumption}

\begin{algorithm}[H]
\caption{Accelerated PF-GRPDA when $h$ is strongly convex}
\label{algorithm:hstrong_acc}
\KwIn{Choose $x_0\in\mathbb H_1$, $y_0\in\mathbb H_2$ and set $z_0=x_0$. Choose $\psi\in(1,\varphi)$, $\rho=\psi^{-1}+\psi^{-2}$, $0<\alpha< \frac{1}{3}$, $\beta_0>0$, and $\theta_0=1$.  Take  $C:=\psi\|K\|^2,$  $\Theta:=\max\{\theta_0,\psi\rho\}$, and $U:=\sqrt{\frac{\rho\psi\alpha^2\Theta}{C\beta_0}}$. Choose $0<\tau_0\le U$, and $\gamma>0$ such that $0\le\gamma\le \bar\gamma,$ where
\begin{equation}\label{eq:hstrong_gamma_bar}
    \bar\gamma:=
    \min\left\{
    \frac{\mu_h(\psi-1)}{3\psi},\,
    \frac{\psi^2(\psi-1)}{3U},\,
    \frac{(1-3\alpha)\psi}
    {U\left(\frac{3(\psi-1)}{\psi}+\alpha\Theta\right)}
    \right\}.
\end{equation}}

\For{$n=1,2,\ldots$}{
\textbf{Step 1} (Compute)
\begin{align}
    z_n
    &=\frac{\psi-1}{\psi}x_{n-1}+\frac1\psi z_{n-1},\nonumber\\
    x_n
    &=\prox_{\tau_{n-1}f}
    \left(z_n-\tau_{n-1}K^*y_{n-1}
    -\tau_{n-1}\nabla h(x_{n-1})\right).\nonumber
\end{align}

\textbf{Step 2} (Update)
\begin{align}
\beta_n&=\beta_{n-1}(1+\gamma\tau_{n-1}). \label{eq:hstrong_beta_update}\\
    \tau_n
    &=\min\left\{
    \rho\tau_{n-1},
    \frac{\psi\alpha^2\theta_{n-1}}
    {(L_n^2+\beta_n\psi\|K\|^2)\tau_{n-1}}
    \right\},
    \qquad
    \sigma_n=\beta_n\tau_n . \label{eq:hstrong_tau_update}
\end{align}

\textbf{Step 3} (Compute)
\begin{align}
    w_n
    &=\prox_{\frac1{\sigma_n}g}
    \left(\frac{y_{n-1}}{\sigma_n}+Kx_n\right),\nonumber\\
    y_n
    &=y_{n-1}+\sigma_n(Kx_n-w_n).\nonumber
\end{align}

\textbf{Step 4} (Update)
\begin{align*}
    \theta_n=\frac{\psi\tau_n}{\tau_{n-1}}.
\end{align*}
}
\end{algorithm}

 Some comments regarding Algorithm \ref{algorithm:hstrong_acc} are in order.
\begin{remark}
    Note that the additional update~\eqref{eq:hstrong_beta_update} is the main difference from
Algorithm~\ref{algorithm 2}.  It is also different from PF-AGRPDA, because here the growth of $(\beta_n)$ is driven
by a fixed curvature parameter $\gamma$, chosen below the threshold
\eqref{eq:hstrong_gamma_bar}. In contrast to Algorithm~\ref{algorithm 2 acclerated}, the restriction $\psi>\psi_0$ is not needed in the present case. It is enough to take $\psi\in(1,\varphi)$.
\end{remark}
\begin{remark}
     It is important to observe that, in the absence of strong convexity of $h$, that is, when $\mu_h=0$, we have $\bar\gamma=0$ and consequently $\gamma=0$. As a result, the update rule \eqref{eq:hstrong_beta_update} becomes $\beta_n=\beta_0$ for all $n$. Hence, the Algorithm \ref{algorithm:hstrong_acc} simplifies to Algorithm~\ref{algorithm 2}. Otherwise, in all cases, $\gamma>0$.
\end{remark}

\begin{lemma}\label{lem:hstrong_basic_bounds}
Under Assumptions~\ref{assumption1}, \ref{assumption_5.1} and \ref{assump:hstrong_acc}, let
$\{(z_n,x_n,w_n,y_n,\tau_n,\beta_n)\}$ be generated by
Algorithm~\ref{algorithm:hstrong_acc}.  Then the following holds.
\begin{enumerate}[label=\textup{(\roman*)}]
    \item For all $n\ge 1$, $\tau_n\le U$ and  $\theta_n\le\Theta$. Furthermore, we can prove that 
\begin{equation*}
        \tau_n^2\beta_n\le \widehat B:=\frac{\alpha^2\Theta^2}{C}~~\forall n\ge 1.
    \end{equation*}
\end{enumerate}
\end{lemma}

\begin{proof}
The proof follows from similar arguments as in Lemma \ref{lemma:acc_beta_tau_zeta_1} and Remark \ref{tau_n_bdd_above}.
\end{proof}

\begin{lemma}\label{lem:hstrong_one_step}
Under Assumptions~\ref{assumption1}, \ref{assumption_5.1} and
\ref{assump:hstrong_acc}, let the sequence $\{(z_n,x_n,w_n,y_n,\tau_n)\}$ be generated by Algorithm~\ref{algorithm:hstrong_acc}.  Let
$(x^*,w^*,y^*)\in\mathbf\Pi$ be a saddle point of $\mathbb{L}$.  Then, for every
$y\in\mathbb H_2$ and every $n\ge1$, one has
\begin{multline}\label{eq:hstrong_one_step}
 2\tau_n\mathbb J(x_n,w_n,y)
 +\frac{\psi}{\psi-1}\|x^*-z_{n+2}\|^2
 +\frac1{\beta_n}\|y-y_n\|^2
 +\alpha\theta_n\|x_{n+1}-x_n\|^2                                      \\
 +\mu_h\tau_n\|x_n-x^*\|^2
 +\theta_n\|x_n-z_{n+1}\|^2
 +(1-3\alpha)\theta_n\|x_{n+1}-x_n\|^2                                  \\
 \le \frac{\psi}{\psi-1}\|x^*-z_{n+1}\|^2
 +\frac1{\beta_n}\|y-y_{n-1}\|^2
 +\alpha\theta_{n-1}\|x_n-x_{n-1}\|^2 .
\end{multline}
\end{lemma}

\begin{proof}
The proof is the same as that of the basic one-step inequality for Algorithm~\ref{algorithm 2},
except that the convexity estimate for $h$ is replaced by its strong convexity estimate. Indeed, \eqref{eq:hstrong_global} gives
\begin{equation}\label{eq:hstrong_convexity_insert}
    h(x_n)-h(x^*)
    \le
    \langle\nabla h(x_n),x_n-x^*\rangle
    -\frac{\mu_h}{2}\|x_n-x^*\|^2 .
\end{equation}
Using \eqref{eq:hstrong_convexity_insert} in the Lemma \ref{lemma_5.0.1} yields
\begin{multline*}
    2\tau_n\mathbb J(x_n,w_n,y)
    +\mu_h\tau_n\|x_n-x^*\|^2
    +\frac{\psi}{\psi-1}\|z_{n+2}-x^*\|^2 +\frac{1}{\beta_n}\|y_n-y\|^2
    +\theta_n\|x_n-z_{n+1}\|^2 \\
    {}+(1-2\alpha)\theta_n\|x_{n+1}-x_n\|^2
    \le
    \frac{\psi}{\psi-1}\|z_{n+1}-x^*\|^2
    +\frac{1}{\beta_n}\|y_{n-1}-y\|^2
    +\alpha\theta_{n-1}\|x_n-x_{n-1}\|^2 .
\end{multline*}
Finally, since $(1-2\alpha)\theta_n
    =\alpha\theta_n+(1-3\alpha)\theta_n,$ we obtain  \eqref{eq:hstrong_one_step}.
\end{proof}

\begin{lemma}\label{lem:hstrong_weighted_fejer}
Under the assumptions of Lemma~\ref{lem:hstrong_one_step}, for
$n\ge1$ and $y\in\mathbb H_2$, define
\begin{equation}\label{eq:hstrong_energy}
    E_n(y):=
    \frac{\psi}{2(\psi-1)}\|x^*-z_{n+1}\|^2
    +\frac1{2\beta_n}\|y-y_{n-1}\|^2
    +\frac{\alpha\theta_{n-1}}2\|x_n-x_{n-1}\|^2 .
\end{equation}
Then, for every $y\in\mathbb H_2$ and every $n\ge1$,
\begin{equation}\label{eq:hstrong_weighted_descent}
    \beta_n\tau_n\mathbb J(x_n,w_n,y)+\beta_{n+1}E_{n+1}(y)
    \le \beta_nE_n(y).
\end{equation}
Consequently, the sequences $(z_n)$, $(x_n)$ and $(y_n)$ are bounded.
\end{lemma}
\begin{proof}
Multiplying \eqref{eq:hstrong_one_step} by $\frac{\beta_n}{2}$, we get
\begin{align}
&\beta_n\tau_n\mathbb J(x_n,w_n,y)
+\frac{\beta_n\psi}{2(\psi-1)}\|x^*-z_{n+2}\|^2
+\frac12\|y-y_n\|^2
+\frac{\alpha\beta_n\theta_n}{2}\|x_{n+1}-x_n\|^2                  \notag\\
&\quad
+\frac{\beta_n}{2}\Big(
\mu_h\tau_n\|x_n-x^*\|^2
+\theta_n\|x_n-z_{n+1}\|^2
+(1-3\alpha)\theta_n\|x_{n+1}-x_n\|^2
\Big)
\le \beta_nE_n(y).                                                \label{eq:hstrong_exact_before_absorb}
\end{align}
On the other hand, using \eqref{eq:hstrong_beta_update}, we have
\begin{align}
\beta_{n+1}E_{n+1}(y)
&=\frac{\beta_n\psi}{2(\psi-1)}\|x^*-z_{n+2}\|^2
+\frac12\|y-y_n\|^2
+\frac{\alpha\beta_n\theta_n}{2}\|x_{n+1}-x_n\|^2                   \notag\\
&\quad
+\frac{\beta_n\gamma\tau_n}{2}
\left(
 \frac{\psi}{\psi-1}\|x^*-z_{n+2}\|^2
 +\alpha\theta_n\|x_{n+1}-x_n\|^2
\right).                                                         \label{eq:hstrong_energy_decomp}
\end{align}
To control the last two terms in
\eqref{eq:hstrong_energy_decomp}, we have to use \eqref{eqn:39a}. Observe that
\[
    z_{n+2}-x^*
    =x_n-x^*+\frac{\psi-1}{\psi}(x_{n+1}-x_n)
    -\frac1\psi(x_n-z_{n+1}).
\]
Therefore,
\begin{align*}
\|x^*-z_{n+2}\|^2
&\le
3\|x_n-x^*\|^2
+3\left(\frac{\psi-1}{\psi}\right)^2\|x_{n+1}-x_n\|^2
+\frac3{\psi^2}\|x_n-z_{n+1}\|^2.
\end{align*}
Hence, it follows that
\begin{align}
\frac{\psi}{\psi-1}\|x^*-z_{n+2}\|^2
+\alpha\theta_n\|x_{n+1}-x_n\|^2       
&\le
\frac{3\psi}{\psi-1}\|x_n-x^*\|^2
+\frac{3}{\psi(\psi-1)}\|x_n-z_{n+1}\|^2                       \notag\\
&\quad+
\left(
3\left(\frac{\psi-1}{\psi}\right)+\alpha\theta_n
\right)\|x_{n+1}-x_n\|^2 .                                     \label{eq:hstrong_absorb_expand}
\end{align}
Since $\frac{\tau_n}{\theta_n}=\frac{\tau_{n-1}}{\psi}
    \le \frac{U}{\psi}, ~~\forall n\ge 1$, the three bounds in \eqref{eq:hstrong_gamma_bar} yield
\begin{align*}
\frac{3\psi}{\psi-1}\gamma\tau_n\|x_n-x^*\|^2
&\le \mu_h\tau_n\|x_n-x^*\|^2,\\
\frac{3}{\psi(\psi-1)}\gamma\tau_n\|x_n-z_{n+1}\|^2
&\le \theta_n\|x_n-z_{n+1}\|^2,\\
\gamma\tau_n
\left(
3\left(\frac{\psi-1}{\psi}\right)+\alpha\theta_n
\right)\|x_{n+1}-x_n\|^2
&\le (1-3\alpha)\theta_n\|x_{n+1}-x_n\|^2.
\end{align*}
Combining these estimates with \eqref{eq:hstrong_absorb_expand}, we obtain
\begin{align}
&\gamma\tau_n
\left(\frac{\psi}{\psi-1}\|x^*-z_{n+2}\|^2
 +\alpha\theta_n\|x_{n+1}-x_n\|^2
\right)                                                        \notag\\
&\le
\mu_h\tau_n\|x_n-x^*\|^2
+\theta_n\|x_n-z_{n+1}\|^2
+(1-3\alpha)\theta_n\|x_{n+1}-x_n\|^2 .                         \label{eq:hstrong_absorb_final}
\end{align}
Now \eqref{eq:hstrong_weighted_descent} follows from
\eqref{eq:hstrong_exact_before_absorb}, \eqref{eq:hstrong_energy_decomp} and
\eqref{eq:hstrong_absorb_final}. Taking $y=y^*$ in \eqref{eq:hstrong_weighted_descent} and using
\eqref{Eq:measure}, we obtain
\[
    \beta_{n+1}E_{n+1}(y^*)
    \le \beta_1E_1(y^*).
\]
Since $\beta_n\ge\beta_0>0$, it follows that $(z_n)$ is bounded, and following \eqref{eqn:39a}, gives the boundedness of $(x_n)$.  Finally,
\[
    \frac12\|y^*-y_n\|^2
    \le \beta_{n+1}E_{n+1}(y^*)
    \le \beta_1E_1(y^*)
\]
shows that $(y_n)$ is bounded.
\end{proof}

\begin{theorem}\label{thm:hstrong_acc_rate}
Under Assumptions~\ref{assumption1}, \ref{assumption_5.1} and \ref{assump:hstrong_acc}, let
$\{(z_n,x_n,w_n,y_n,\tau_n,\beta_n)\}$ be generated by
Algorithm~\ref{algorithm:hstrong_acc}, and let
$(\bar x,\bar w,\bar y)\in\mathbf\Pi$.  Then there exist constants
$c_1>0$ and $P_4>0$ such that, for every $N\ge1$,
\begin{equation*}
    \left|
    \Phi(\widetilde x_N,\widetilde w_N)-\Phi(\bar x,\bar w)
    \right|
    \le
    \frac{P_4}{c_1(N^2+N)},
    \qquad
    \|K\widetilde x_N-\widetilde w_N\|
    \le
    \frac{P_4}{b\,c_1(N^2+N)},
\end{equation*}
where $b>0$ is any constant satisfying $b\ge2\|\bar y\|$.
\end{theorem}

\begin{proof}
The proof follows an analogous argument as in Remark \ref{eq:acc_bdd_x_n}, Lemma \ref{lemma:acc_beta_tau_zeta}, Lemma \ref{lem:beta_tau_linear}  and Theorem~\ref{Accelerated sublinear rate 2}. Because of its simplicity, we omit the details.
\end{proof}

\section{Numerical results}\label{Num_section}

In this section, we illustrate the behaviour of the proposed primal--dual schemes on a Poisson inverse imaging problem.  We compare Algorithm~\ref{algorithm 2} (PF-GRPDA), its accelerated variant Algorithm~\ref{algorithm 2 acclerated} (PF-AGRPDA), and the linesearch method aPDAc-L~\cite{chang2026convex}.  We also include aGRAAL~\cite{malitsky2020golden} and adaPDM~\cite{latafat2023adaptive} as representative adaptive first-order competitors.

When the true solution is not available, the plotted objective gaps are computed with respect to a common reference value
\[
    F^*:=\inf\{F(x_n):\ (x_n) \text{ is generated by any method in the reference runs}\}.
\]
For the methods that generate the split variables $(x_n,w_n)$, namely PF-GRPDA, PF-GRPDA-ad$\beta$ (PFGRPDA with adaptive $\beta_n$), PF-AGRPDA and aPDAc-L, we also report the objective residual $|\Phi(x_n,w_n)-\Phi^*|$ and the feasibility violation $\|Kx_n-w_n\|$.  For aGRAAL and adaPDM, whose implemented iterates are not analysed here in the same proximal-split form, we restrict the comparison to the primal objective gap ($F(x_n)-F^*$), PSNR, step-sizes, and CPU time.  This distinction is important; one would be comparing residuals that are not defined in the same variables. The numerical performance of primal--dual methods is sensitive not only to the primal and dual step-sizes $\tau_n$ and $\sigma_n$, but also to the ratio $\beta_n=\sigma_n/\tau_n$; see, for example, \cite{chang2022grpdarevisited,sun2015convergent,chambolle2011first,malitsky2018first}.  To test whether this ratio can be tuned automatically, we also run a heuristic adaptive-$\beta_n$ version of PF-GRPDA.  Given the residual ratio $r_n=\mathbf{pinf}_n/\mathbf{dinf}_n$, we update

\begin{equation*}
    \beta_{n+1}=  \begin{cases}
  0.5\beta_n, & \text{if } r_n\leq 0.5,\\
  \beta_n, & \text{if } r_n\in(0.5,1.2),\\
  1.2\beta_n, & \text{if } r_n\geq 1.2,
\end{cases}
\end{equation*}
where
\begin{align*}
    \mathbf{pinf}_n
    &=\|Kx_n-w_n\|_1,\\
    \mathbf{dinf}_n
    &=\operatorname{dist}_1\big(-K^*y_n-\nabla h(x_n),\partial f(x_n)\big)+\operatorname{dist}_1\big(y_n,\partial g(w_n)\big).
\end{align*}
Here $\operatorname{dist}_1(u,S')$ is the distance between a set $S'$ and the vector $u$ measured by $\ell_1$-norm. Again, from~\eqref{eq:w_n update}, we have 
\begin{equation*}
    \frac{y_{n-1}}{\sigma_n} + Kx_n - w_n\in \partial\Big( \frac{g}{\sigma_n}\Big)(w_n)\implies y_{n-1}+\sigma_n(Kx_n-w_n)\in\partial g(w_n).
\end{equation*}
Now combining this with \eqref{eq:y_n update} gives $y_n\in\partial g(w_n)$. Thus $\operatorname{dist}(y_n,\partial g(w_n))\equiv0$. We emphasise that this adaptive-$\beta_n$ rule is used only as a heuristic.  The convergence results proved in this paper do not cover this additional update of $\beta_n$.


\subsection{Poisson inverse problem}\label{subsec:poisson}

Let $x\in\mathbb{R}^n$ represent the unknown image and let
$y\in\mathbb{R}^m$ denote the measured data. We model the observations using
Poisson statistics. Specifically, the components $y_j$ are independent
realisations of random variables $Y_j$ satisfying
\[
    Y_j\sim\operatorname{Poisson}\big((Ax+b)_j\big),
    \qquad j=1,\ldots,m,
\]
or, equivalently,
\[
    \mathbb{P}(Y_j=y_j)
    =
    \frac{\exp\!\big(-(Ax+b)_j\big)(Ax+b)_j^{\,y_j}}{y_j!},
    \qquad j=1,\ldots,m,
\]
where $A\in\mathbb{R}^{m\times n}$ is the observation matrix and
$b\in\mathbb{R}^m$ represents the background noise.  Our objective is
to reconstruct a nonnegative image $x$ from the resulting Poisson data. We examine three formulations of this inverse problem. \textbf{Setting~1} considers the convex KL--TV reconstruction model. In \textbf{Setting~2}, a quadratic Tikhonov term is
added to the primal function, making this component strongly convex and
allowing us to investigate the practical effect of the acceleration used by
PF-AGRPDA. \textbf{Setting~3} retains the same strongly convex reconstruction objective
as \textbf{Setting~2} but applies a different splitting: the quadratic term is included
in the differentiable component $h$. This formulation is used to evaluate
the accelerated method in Algorithm~\ref{algorithm:hstrong_acc}. 

The same Poisson sampling process and initialisation are used throughout the
experiments. For each setting, both motion and defocus blur have been used as the observation mechanism. For the
motion-blur experiments, we use a kernel with angle $30^\circ$ and length
$3$ in \textbf{Setting~1}, increasing the length to $5$ in \textbf{Settings~2} and~\textbf{3}. For
the defocus experiments, the kernel radius is $1.25$ in \textbf{Setting~1} and $2.5$
in \textbf{Settings~2} and~\textbf{3}. These two operators introduce different forms of
ill-conditioning. Six standard greyscale images with different spatial dimensions were used in these experiments. The \texttt{cameraman} image has a resolution of
$256\times256$ pixels, while \texttt{columbia} has a resolution of
$480\times480$ pixels. The \texttt{boats} and \texttt{goldhill} images both
have a resolution of $720\times576$ pixels. The \texttt{clock} image and the \texttt{chemical plant} image each have a
resolution of $256\times256$ pixels. All six images were used at their
original resolutions, without resizing or altering their aspect
ratios. The \texttt{cameraman}, \texttt{columbia}, \texttt{boats}, and
\texttt{goldhill} images were obtained from
\href{https://www.dip.ee.uct.ac.za/imageproc/stdimages/greyscale/}
{https://www.dip.ee.uct.ac.za/imageproc/stdimages/greyscale/}, whereas the
\texttt{clock} and \texttt{chemical plant} images were obtained from 
\href{https://sipi.usc.edu/database/database.php?volume=misc}
{https://sipi.usc.edu/database/database.php?volume=misc}. In what follows, we describe each setting below.

\subsubsection{Setting 1}
By the maximum-likelihood principle, the first reconstruction model \cite{di2020acquire} is

\begin{equation*}
    \min_{x\ge 0} F(x):=\mathrm{KL}(Ax,y)+\lambda\|\nabla x\|_{2,1},
\end{equation*}

where

\begin{equation*}
    \mathrm{KL}(s,y)=\sum_{i=1}^{m}\left[y_i\log\!\left(\frac{y_i}{s_i}\right)+s_i-y_i\right]\qquad (s>0),
\end{equation*}
and $\|\nabla x\|_{2,1}$ denotes the isotropic total-variation seminorm.  This is a special case of~\eqref{main_prob} with

\[
    f(x)=\iota_{\{x\ge 0\}}(x),\qquad
    g(w)=\lambda\|w\|_{2,1},\qquad
    K=\nabla,\qquad
    h(x)=\mathrm{KL}(Ax,y).
\]

\begin{figure}[htbp]
\centering
\subfloat[PF-GRPDA-ad$\beta$]{
  \includegraphics[width=0.20\linewidth]{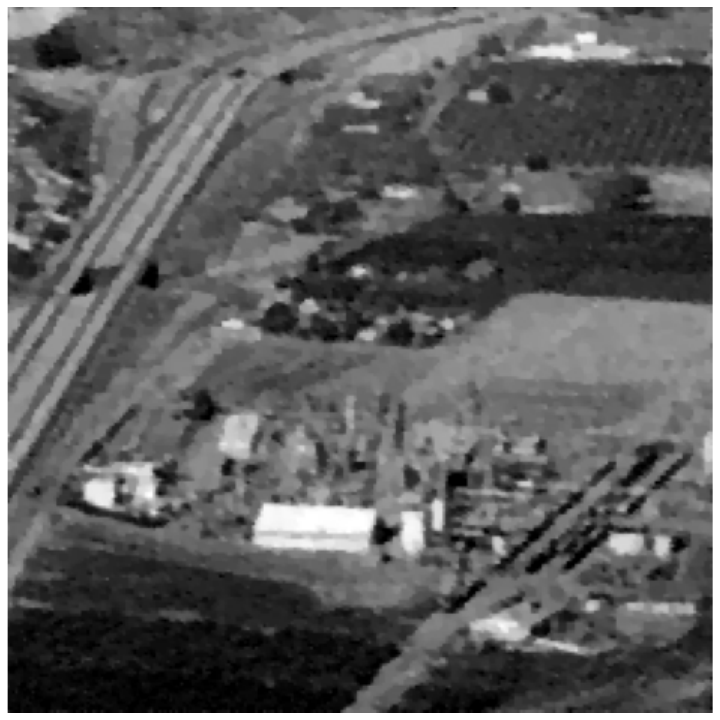}
}\hfill
\subfloat[PF-GRPDA]{
  \includegraphics[width=0.20\linewidth]{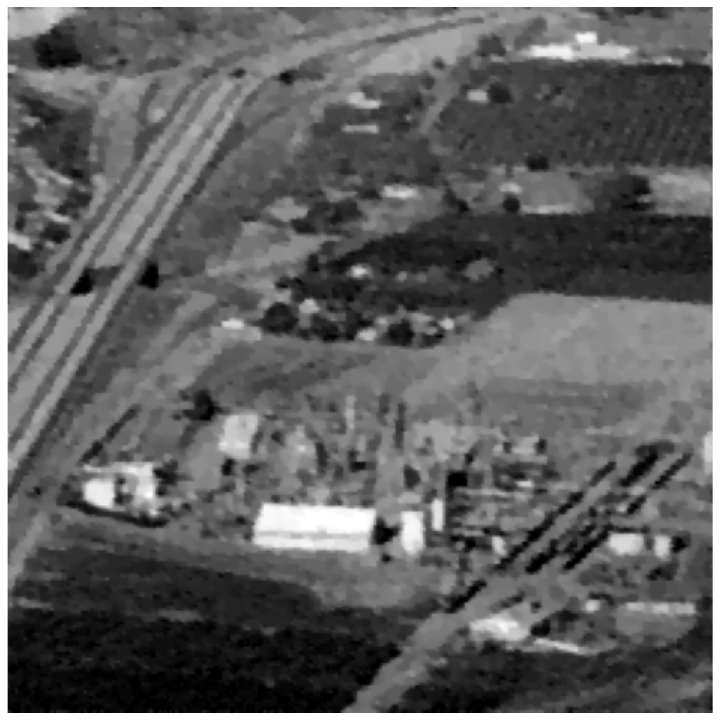}
}\hfill
\subfloat[aGRAAL]{
  \includegraphics[width=0.20\linewidth]{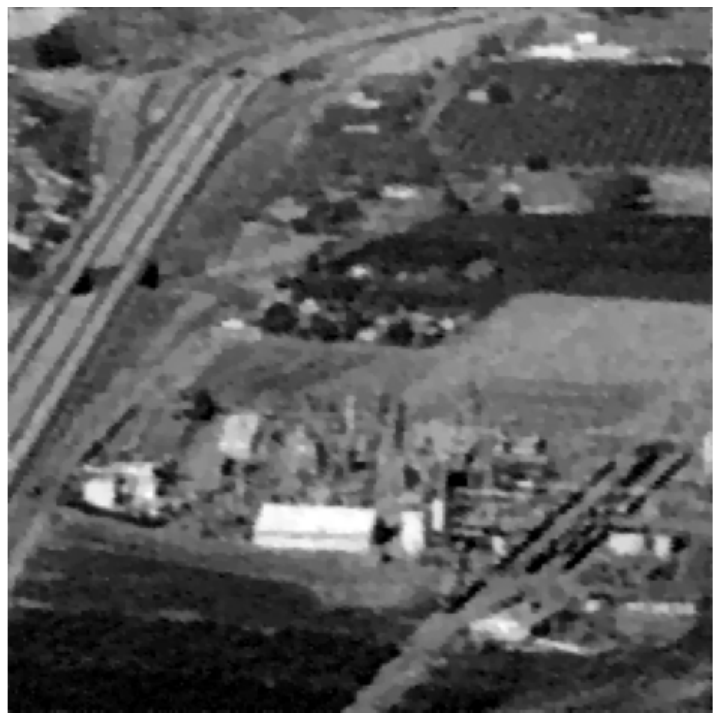}
}\hfill
\subfloat[adaPDM]{
  \includegraphics[width=0.20\linewidth]{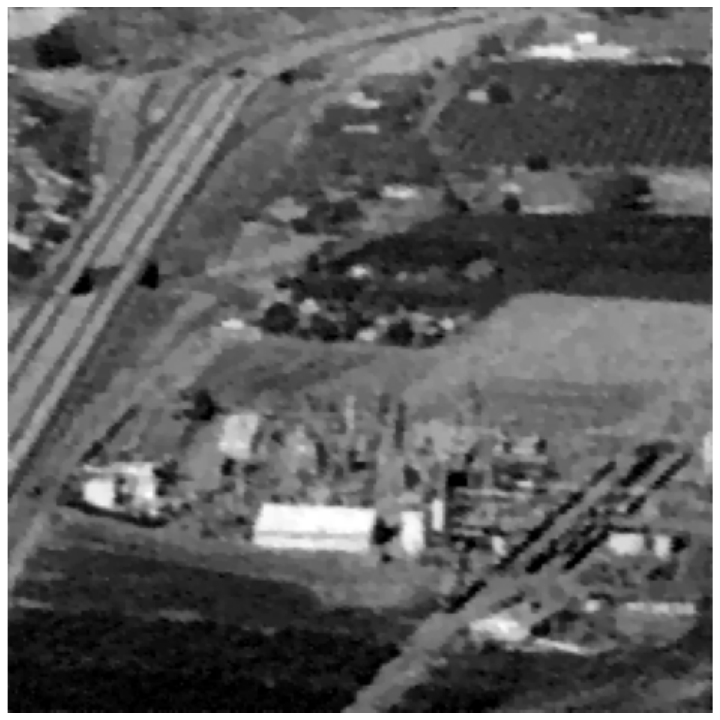}
}

\caption{Image reconstruction results for \textbf{Setting~1} using the \texttt{Chemical plant} image degraded by motion blur}
\label{fig:setting_1_motion_blur_recovery_results}
\end{figure}
\begin{figure}[htbp]
\centering
\subfloat[Feasibility residual]{
  \includegraphics[width=0.27\linewidth]{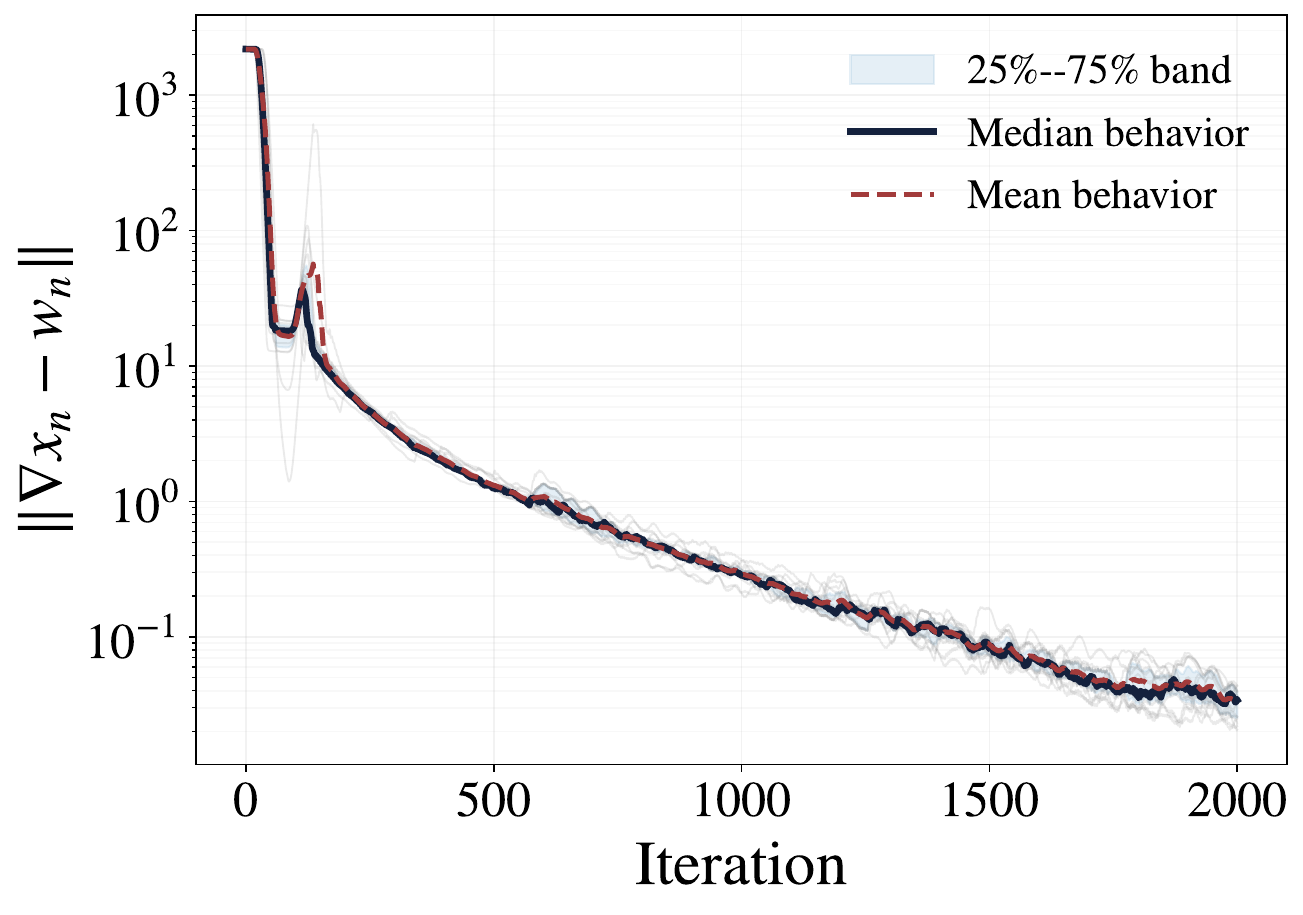}
}\hfill
\subfloat[Objective residual]{
  \includegraphics[width=0.27\linewidth]{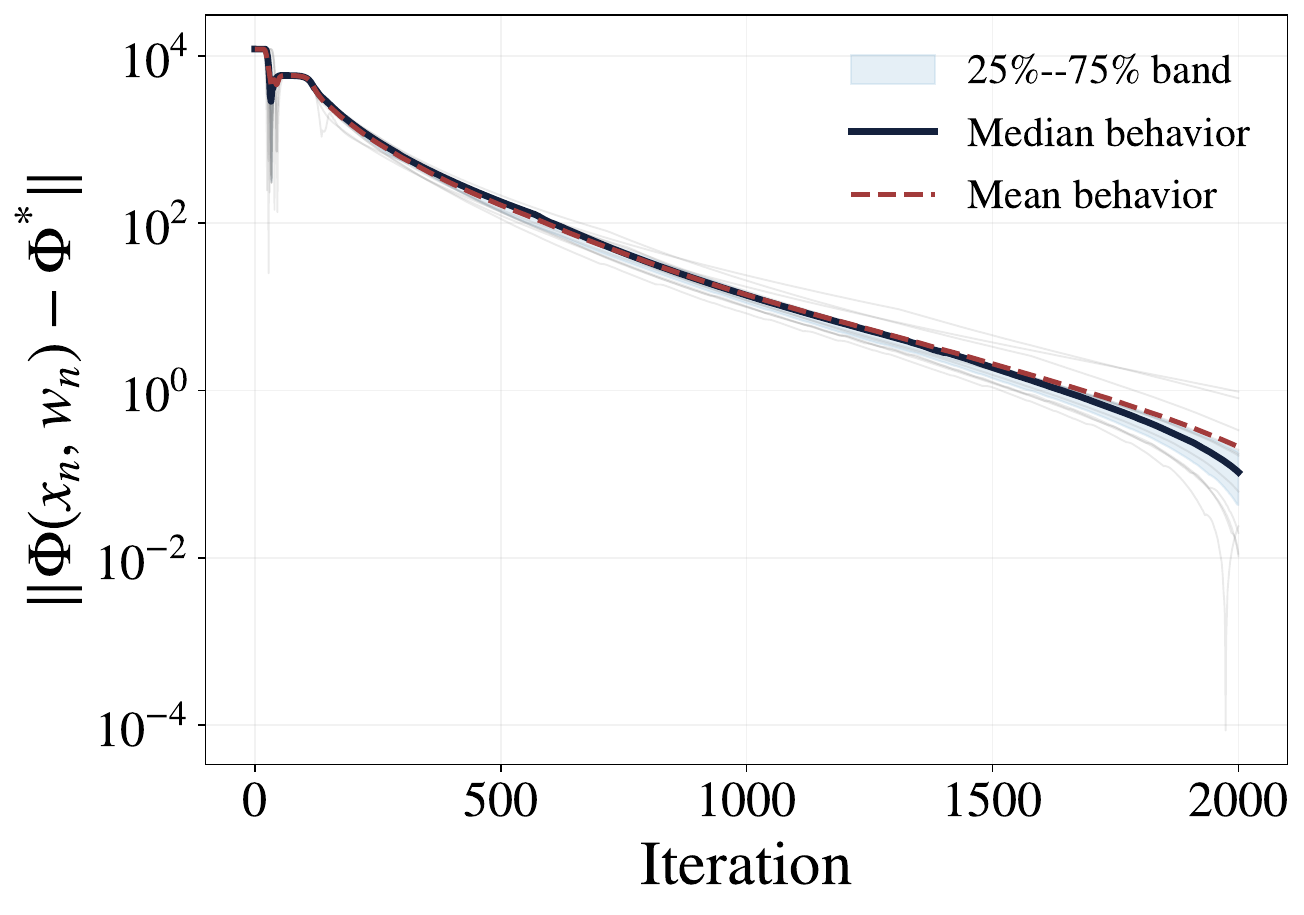}
}\hfill
\subfloat[Residual ratio $r_n$]{
  \includegraphics[width=0.27\linewidth]{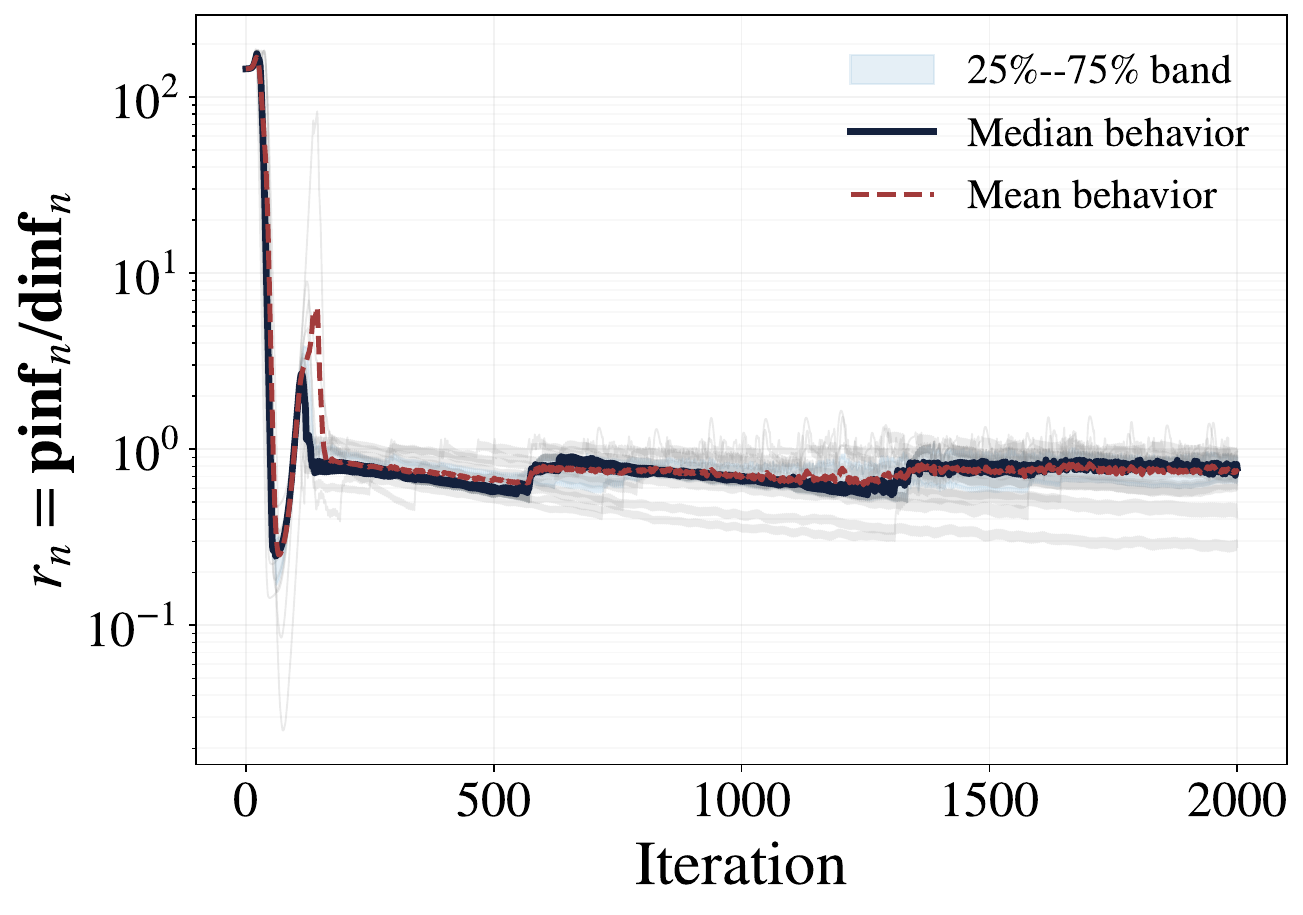}
}

\caption{Several heuristic adaptive-$\beta_n$ for PFGRPDA in \textbf{Setting~1} with motion blur. Thin faded curves show individual heuristic choices, bold solid and dashed curves show the median and mean trends, and the shaded band indicates the interquartile range.}
\label{fig:adaptive-beta-three}
\end{figure}
When the background noise is zero (i.e., $b=0$), which is generally achievable, the gradient $\nabla h$ is locally Lipschitz \cite[Section~5]{bauschke2017descent}.  We use the photon scaling factor \texttt{counts\_scale}$=150$, the regularisation parameter $\lambda=0.04$, and $2000$ iterations for the reported runs.  The convolution operator $A$ is implemented by FFT.  The initialisation is $x_0=[A^\top y]_+$, and $y_0=0$.  The discrete gradient operator has an operator norm $\|K\|\approx\sqrt 8$.  To avoid division by zero, KL evaluations and gradients are stabilised by replacing denominators with at least $10^{-10}$. To test the reconstruction quality of the images, we measure the PSNR

\[
\operatorname{PSNR}(x,x_{\rm true})=
10\log_{10}\!\left(\frac{n}{\|x-x_{\rm true}\|^2}\right).
\]

\begin{figure}[htbp!]
\centering
\subfloat[Primal objective gap]{
  \includegraphics[width=0.27\linewidth]{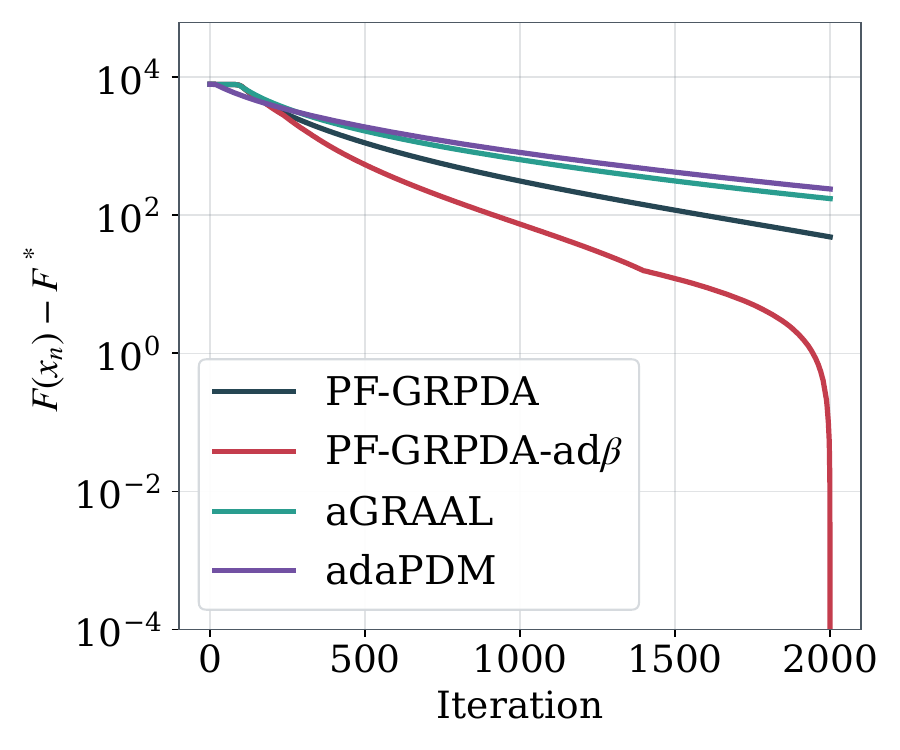}
}\hfill
\subfloat[Objective residual]{
  \includegraphics[width=0.27\linewidth]{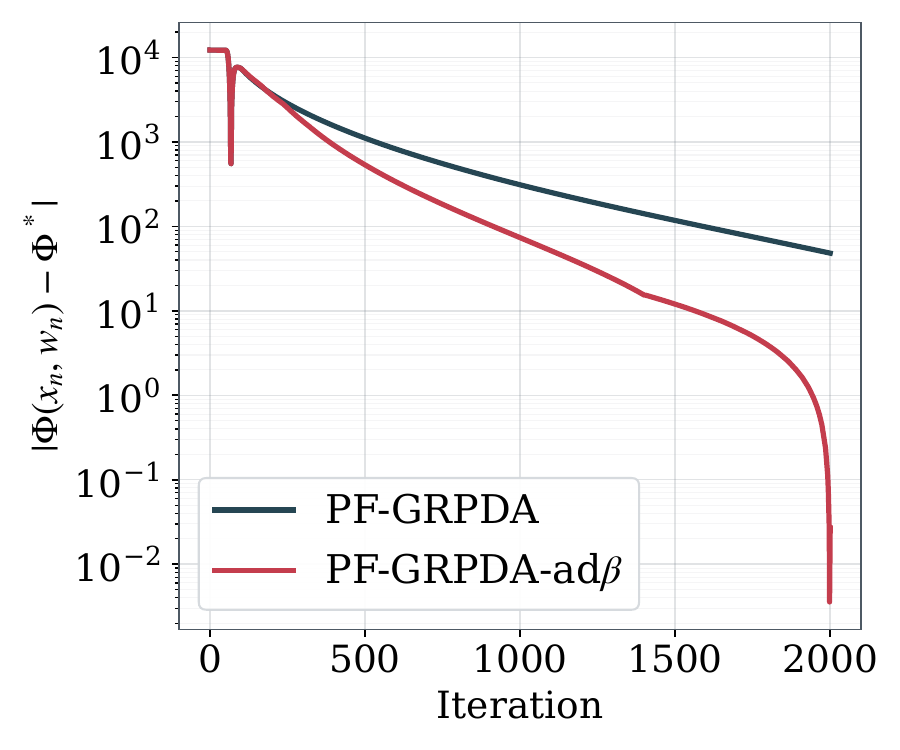}
}\hfill
\subfloat[Feasibility residual]{
  \includegraphics[width=0.27\linewidth]{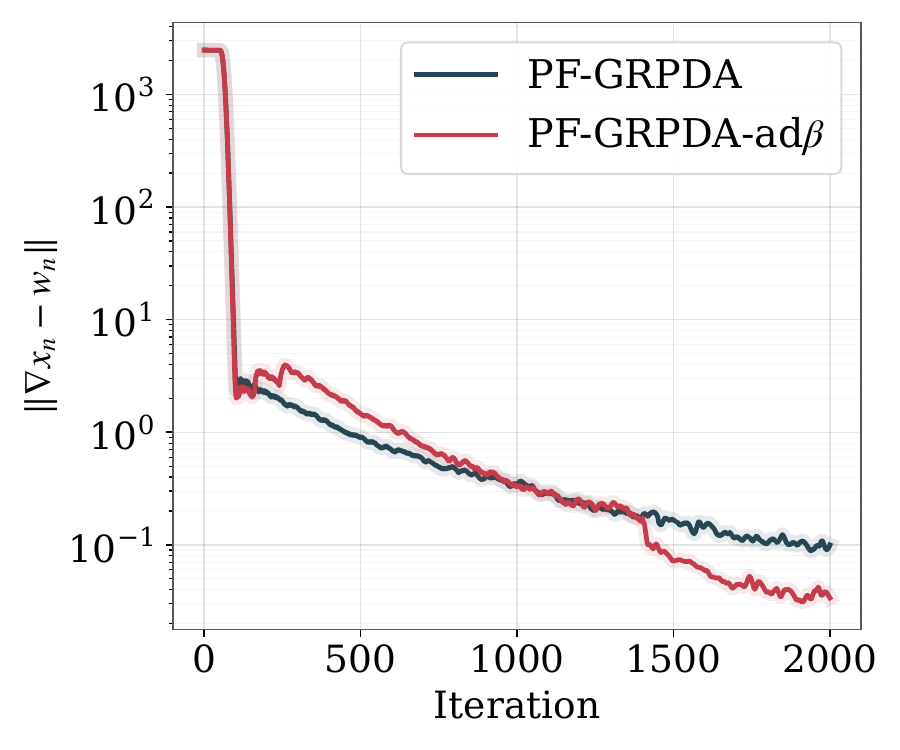}
}

\vspace{0.6em}
\subfloat[PSNR]{
  \includegraphics[width=0.27\linewidth]{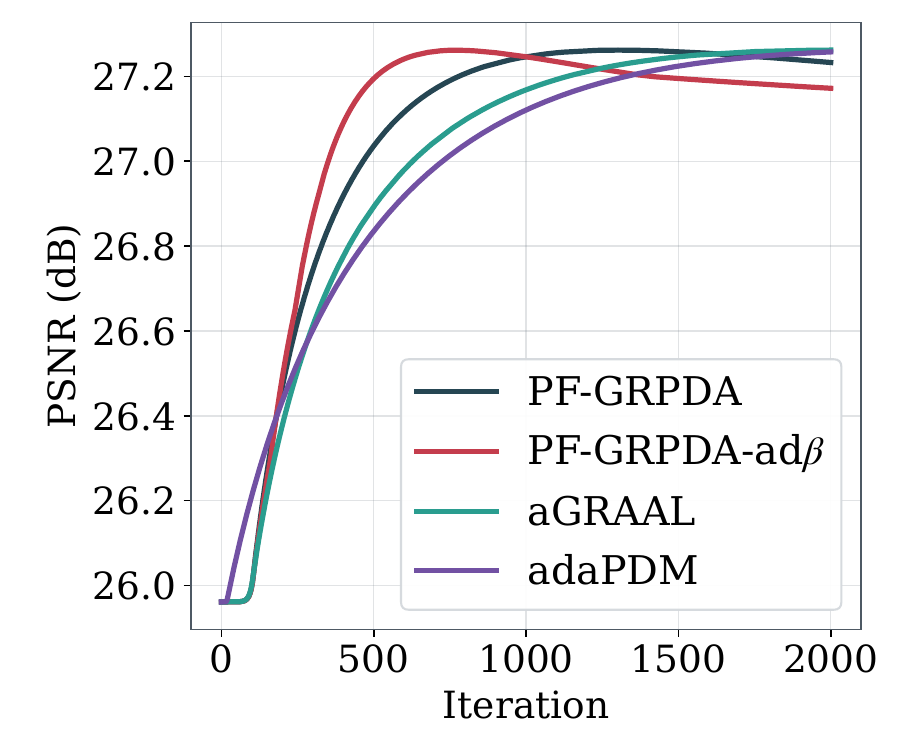}
}\hfill
\subfloat[Primal step-size $(\tau_n)$]{
  \includegraphics[width=0.27\linewidth]{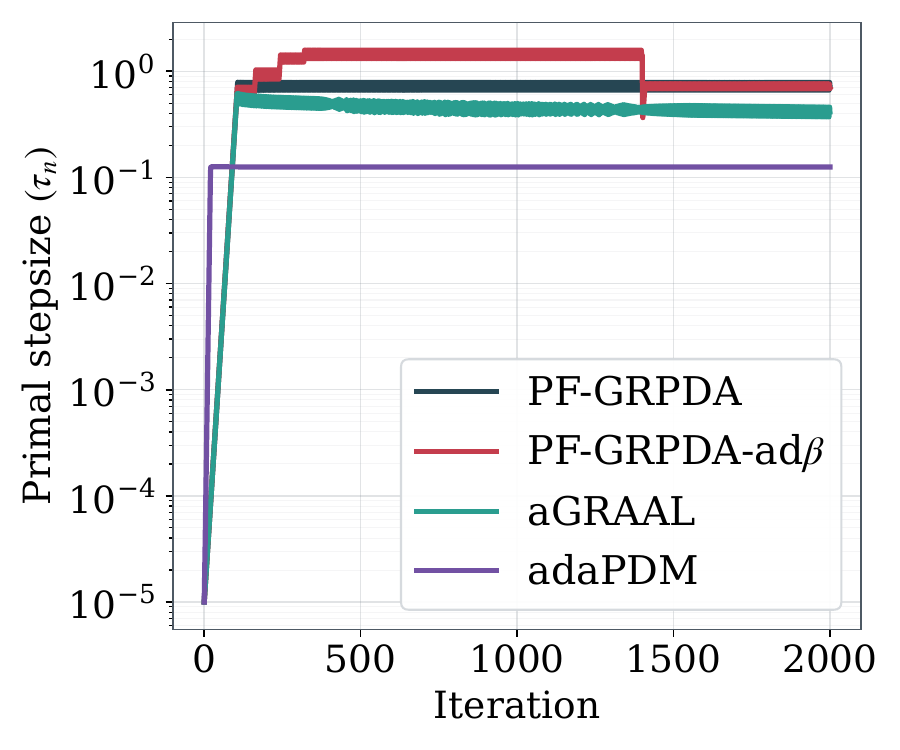}
}\hfill
\subfloat[PF-GRPDA selected $(\tau_n)$]{
  \includegraphics[width=0.28\linewidth]{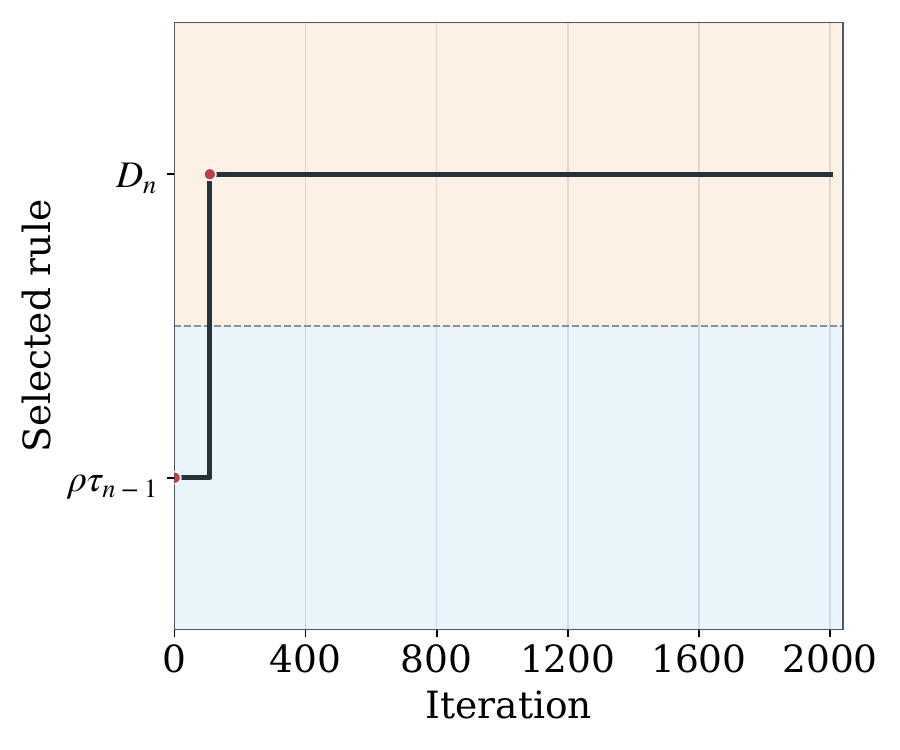}
}

\caption{Convergence results for \textbf{Setting~1} using the \texttt{Chemical plant} image degraded by motion blur}
\label{fig:setting_1_motion_blur_convergence_results}
\end{figure}
To better understand the behaviour of the adaptive step-size rule \eqref{eq:39c}, we also record which branch in the update of $(\tau_n)$ is selected during the PF-GRPDA run.  In the plot, the branch $(\rho\tau_{n-1})$ represents the growth step, while the second term in \eqref{eq:39c}, which we call $D_n$, corresponds to the local bound imposed by the curvature estimate. This result indicates whether the method primarily increases the step-size or is constrained by local smoothness information.  It is included only to illustrate the adaptive rule's internal behaviour.

We first consider the motion-blur instance.  Figure~\ref{fig:adaptive-beta-three}
shows the behaviour of the adaptive-$\beta_n$ heuristic.  The median curves
indicate that the residual-balancing rule can reduce both the feasibility violation and the objective residual compared with fixed choices of
$\beta_n$.  At the same time, the spread of the faded individual curves indicates that this update is not uniformly stabilising across all parameter choices. The reconstructions in Figure~\ref{fig:setting_1_motion_blur_recovery_results} show that all
methods remove the Poisson noise and recover the main geometric features of the true image.  However, from the convergence plots in Figure~\ref{fig:setting_1_motion_blur_convergence_results}, it can be seen that during the initial iterations, the objective gaps and PSNR curves of the methods are close; however, after this phase, the adaptive-$\beta_n$ variant (PF-GRPDA-ad $\beta$) begins to separate in the feasibility and objective gap plot, and gives noticeably smaller residuals.  This is consistent with the purpose that it can
improve the relative scaling of the primal and dual updates. It can also be observed that after $\approx 1400$ iterations, the PSNR values of the algorithms have essentially
saturated except PF-GRPDA-ad $\beta$, where PSNR have decreased when reaching $2000$ iterations. This explains that the heuristic adaptive $\beta_n$ does not guarantee better image reconstruction always, but it is a great alternative to have a faster version of PF-GRPDA.

We next repeat the convex KL--TV experiment to reconstruct the images degraded by the \emph{defocus-blur}.  The residuals decay of the adaptive-$\beta_n$ rule is shown in Figure~\ref{fig:adaptive-beta_new}, the reconstructed images are displayed in Figure~\ref{fig:setting_1_defocus_blur_recovery_results}, and the convergence results are reported in Figure~\ref{fig:setting_1_defocus_blur_convergence_results}.  Compared with motion blur, the defocus blur is less directional, but the main numerical behaviour is similar.  The adaptive-$\beta_n$ version improves the PSNR values; however, feasibility and objective residuals remain comparable to PF-GRPDA and aGRAAL and adaPDM algorithms.

\subsubsection{Setting 2}

In the second setting, we add the quadratic term and consider

\begin{equation}\label{eq:poisson_strng_f}
    \min_{x\ge 0} F(x):=
    \frac{\mu}{2}\|x\|^2+\mathrm{KL}(Ax,y)+\lambda\|\nabla x\|_{2,1},
\end{equation}
where $\mu>0$.  This corresponds to~\eqref{main_prob} with

\[
    f(x)=\iota_{\{x\ge 0\}}(x)+\frac{\mu}{2}\|x\|^2,
    \qquad
    g(w)=\lambda\|w\|_{2,1},
    \qquad
    K=\nabla,
    \qquad
    h(x)=\mathrm{KL}(Ax,y).
\]
The aim of this experiment is not simply to add another regularisation, but rather to check whether methods that exploit strong convexity, especially PF-AGRPDA and aPDAc-L, benefit from the additional curvature in practice.  The data generation, blur operators, image size, Poisson sampling, and initialisation are kept exactly the same as in \textbf{\textbf{Setting~1}}.  In this experiment, we set $\mu=0.01$.
\begin{figure}[htbp]
\centering
\subfloat[Feasibility residual]{
  \includegraphics[width=0.27\linewidth]{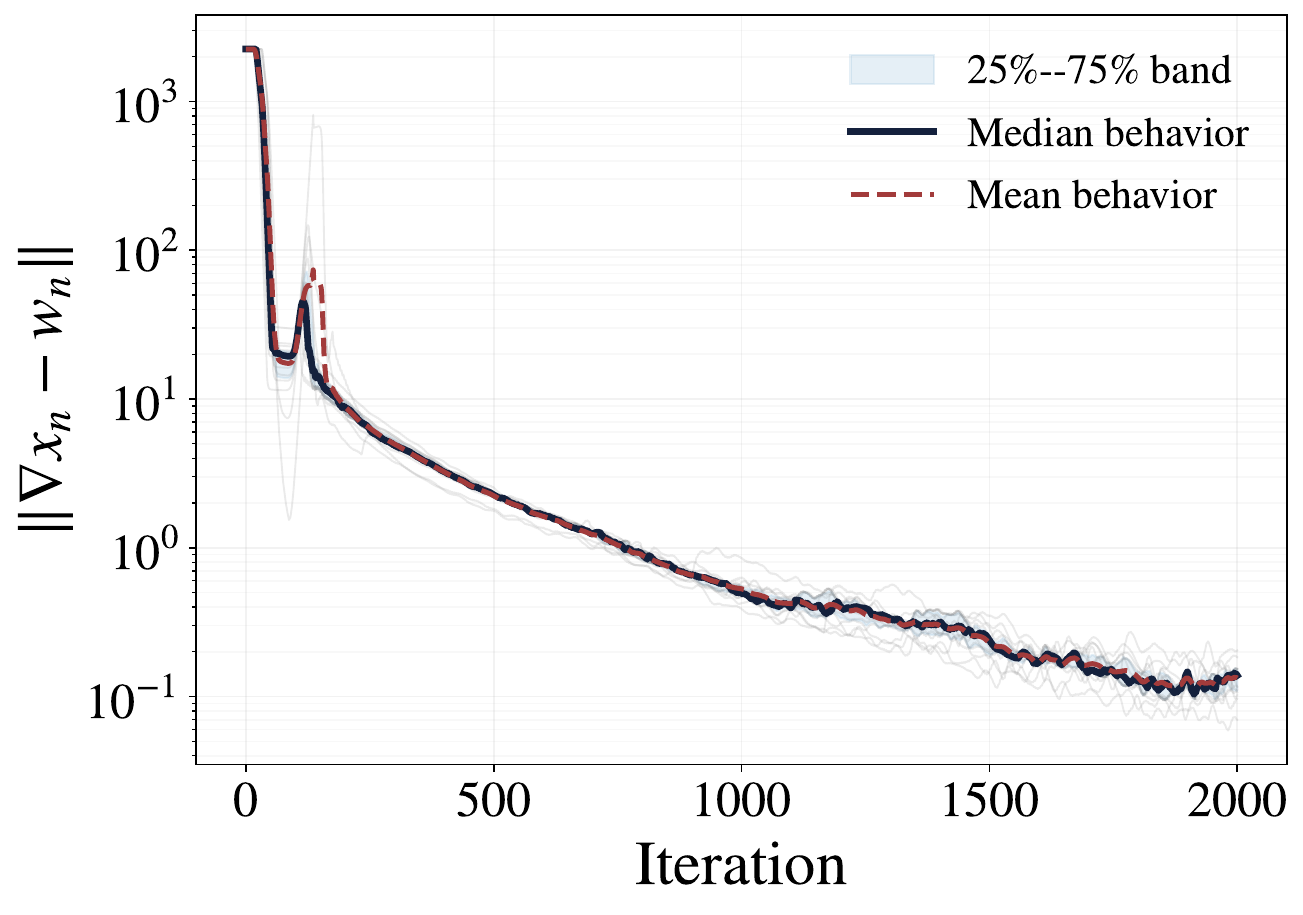}
}\hfill
\subfloat[Objective residual]{
  \includegraphics[width=0.27\linewidth]{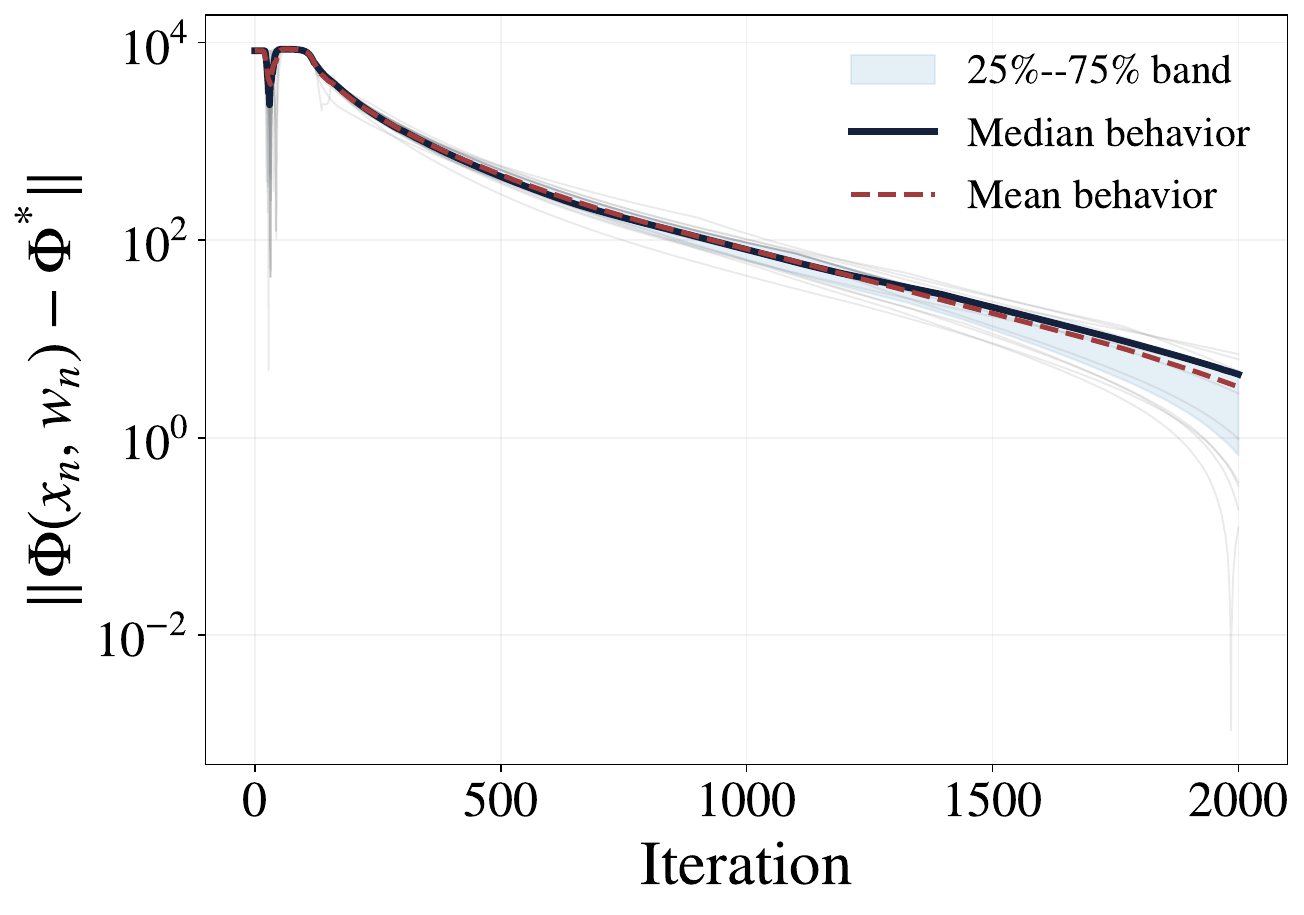}
}\hfill
\subfloat[Residual ratio $r_n$]{
  \includegraphics[width=0.27\linewidth]{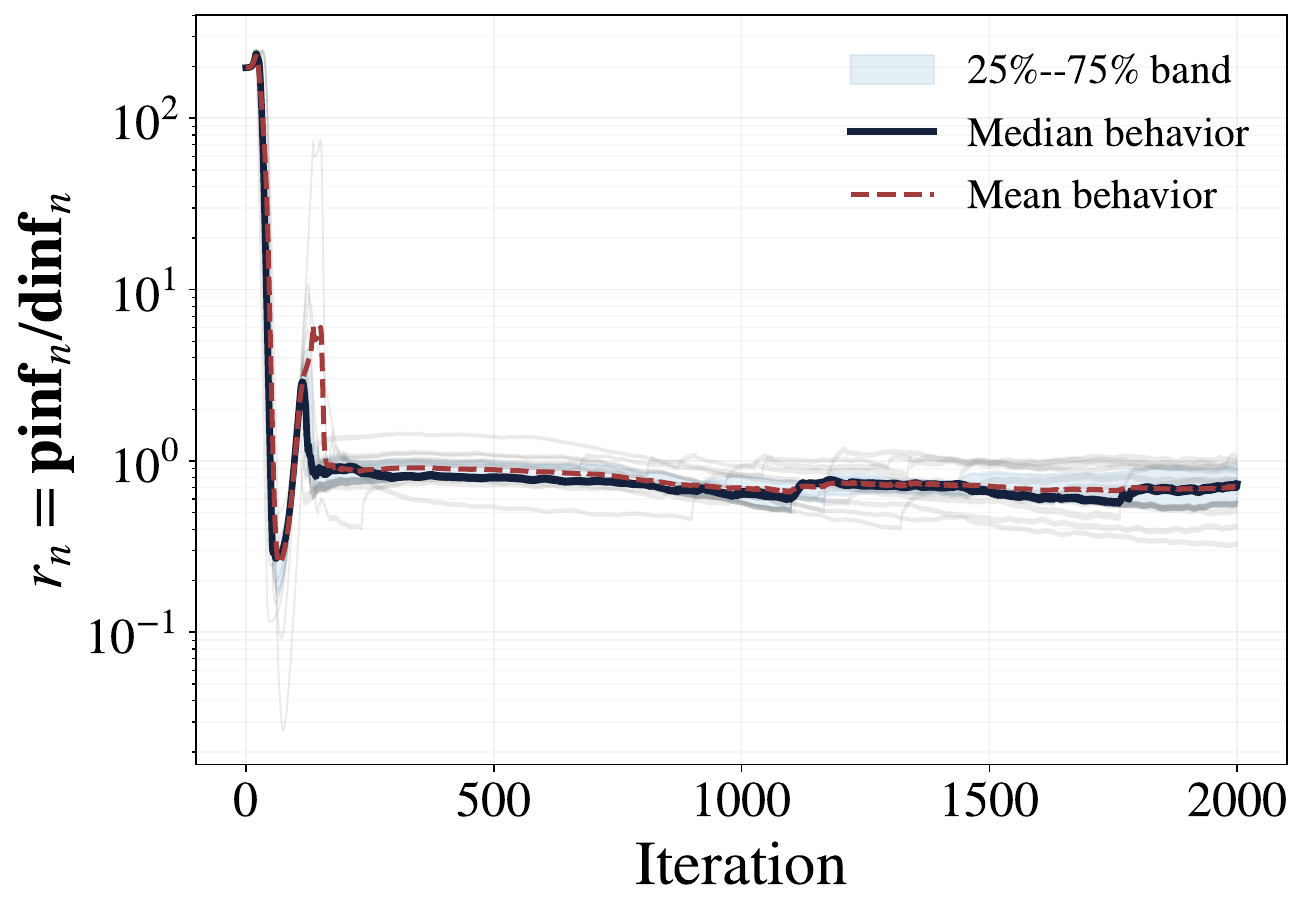}
}

\caption{Several heuristic adaptive-$\beta_n$ for PFGRPDA in \textbf{Setting~1} with defocus blur. Thin faded curves show individual heuristic choices, bold solid and dashed curves show the median and mean trends, and the shaded band indicates the interquartile range.}
\label{fig:adaptive-beta_new}
\end{figure}

\begin{figure}[htbp]
\centering
\subfloat[PF-GRPDA-ad$\beta$]{
  \includegraphics[width=0.20\linewidth]{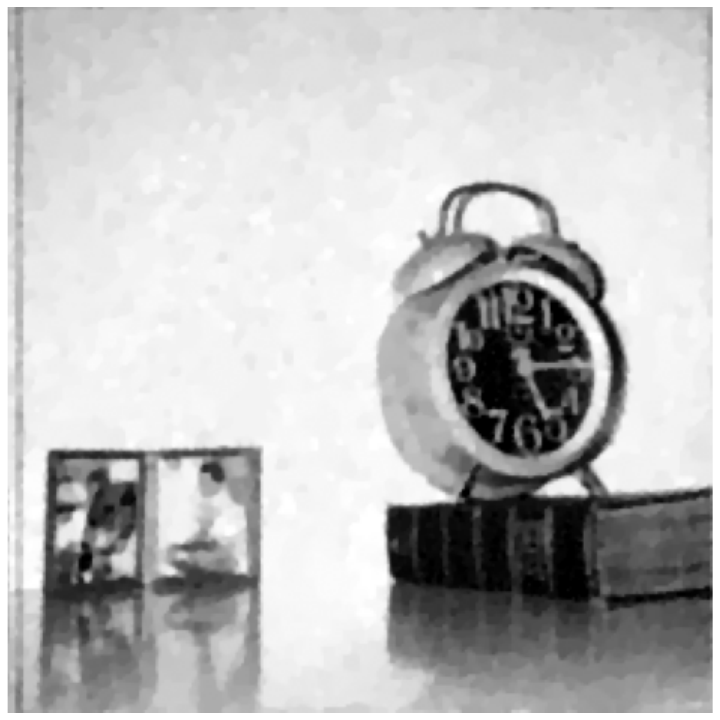}
}\hfill
\subfloat[PF-GRPDA]{
  \includegraphics[width=0.20\linewidth]{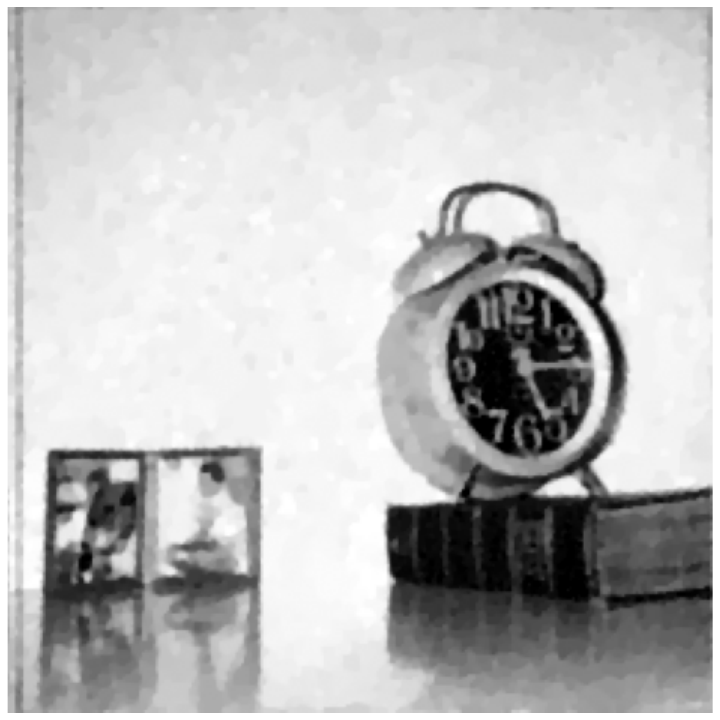}
}\hfill
\subfloat[aGRAAL]{
  \includegraphics[width=0.20\linewidth]{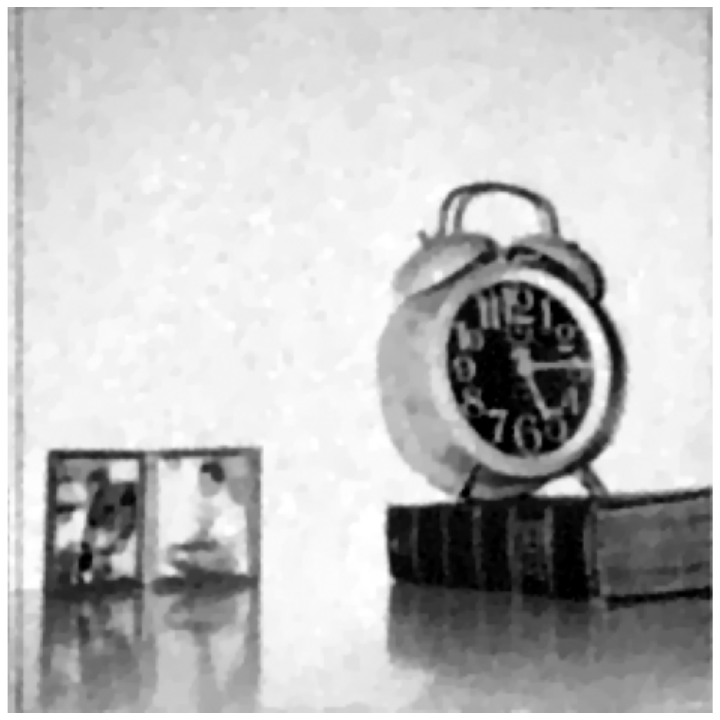}
}\hfill
\subfloat[adaPDM]{
  \includegraphics[width=0.20\linewidth]{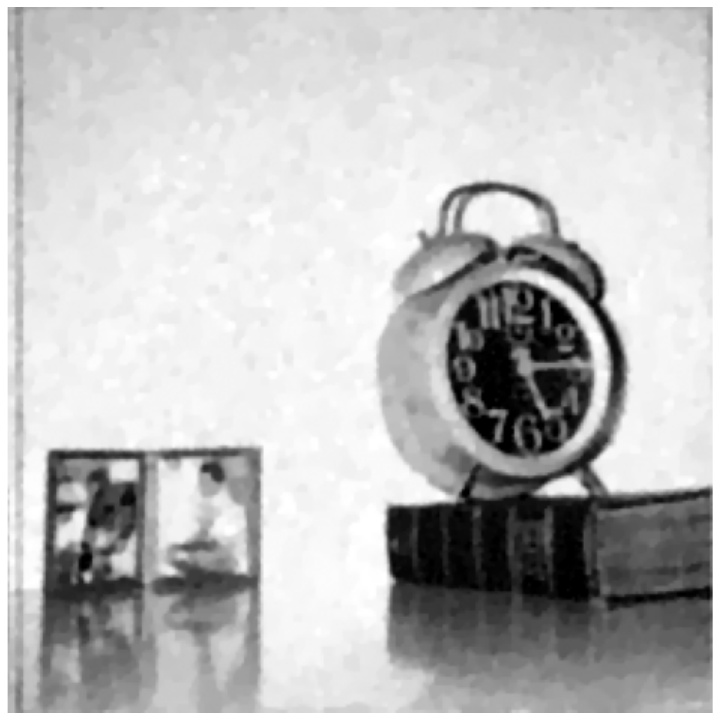}
}

\caption{Image reconstruction results for \textbf{Setting~1} using the \texttt{Clock} image
degraded by defocus blur.}
\label{fig:setting_1_defocus_blur_recovery_results}
\end{figure}
\begin{figure}[htbp!]
\centering
\subfloat[Primal objective gap]{
  \includegraphics[width=0.27\linewidth]{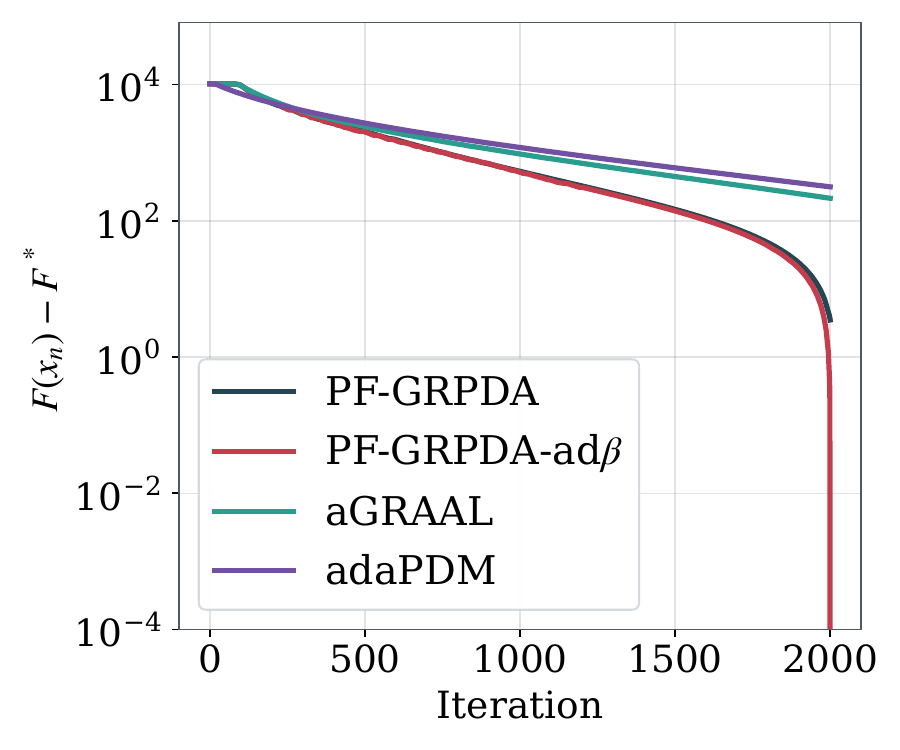}
}\hfill
\subfloat[Objective residual]{
  \includegraphics[width=0.27\linewidth]{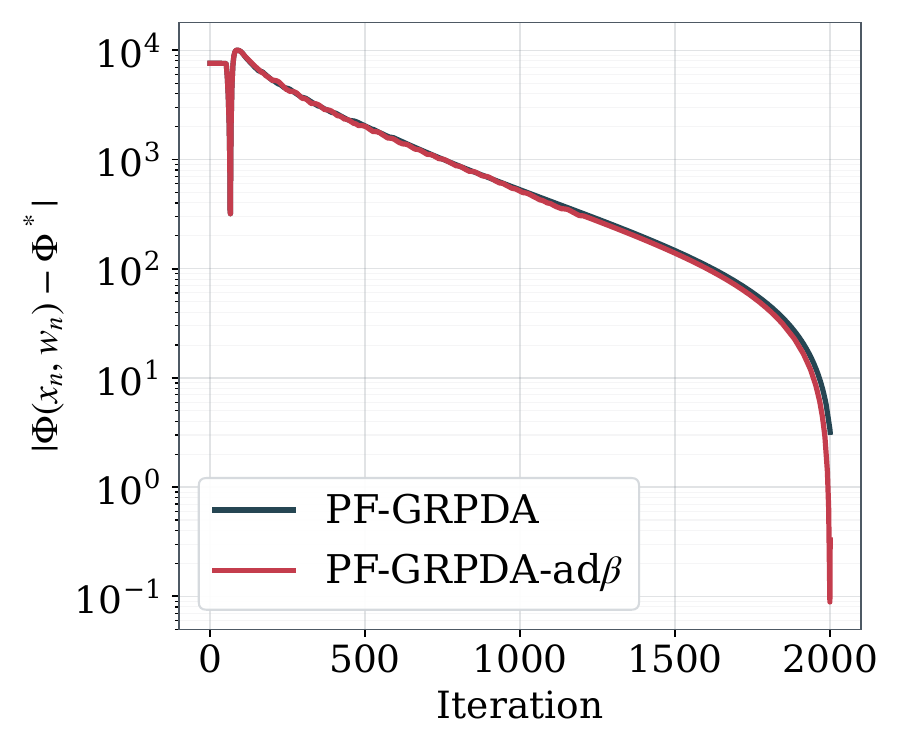}
}\hfill
\subfloat[Feasibility residual]{
  \includegraphics[width=0.27\linewidth]{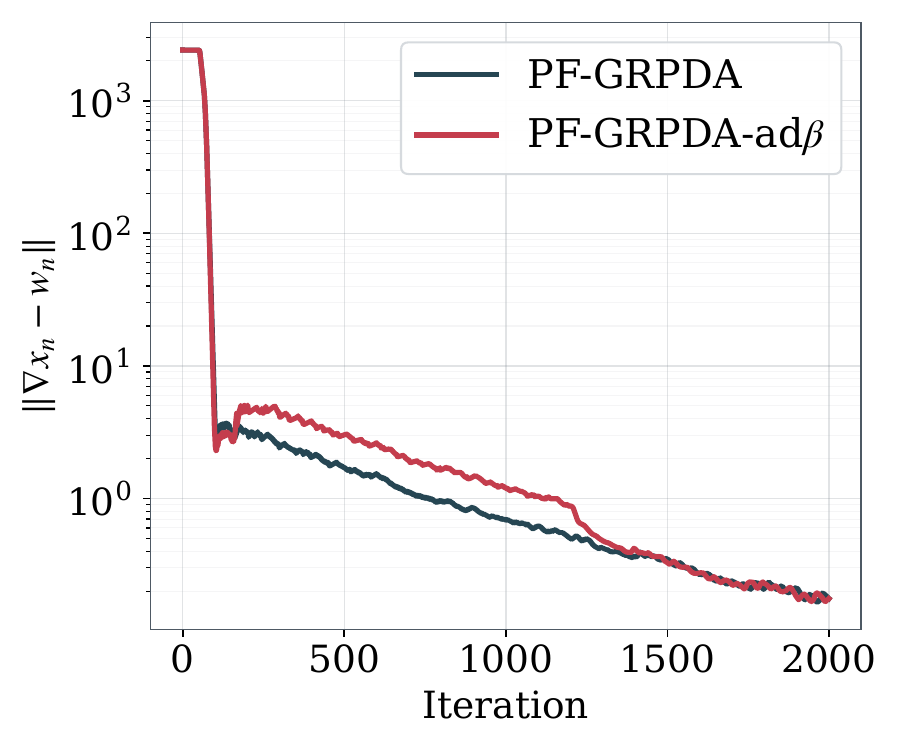}
}

\vspace{0.6em}
\subfloat[PSNR]{
  \includegraphics[width=0.27\linewidth]{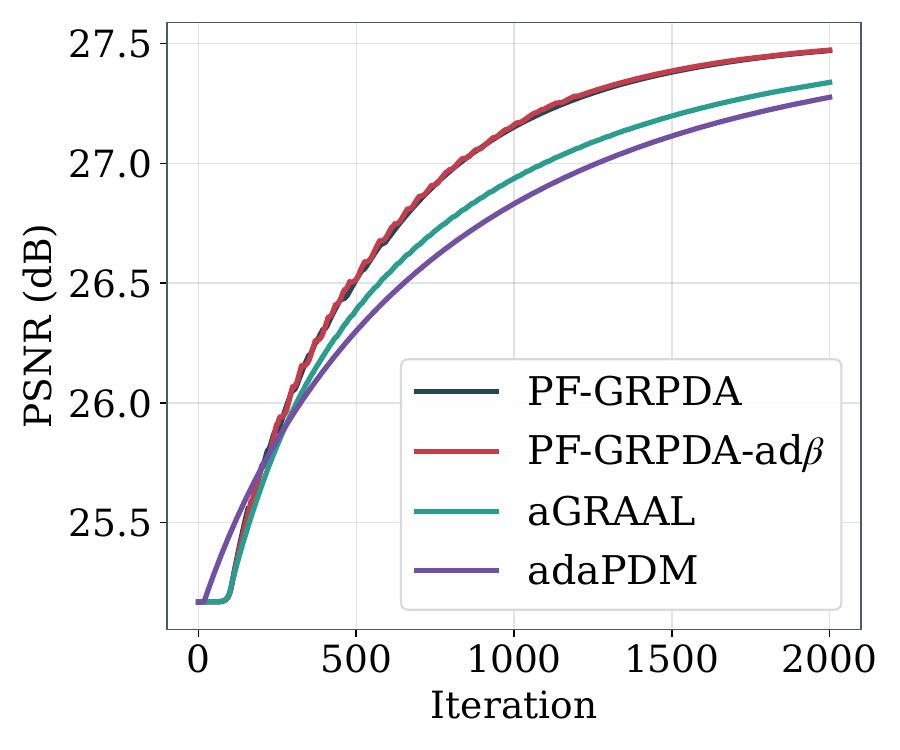}
}\hfill
\subfloat[Primal step-size $(\tau_n)$]{
  \includegraphics[width=0.27\linewidth]{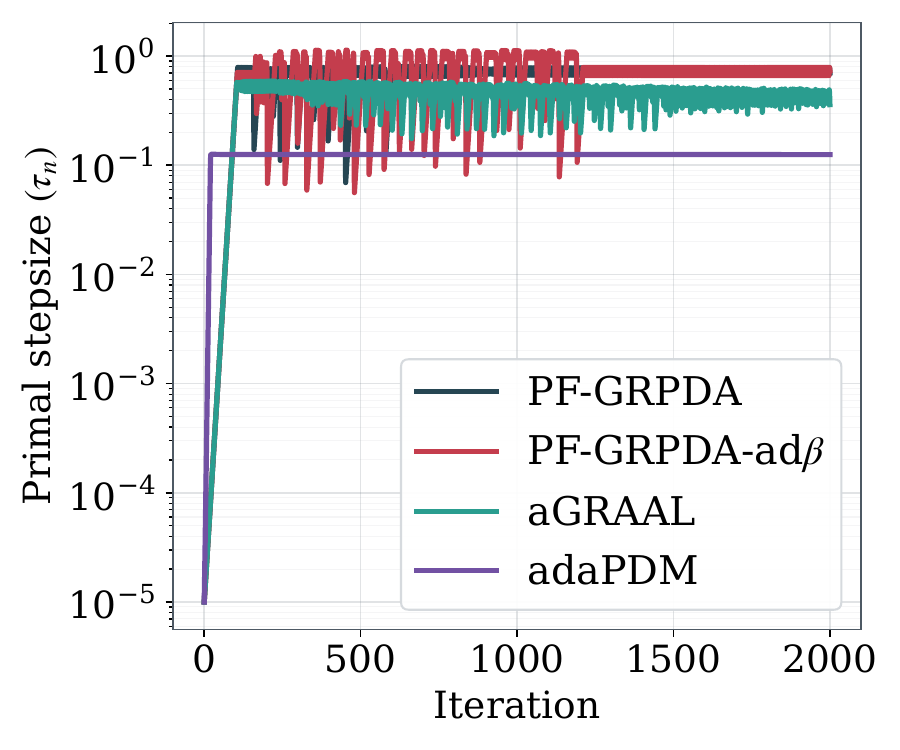}
}\hfill
\subfloat[PF-GRPDA selected $(\tau_n)$]{
  \includegraphics[width=0.27\linewidth]{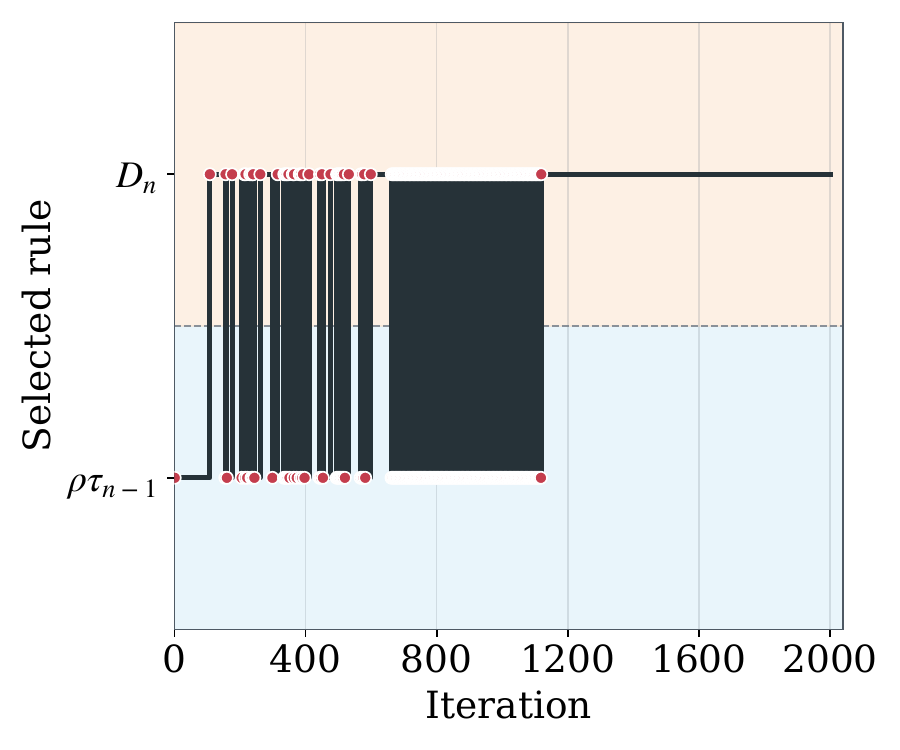}
}

\caption{Convergence results for \textbf{Setting~1} using the \texttt{Clock} image degraded by defocus blur}
\label{fig:setting_1_defocus_blur_convergence_results}
\end{figure}
Figures~\ref{fig:recovery-setting2} and~\ref{fig:metrics-setting2_1} report the results for the strongly convex motion-blur case.  The reconstruction obtained by PF-AGRPDA is visually close to that of PF-GRPDA.  However, the convergence plots in Figure \ref{fig:metrics-setting2_1} show a clear advantage of PF-AGRPDA in terms of feasibility and objective residual.  This is consistent with the role of the acceleration as it is not expected to drastically change the final visual reconstruction, but rather to use the added strong convexity to improve the decay of the residuals. 

\medskip
\noindent
For the defocus-blur case, Figures~\ref{fig:recovery-setting2_1} and
\ref{fig:metrics-setting2} show the same overall pattern. PF-AGRPDA reduces the feasibility residual substantially faster than
PF-GRPDA and PF-GRPDA-ad$\beta$, while maintaining comparable PSNR.  Among all methods, aPDAc-L is very competitive in the residual plots.  This is
not contradictory to the purpose of PF-AGRPDA, because aPDAc-L uses a
backtracking linesearch.  The relevant conclusion is that PF-AGRPDA obtains a
clear improvement over its non-accelerated counterpart while preserving the
same reconstruction quality and avoiding linesearch.

\begin{figure}[htbp]
\centering
\subfloat[True image]{
  \includegraphics[width=0.20\linewidth]{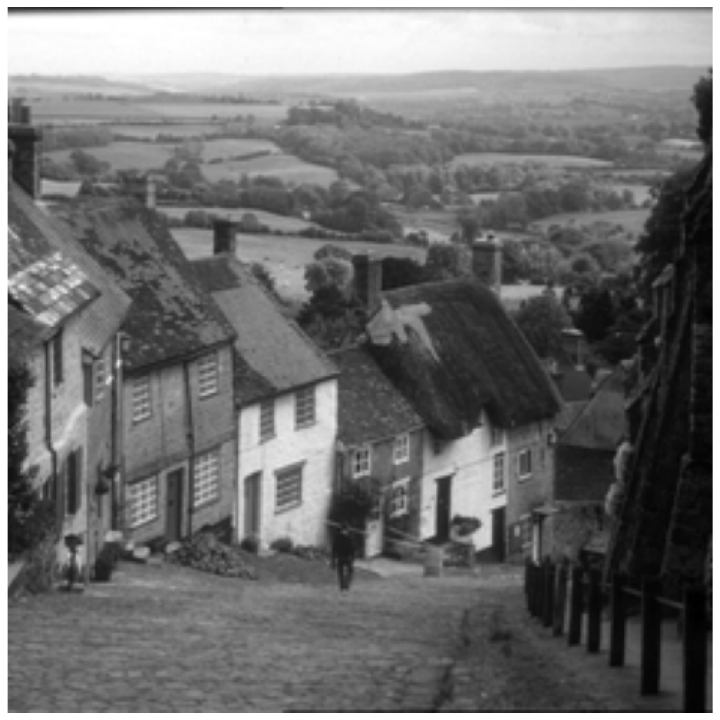}
}\hfill
\subfloat[Noisy data]{
  \includegraphics[width=0.20\linewidth]{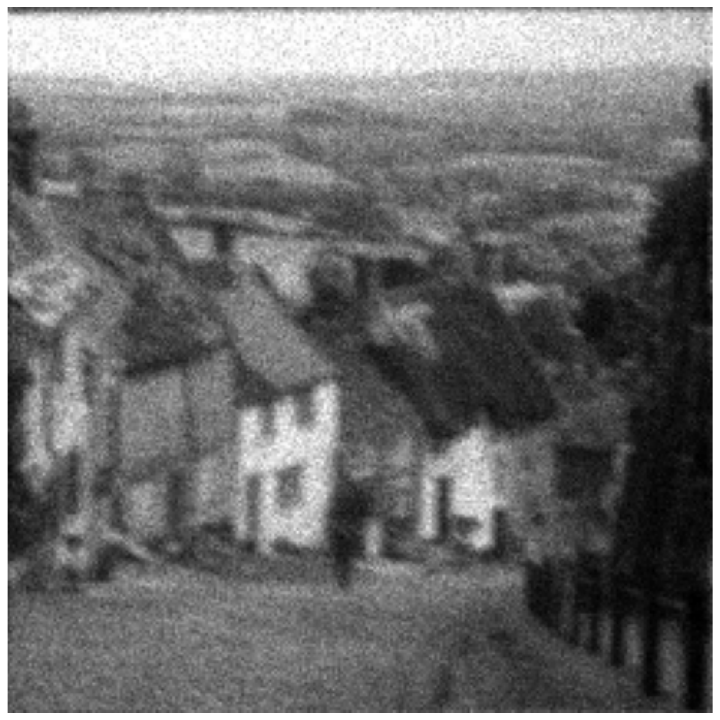}
}\hfill
\subfloat[PF-GRPDA]{
  \includegraphics[width=0.20\linewidth]{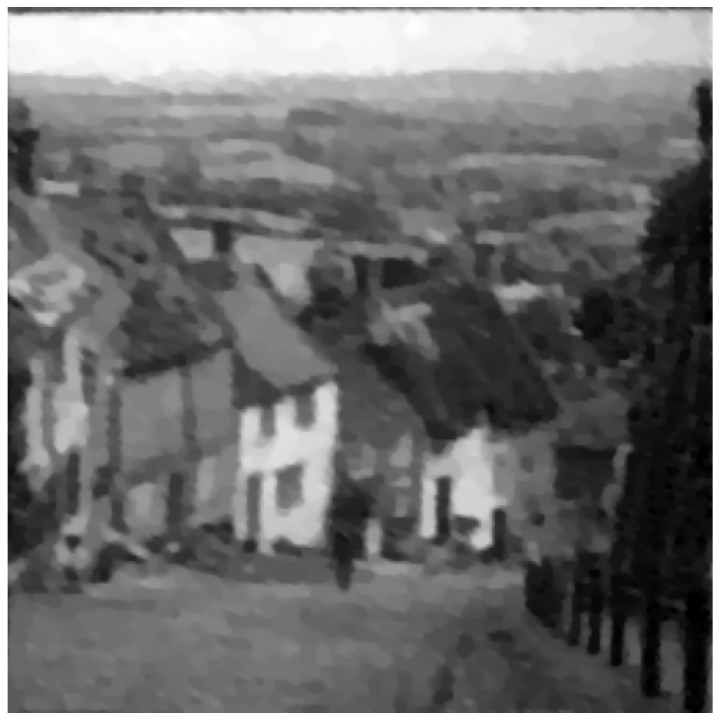}
}\hfill
\subfloat[PF-GRPDA-ad$\beta$]{
  \includegraphics[width=0.20\linewidth]{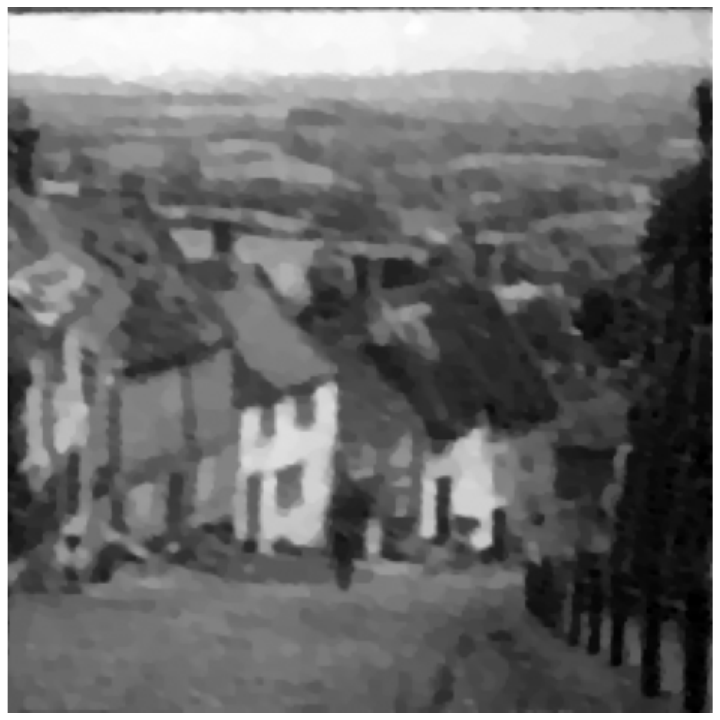}
}

\vspace{0.6em}
\subfloat[PF-AGRPDA]{
  \includegraphics[width=0.20\linewidth]{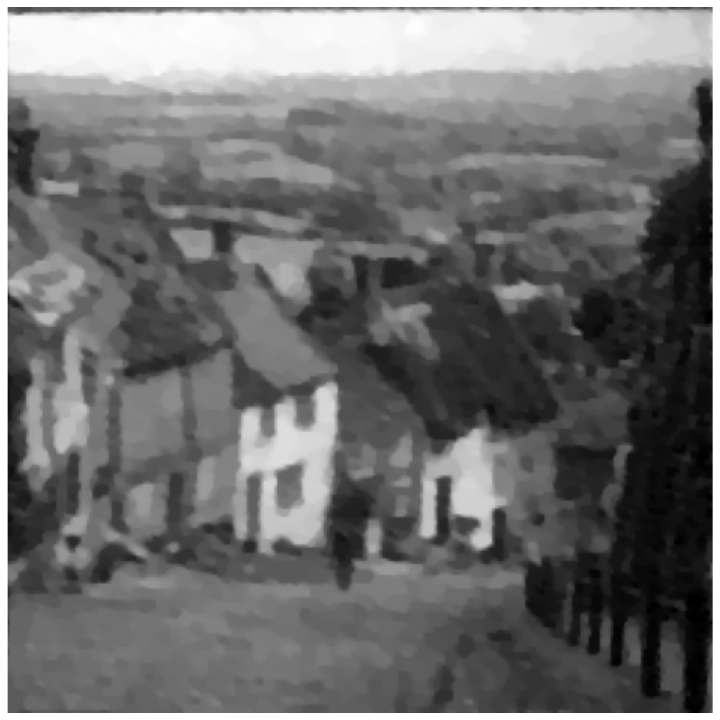}
}\hfill
\subfloat[adaPDM]{
  \includegraphics[width=0.20\linewidth]{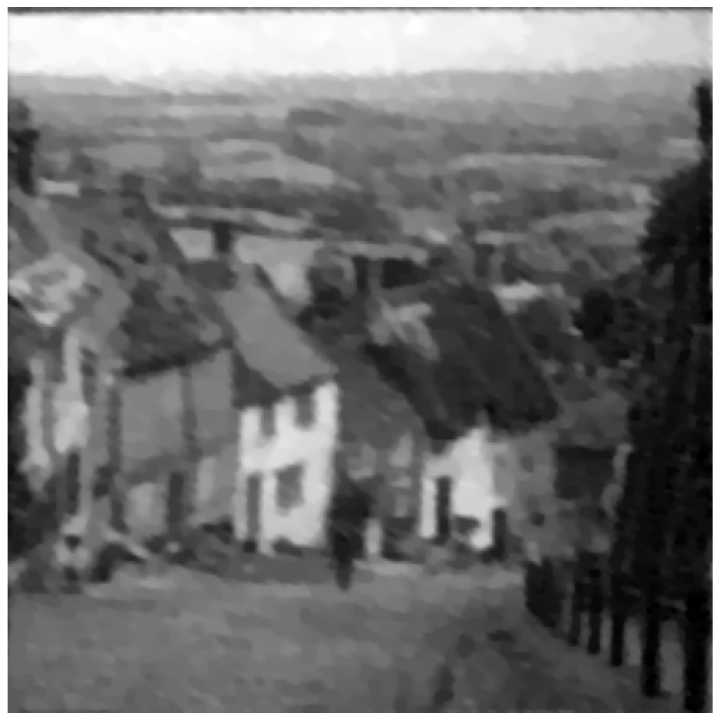}
}\hfill
\subfloat[aGRAAL]{
  \includegraphics[width=0.20\linewidth]{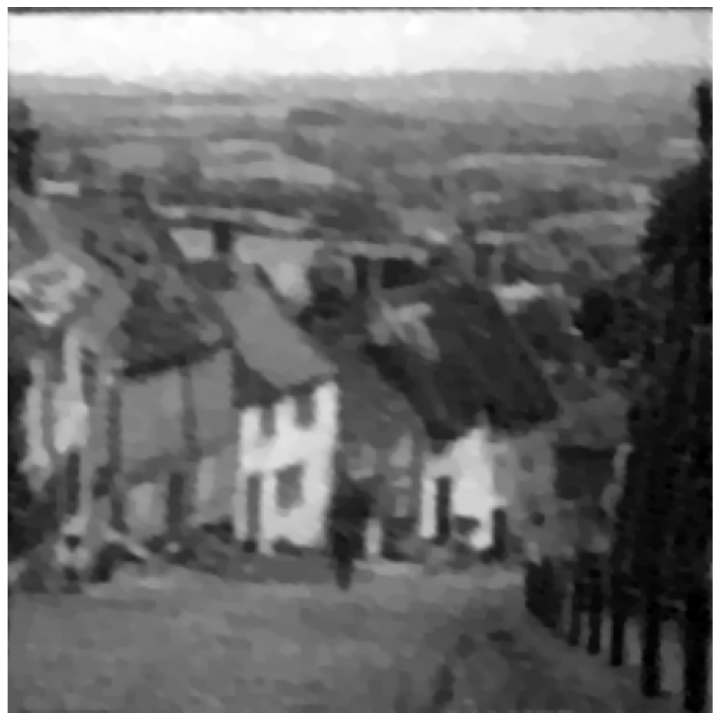}
}\hfill
\subfloat[aPDAc-L]{
  \includegraphics[width=0.20\linewidth]{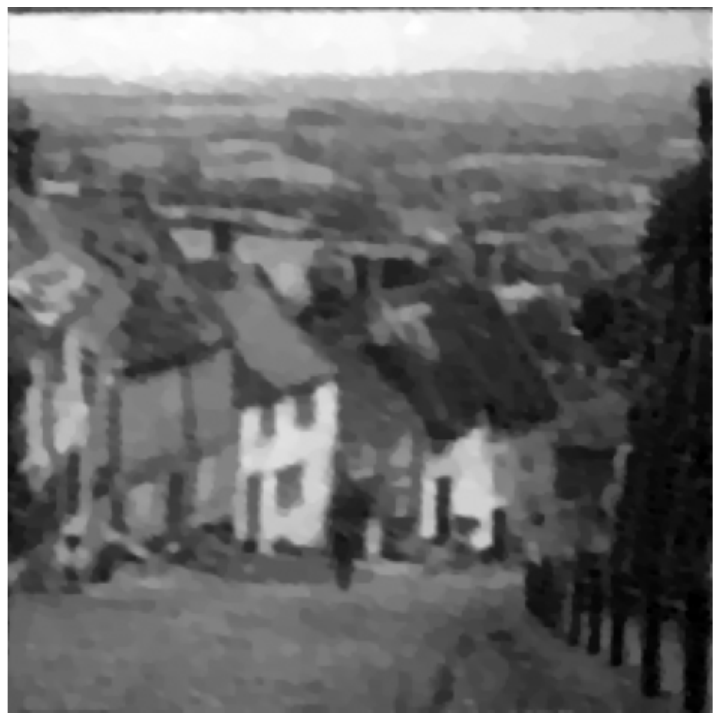}
}

\caption{Image reconstruction results for \textbf{Setting~2} using the \texttt{Goldhill} image degraded by motion blur.}
\label{fig:recovery-setting2}
\end{figure}

\begin{figure}[htbp]
\centering
\subfloat[Primal objective gap]{
  \includegraphics[width=0.27\linewidth]{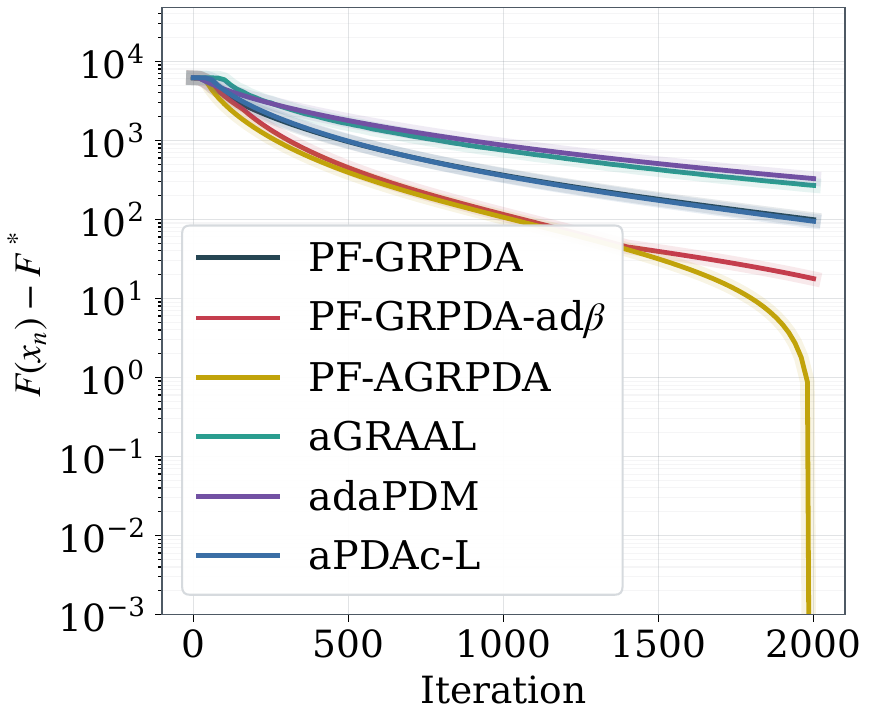}
}\hfill
\subfloat[Objective residual]{
  \includegraphics[width=0.27\linewidth]{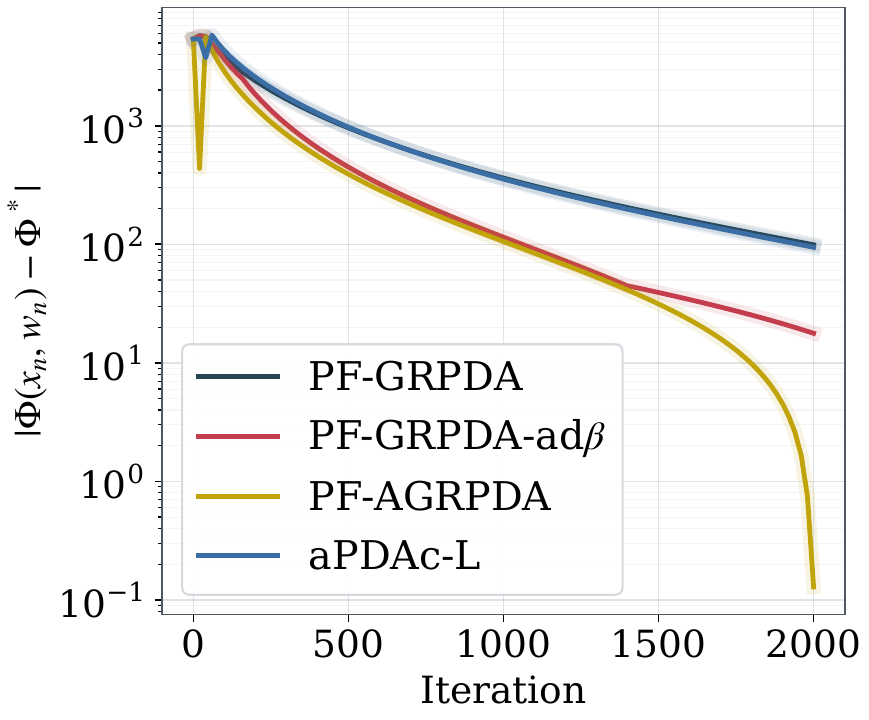}
}\hfill
\subfloat[Feasibility residual]{
  \includegraphics[width=0.27\linewidth]{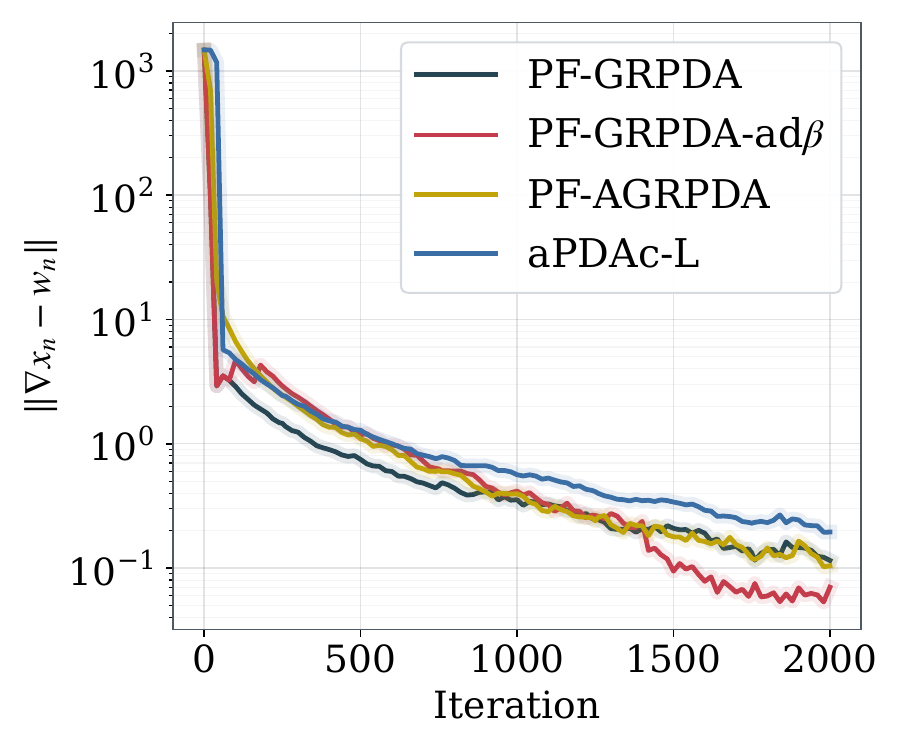}
}

\vspace{0.8em}
\subfloat[PSNR]{
  \includegraphics[width=0.27\linewidth]{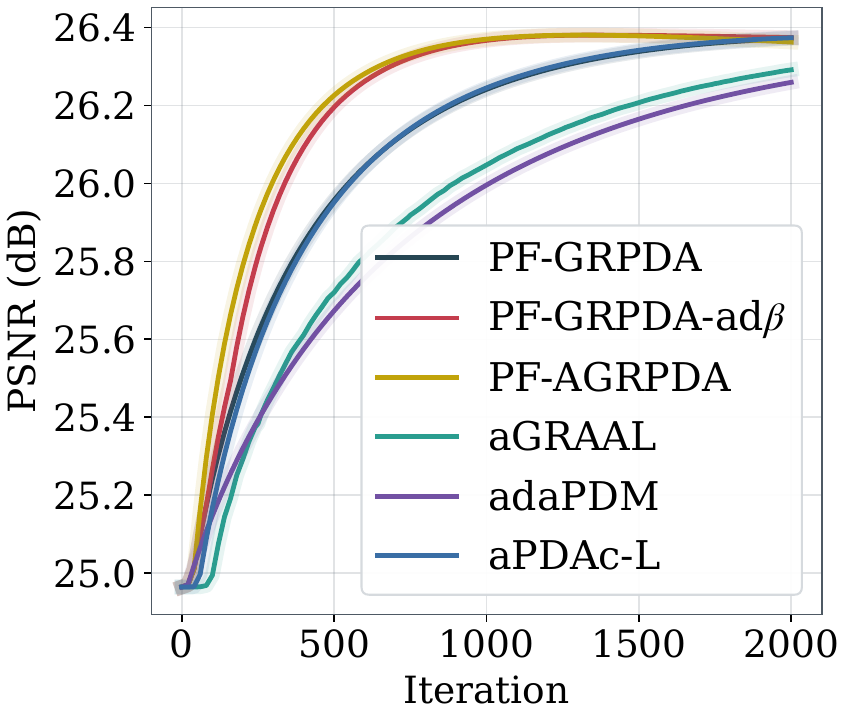}
}\hfill
\subfloat[Primal step-size $(\tau_n)$]{
  \includegraphics[width=0.27\linewidth]{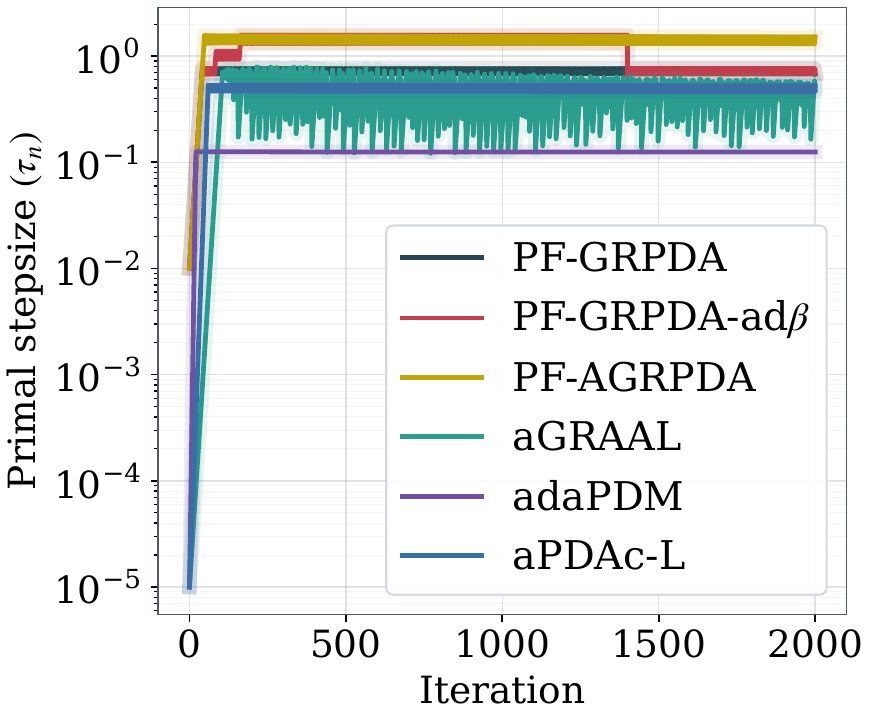}
}\hfill
\subfloat[PF-AGRPDA selected ($\tau_n$)]{
  \includegraphics[width=0.27\linewidth]{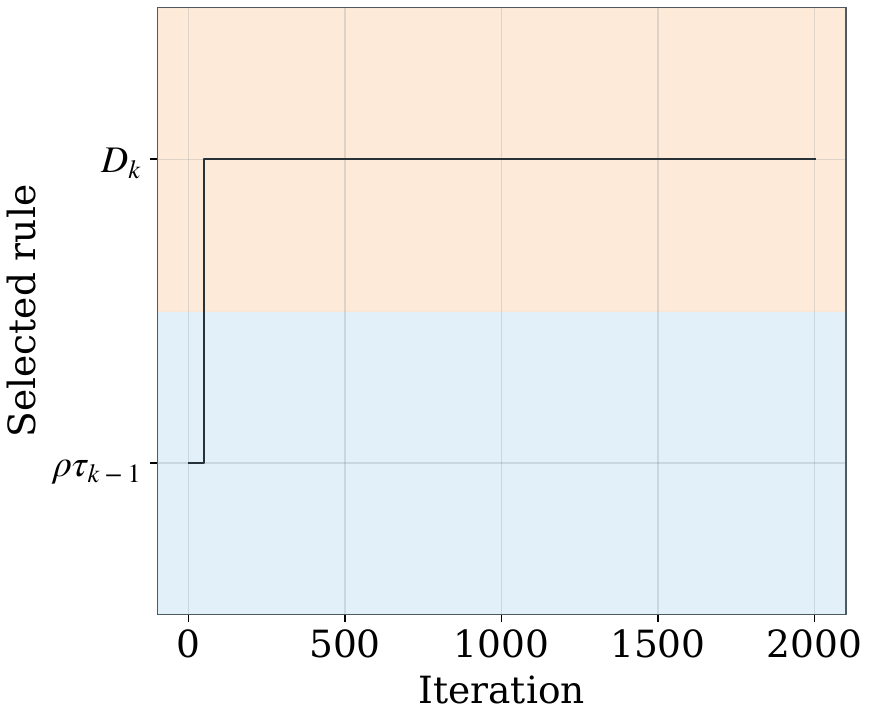}
}

\caption{Convergence results for \textbf{Setting~2} using the \texttt{Goldhill} image degraded by motion blur}
\label{fig:metrics-setting2_1}
\end{figure}

\begin{figure}[htbp]
\centering
\subfloat[True image]{
  \includegraphics[width=0.20\linewidth]{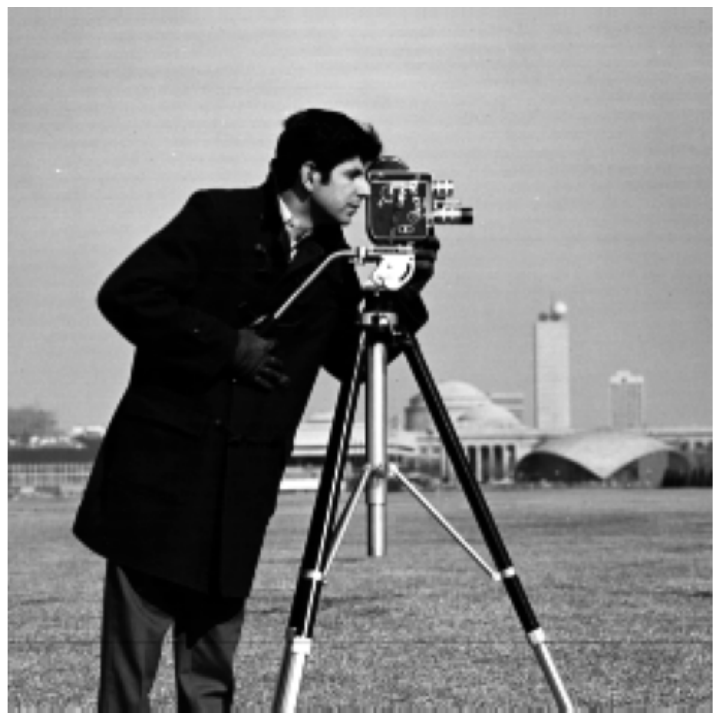}
}\hfill
\subfloat[Noisy data]{
  \includegraphics[width=0.20\linewidth]{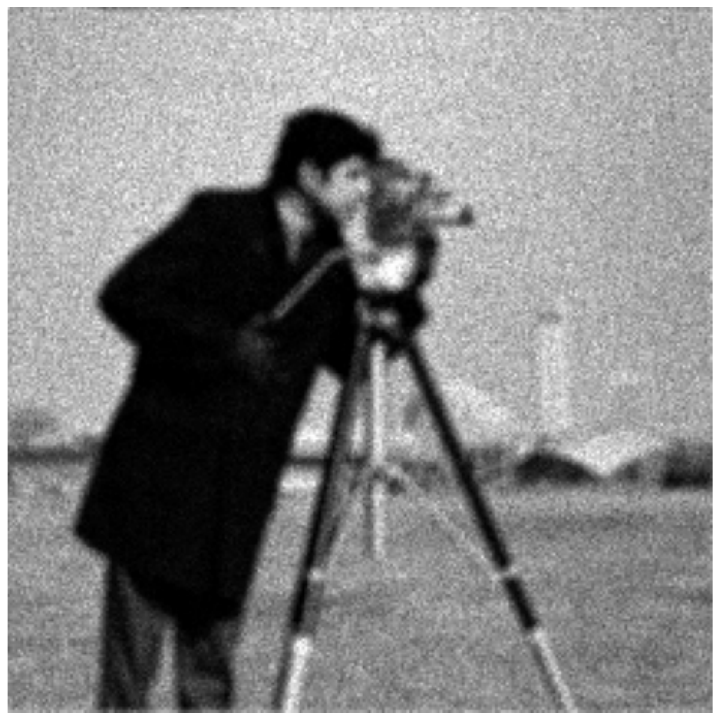}
}\hfill
\subfloat[PF-GRPDA]{
  \includegraphics[width=0.20\linewidth]{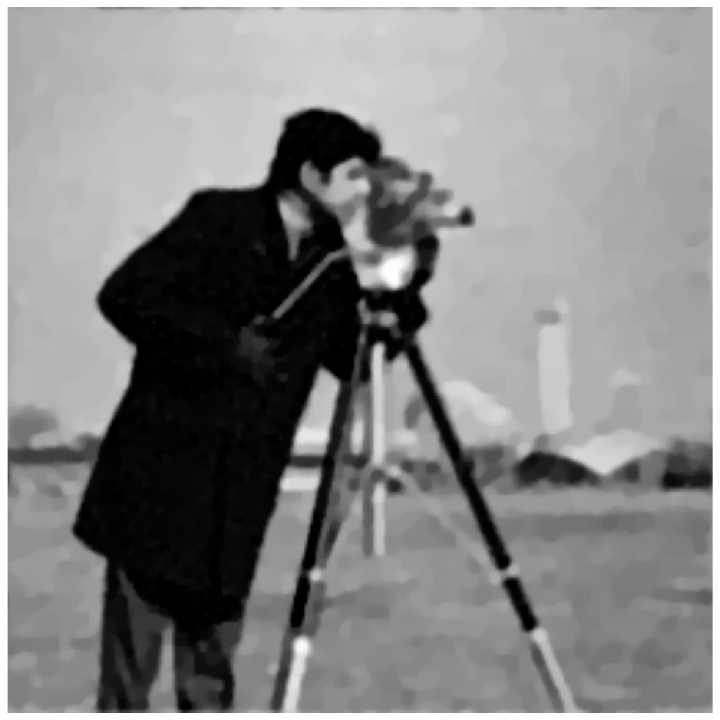}
}\hfill
\subfloat[PF-GRPDA-ad$\beta$]{
  \includegraphics[width=0.20\linewidth]{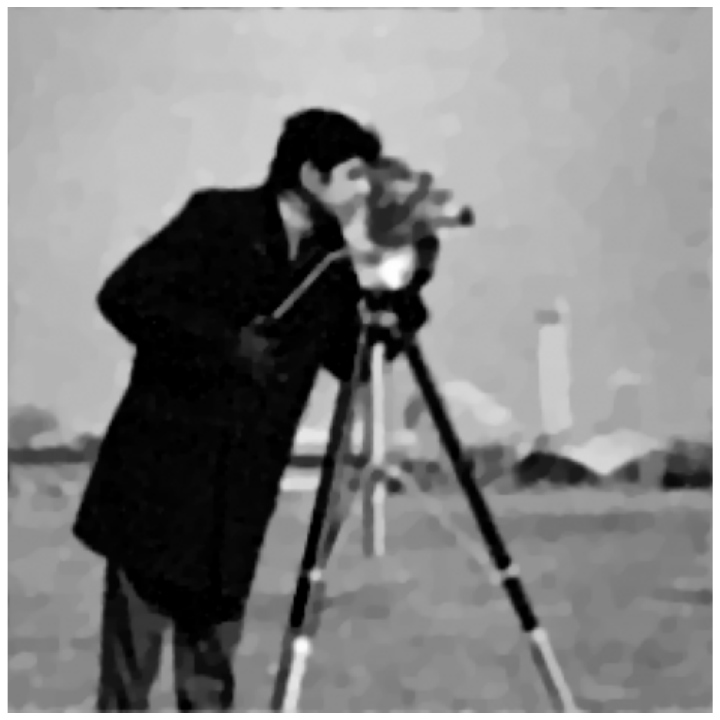}
}

\vspace{0.6em}
\subfloat[PF-AGRPDA]{
  \includegraphics[width=0.20\linewidth]{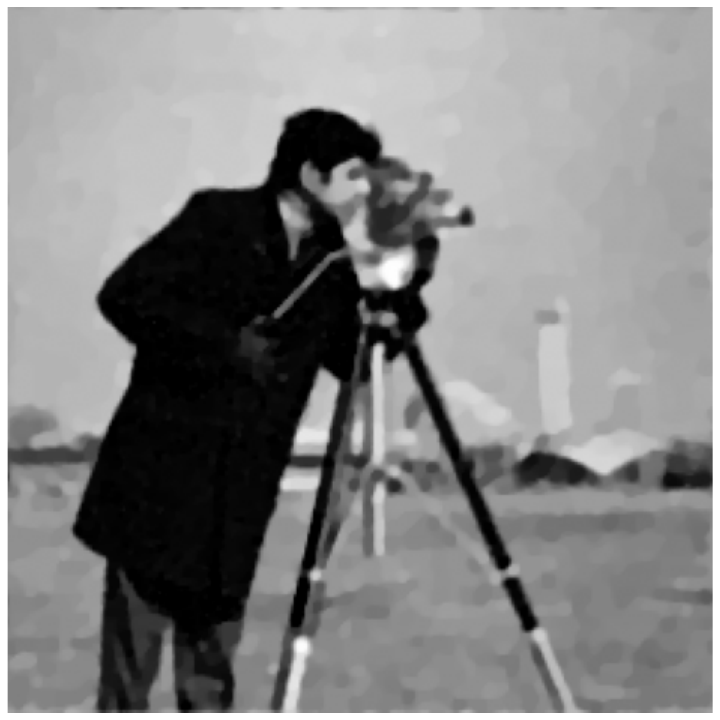}
}\hfill
\subfloat[adaPDM]{
  \includegraphics[width=0.20\linewidth]{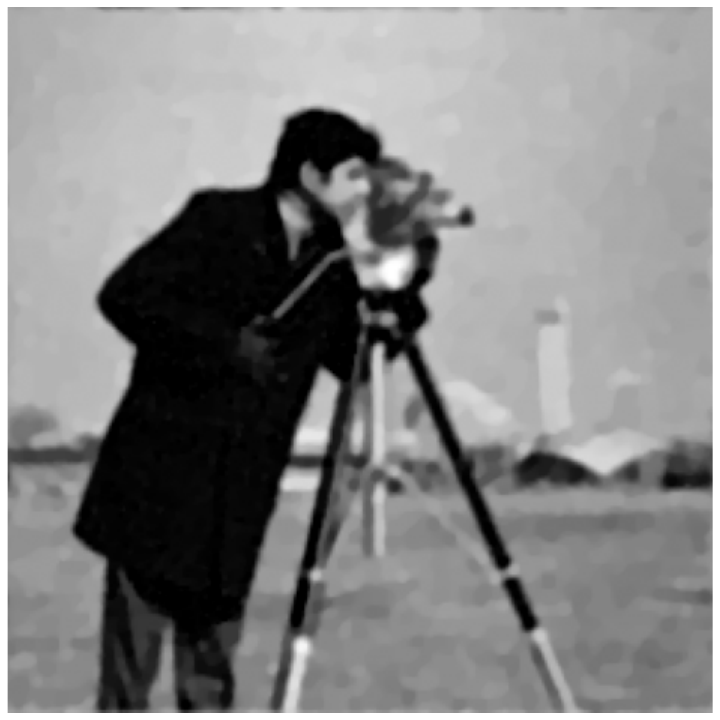}
}\hfill
\subfloat[aGRAAL]{
  \includegraphics[width=0.20\linewidth]{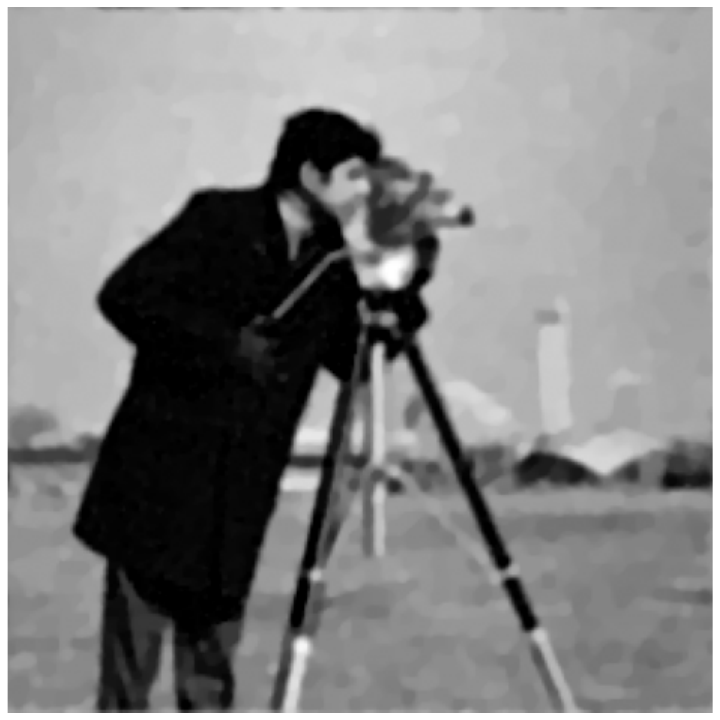}
}\hfill
\subfloat[aPDAc-L]{
  \includegraphics[width=0.20\linewidth]{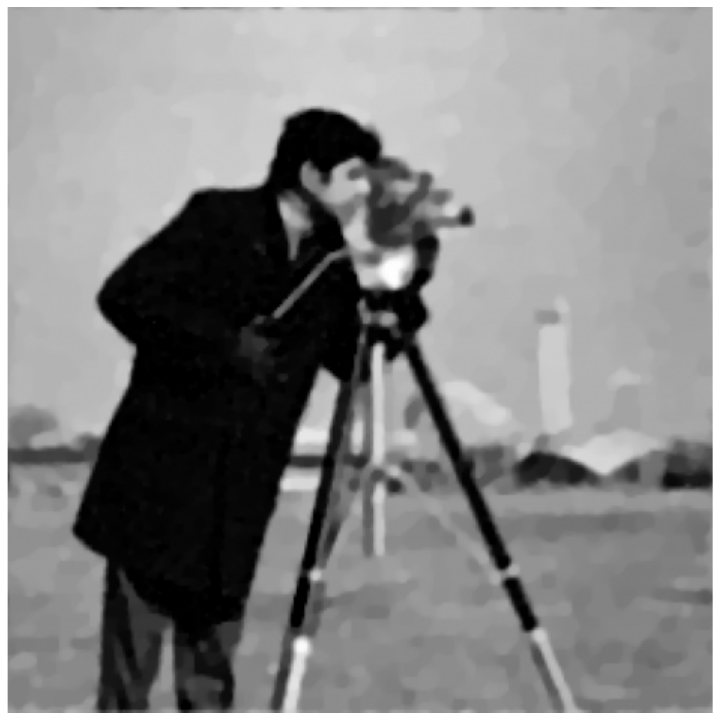}
}

\caption{Image reconstruction results for \textbf{Setting~2} using the \texttt{Cameraman} image degraded by defocus blur.}
\label{fig:recovery-setting2_1}
\end{figure}

\begin{figure}[htbp]
\centering
\subfloat[Primal objective gap]{
  \includegraphics[width=0.27\linewidth]{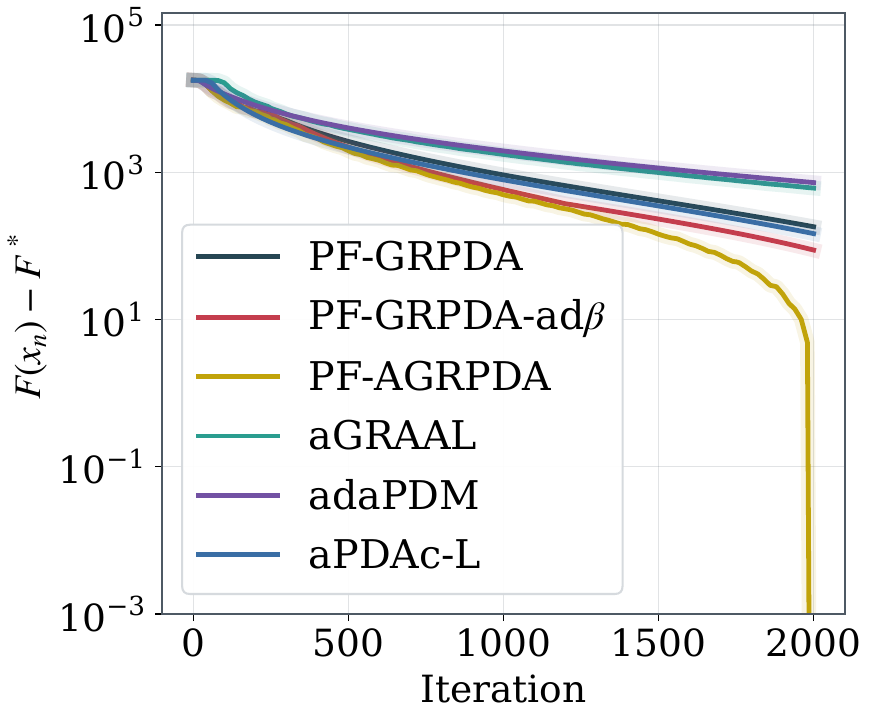}
}\hfill
\subfloat[Objective residual]{
  \includegraphics[width=0.27\linewidth]{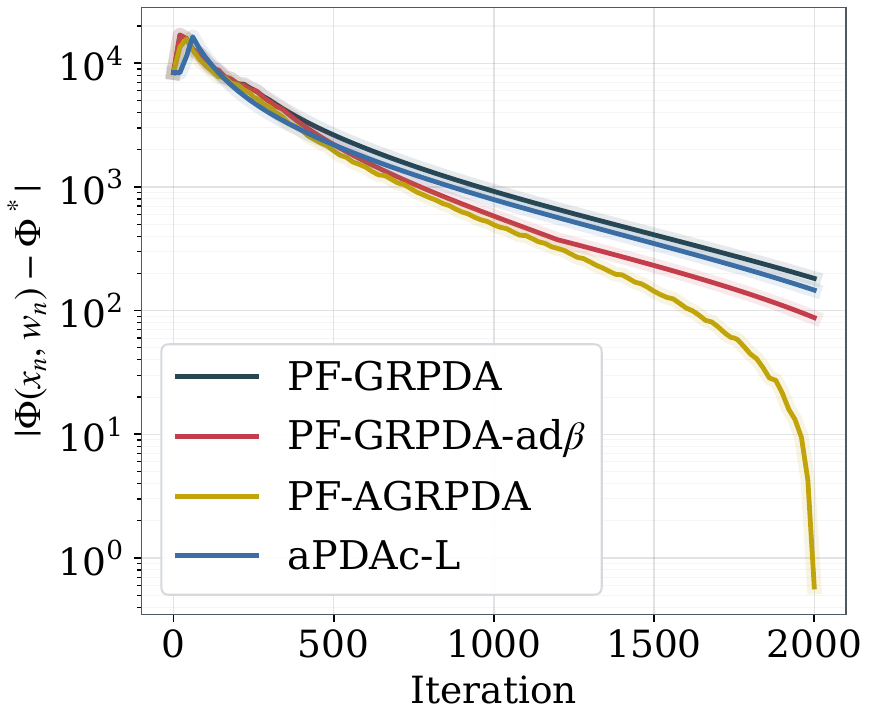}
}\hfill
\subfloat[Feasibility residual]{
  \includegraphics[width=0.27\linewidth]{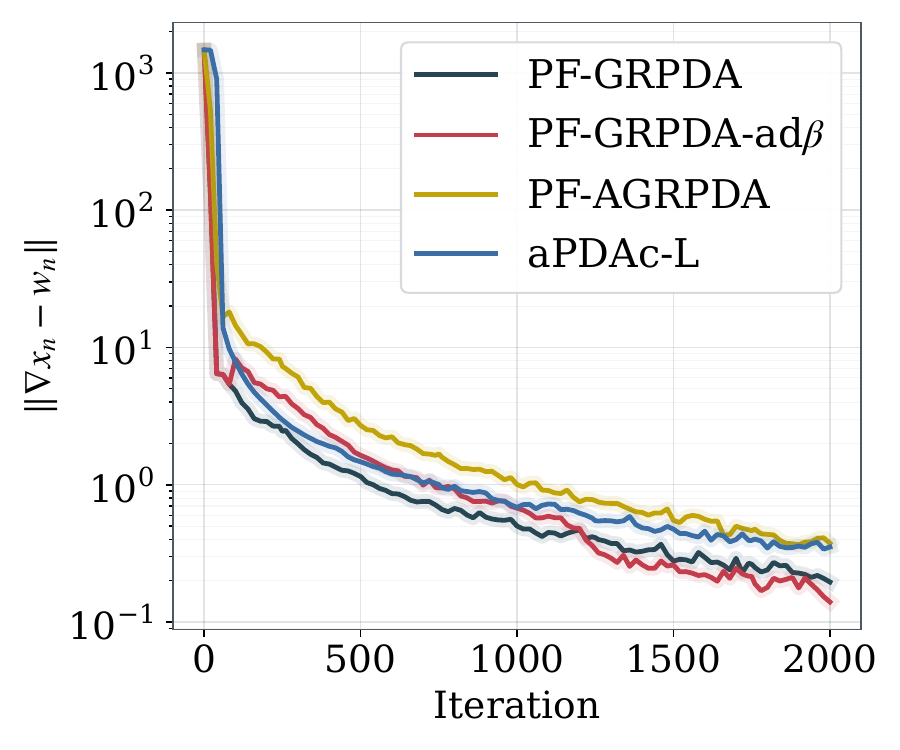}
}

\vspace{0.8em}
\subfloat[PSNR]{
  \includegraphics[width=0.27\linewidth]{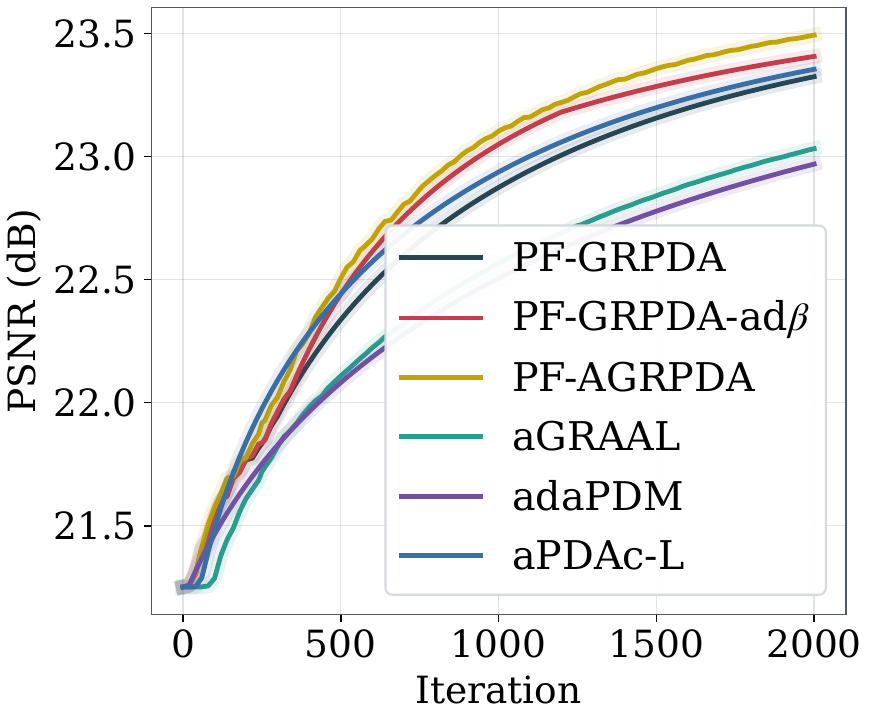}
}\hfill
\subfloat[Primal step-size $(\tau_n)$]{
  \includegraphics[width=0.27\linewidth]{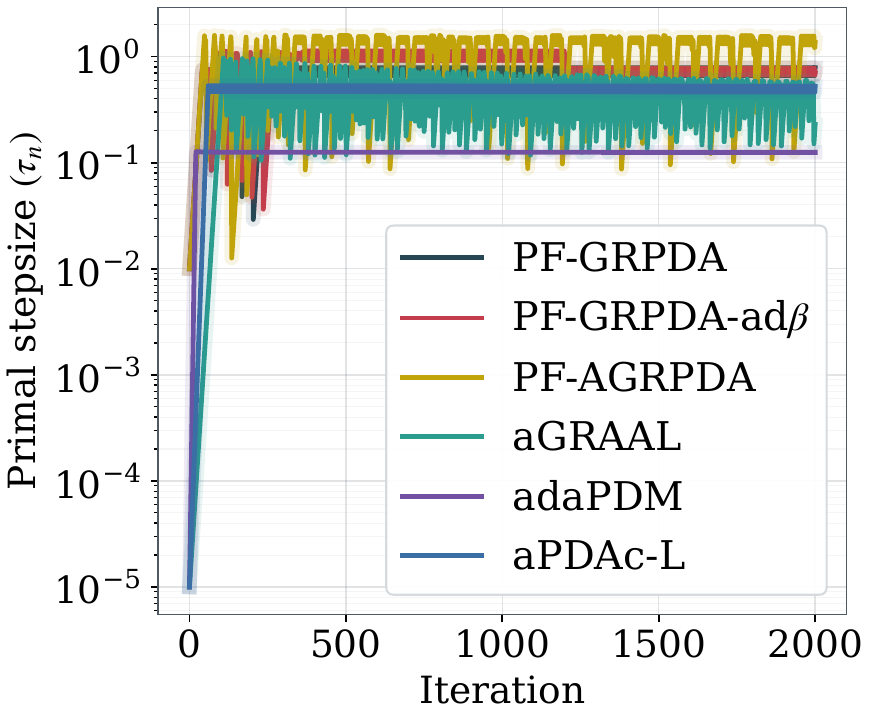}
}\hfill
\subfloat[PF-AGRPDA selected ($\tau_n$)]{
  \includegraphics[width=0.27\linewidth]{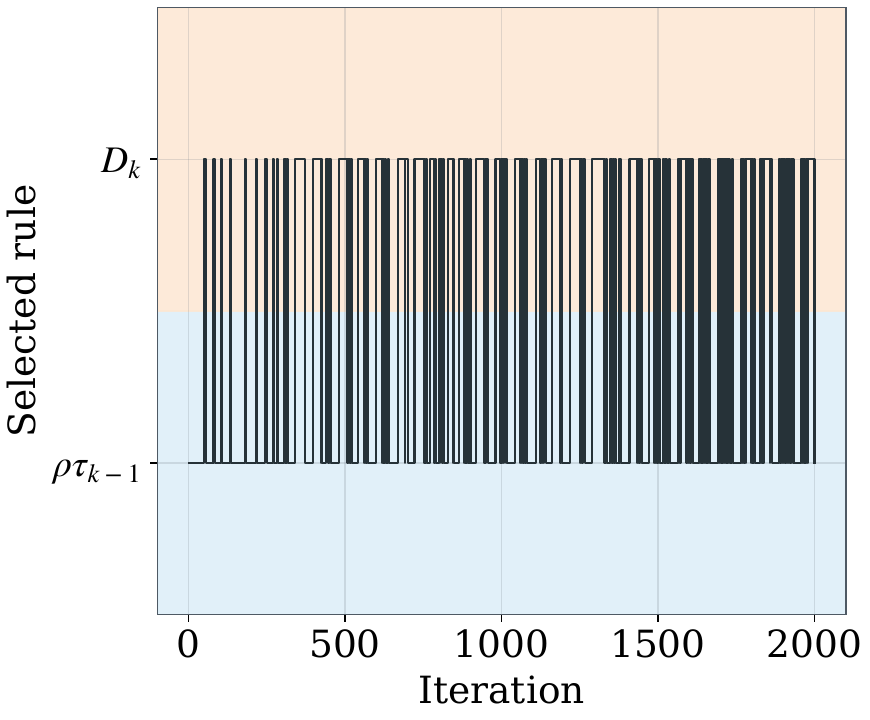}
}

\caption{Convergence results for \textbf{Setting~2} using the \texttt{Cameraman} image degraded by defocus blur.}
\label{fig:metrics-setting2}
\end{figure}
\subsubsection{Setting 3}

We now test the other accelerated algorithm (Algorithm~\ref{algorithm:hstrong_acc})
on the strongly convex Poisson--TV reconstruction problem \eqref{eq:poisson_strng_f}. Since Algorithm~\ref{algorithm:hstrong_acc}
exploits the strong convexity of the differentiable term, we place the quadratic
regularisation inside $h$.  More precisely, we set the component functions as follows.
\[
  h(x)=\mathrm{KL}(Ax,y)+\frac{\mu}{2}\|x\|^2, \quad  f(x)=\iota_{\{x\ge0\}}(x),
    \quad
    g(w)=\lambda\|w\|_{2,1},
    \quad
    K=\nabla.
\]
We recall that
\[
    \mathrm{KL}(s,y)
    =
    \sum_{i=1}^{m}
    \left[
        y_i\log\!\left(\frac{y_i}{s_i}\right)+s_i-y_i
    \right],
    \qquad s_i>0,
\]
with the usual convention that the term corresponding to $y_i=0$ is equal
to $s_i$.  On the interior of the KL domain, the gradient of $h$ is
\[
    \nabla h(x)
    =
    A^*\left(1-\frac{y}{Ax}\right)+\mu x .
\]
Moreover,
\[
    \nabla^2 h(x)
    =
    A^*\operatorname{Diag}\!\left(\frac{y_i}{(Ax)_i^2}\right)A
    +\mu I .
\]
Since the first term is positive semidefinite, we have
\[
    \nabla^2 h(x)\succeq \mu I .
\]
Thus $h$ is globally $\mu$-strongly convex on its effective domain. It is not hard to see that $h$ is not globally smooth. For this splitting, the $x$-update becomes
the projection onto the nonnegative orthant,
\[
    \prox_{\tau f}(v)=P_{\mathbb R^n_+}(v)=[v]_+,
\]
while the $w$-update is the pointwise isotropic shrinkage. We use the same blur operator, Poisson sampling rule, and initialisation
as in the previous Poisson experiments. For Algorithm \ref{algorithm:hstrong_acc}, we set $\psi=1.61$, $\theta_0=1$, and $\tau_0=0.30 U$. It is worth noting that $\bar\gamma$ computed by \eqref{eq:hstrong_gamma_bar} heavily depends on the choices of $\beta_0$ and $\mu_h$. In particular, a very small choice of these parameters can lead to selecting a small value of $\gamma$. As a result, this may leave no acceleration effect on the problem for this Algorithm, as it made $\beta_n$ much closer to the fixed $\beta_0$. Thus, after a few trial and error, we are satisfied with $\beta_0=5$, and after computing $\bar\gamma$ from \eqref{eq:hstrong_gamma_bar}, select $\gamma=0.99\,\bar\gamma$. The regularisation parameters $\lambda, \mu$ remain the same as in \textbf{Setting~2}.

\medskip
Figures~\ref{fig:recovery-setting3} and \ref{fig:convergence_results_setting3} report the numerical behaviour of Algorithm~\ref{algorithm:hstrong_acc} and compareres it with PF-GRPDA and its adaptive $\beta_n$ version on the
same strongly convex Poisson--TV model (i.e., \textbf{Setting 2} with the defocus blur).  The reconstructions are almost visually indistinguishable, and the PSNR curves also match after approximately $700$ iterations.  This is expected as \textbf{Setting~3} does not change the regularised reconstruction model in \eqref{eq:poisson_strng_f}, it only changes
the splitting so that the quadratic curvature is exploited through the differentiable term $h$. In the residual plots, the accelerated method (Algorithm \ref{primal_dual_gap}) yields a lower overall objective residual and a smaller feasibility throughout the iterations.

\begin{figure}[htbp]
\centering
\setlength{\tabcolsep}{2pt}
\renewcommand{\arraystretch}{1}

\begin{tabular}{@{}cccc@{}}
\scriptsize Noisy data
&
\scriptsize Algorithm~\ref{algorithm:hstrong_acc}
&
\scriptsize PF-GRPDA-ad$\beta$
&
\scriptsize PF-GRPDA
\\[0.3em]

\includegraphics[width=0.20\textwidth]
{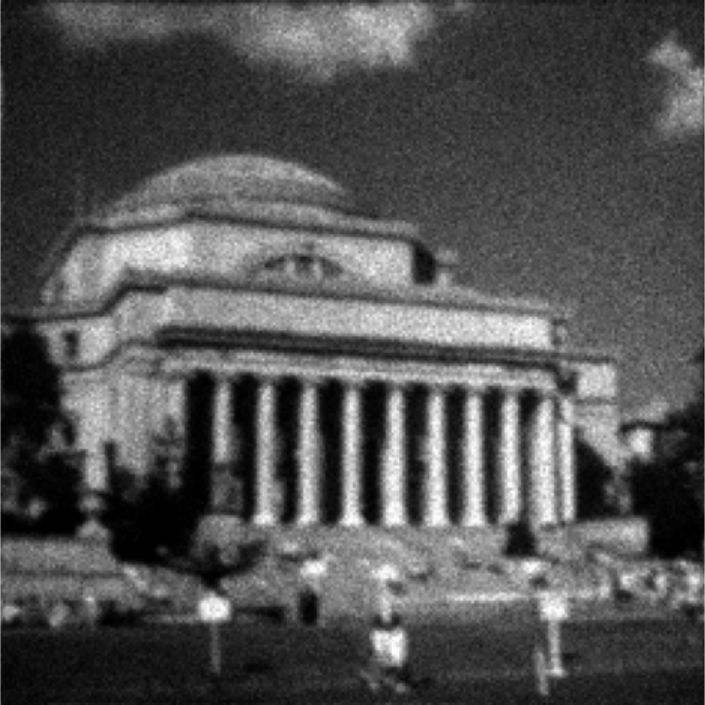}
&
\includegraphics[width=0.20\textwidth]
{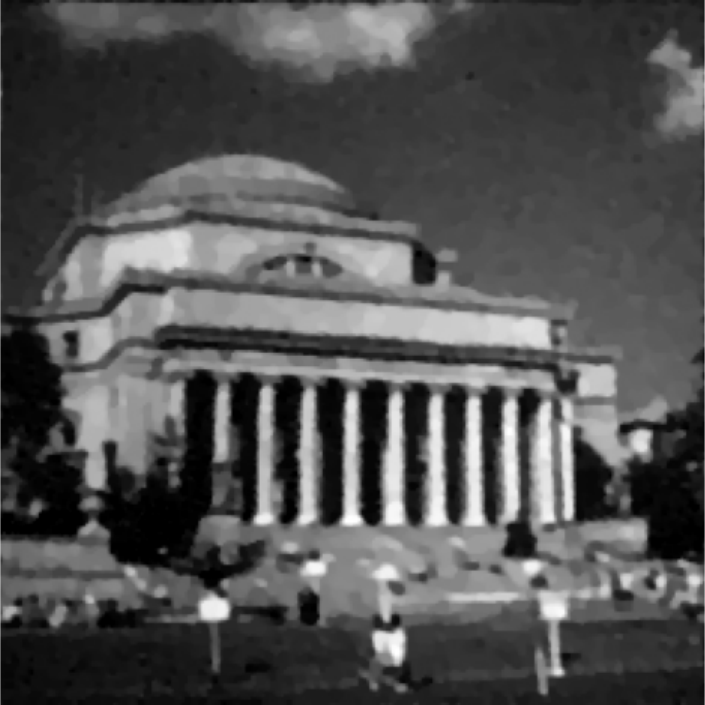}
&
\includegraphics[width=0.20\textwidth]
{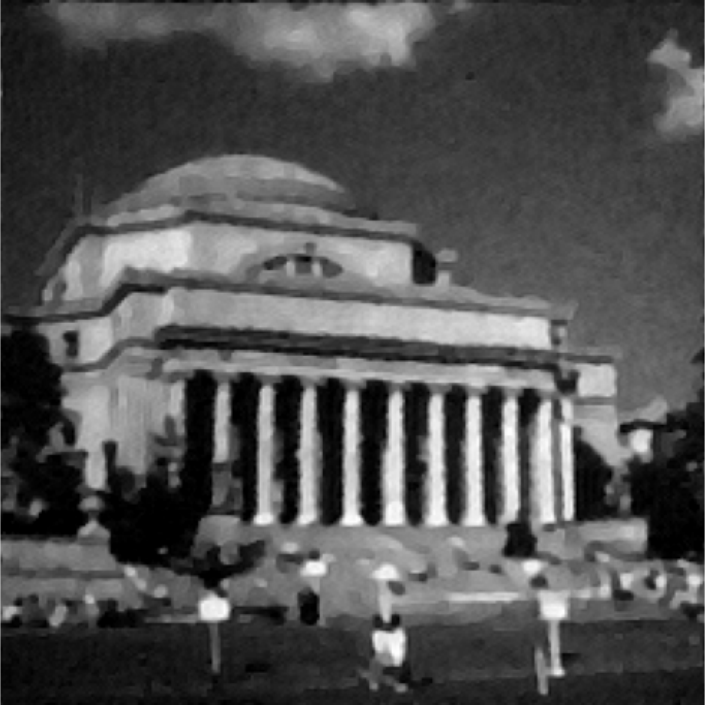}
&
\includegraphics[width=0.20\textwidth]
{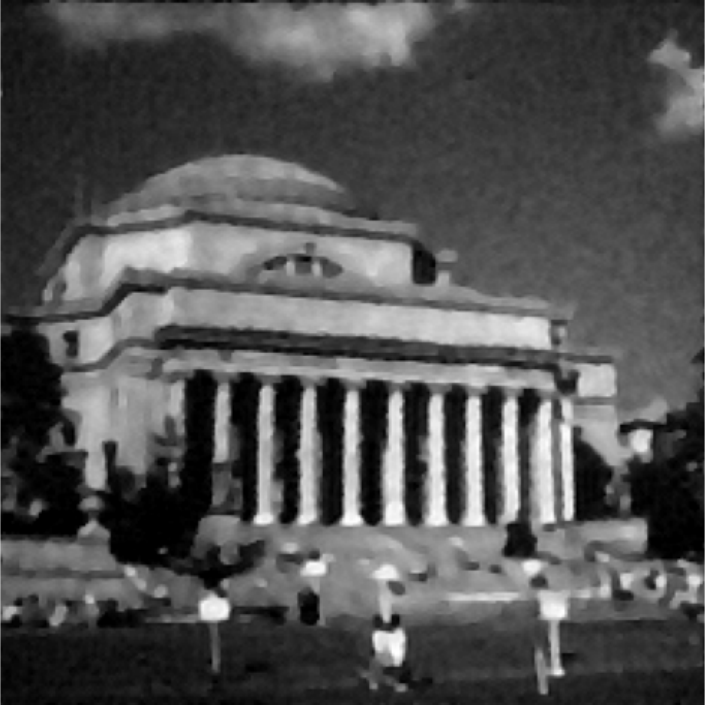}
\\[0.5em]

\includegraphics[width=0.20\textwidth]
{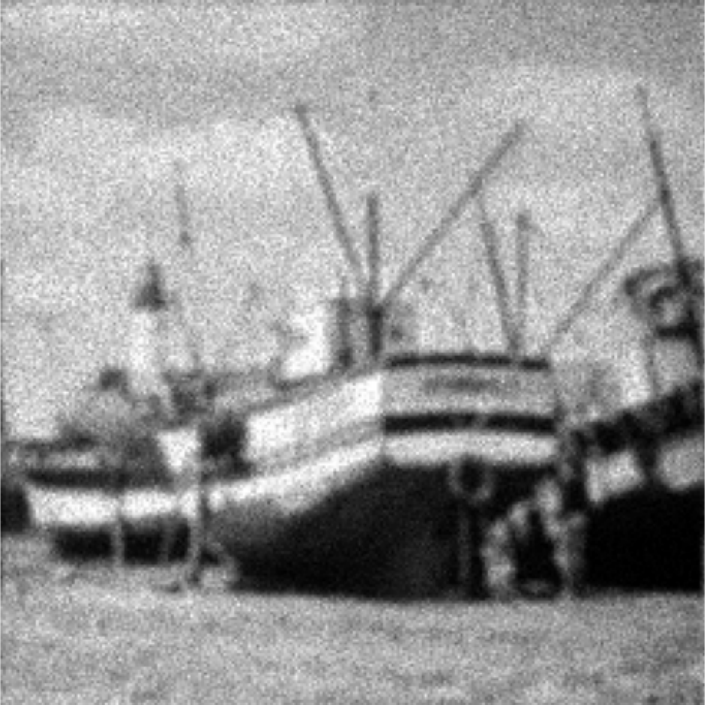}
&
\includegraphics[width=0.20\textwidth]
{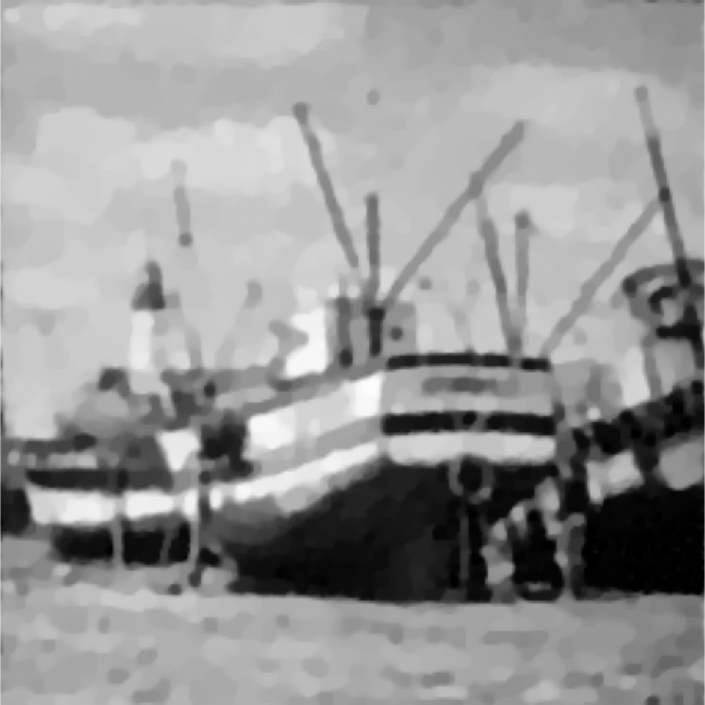}
&
\includegraphics[width=0.20\textwidth]
{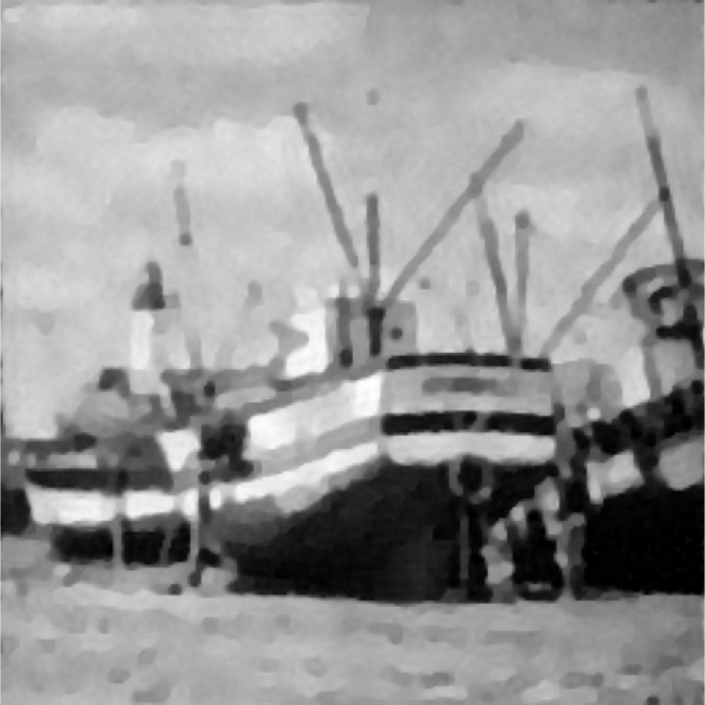}
&
\includegraphics[width=0.20\textwidth]
{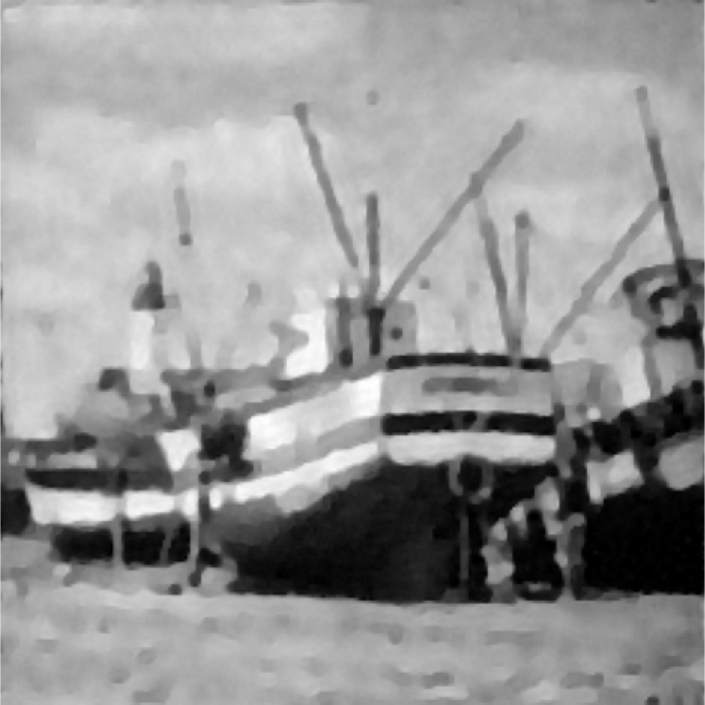}
\end{tabular}

\caption{Image reconstruction results for \textbf{Setting 3} using the \texttt{columbia} image
degraded by motion blur (top row) and the \texttt{boats} image degraded by defocus blur
(bottom row).}
\label{fig:recovery-setting3}
\end{figure}

\begin{figure}[htbp]
\centering

\subfloat[Objective residual]{
  \includegraphics[width=0.27\linewidth]
  {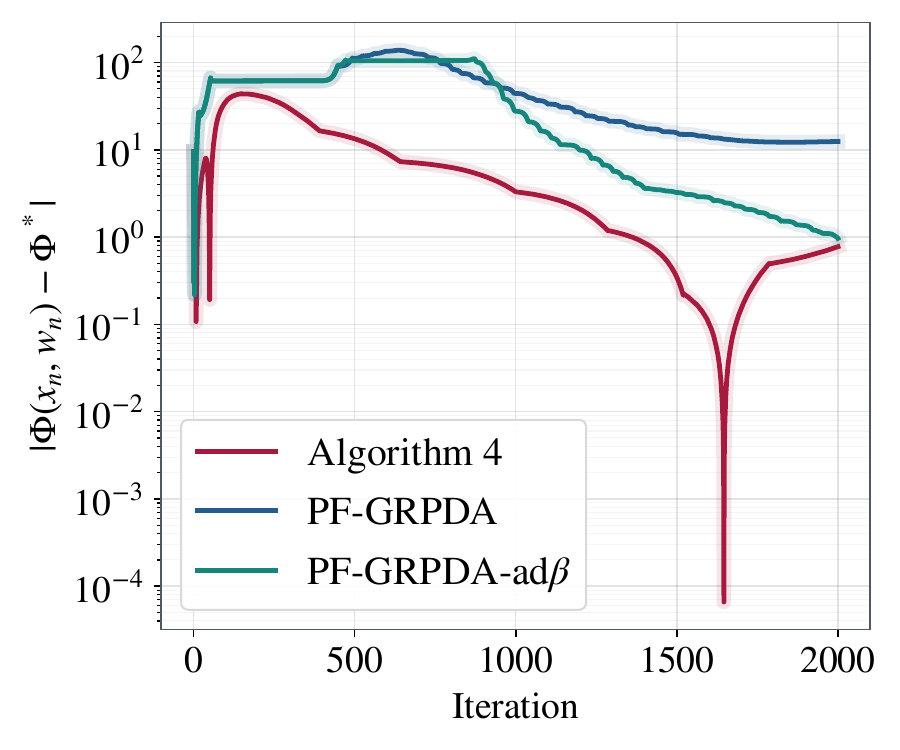}
}\hfill
\subfloat[Feasibility residual]{
  \includegraphics[width=0.27\linewidth]
  {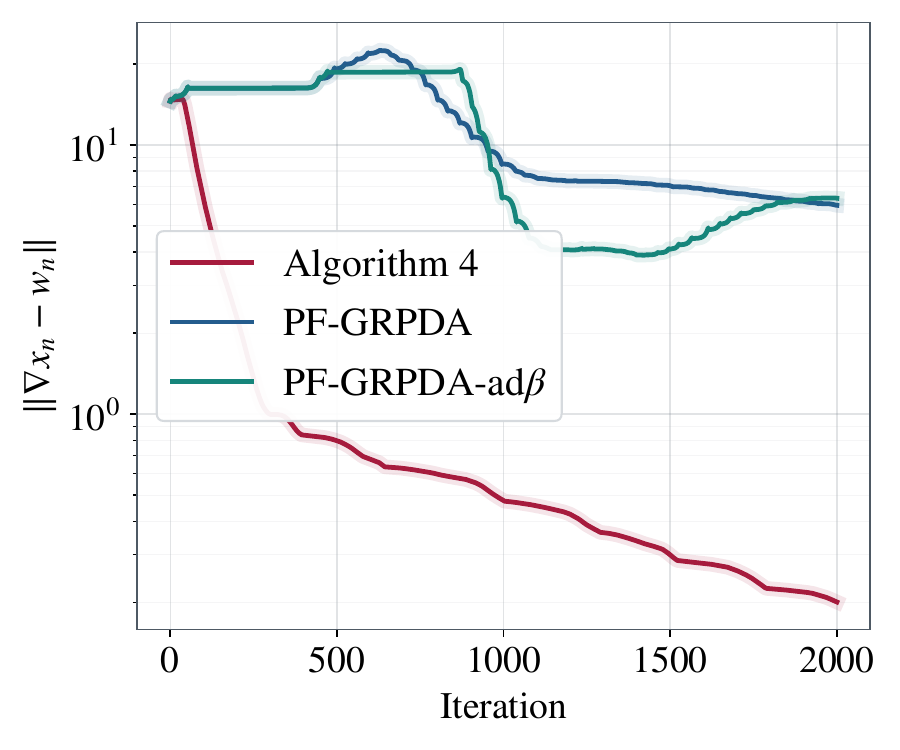}
}\hfill
\subfloat[PSNR]{
  \includegraphics[width=0.27\linewidth]
  {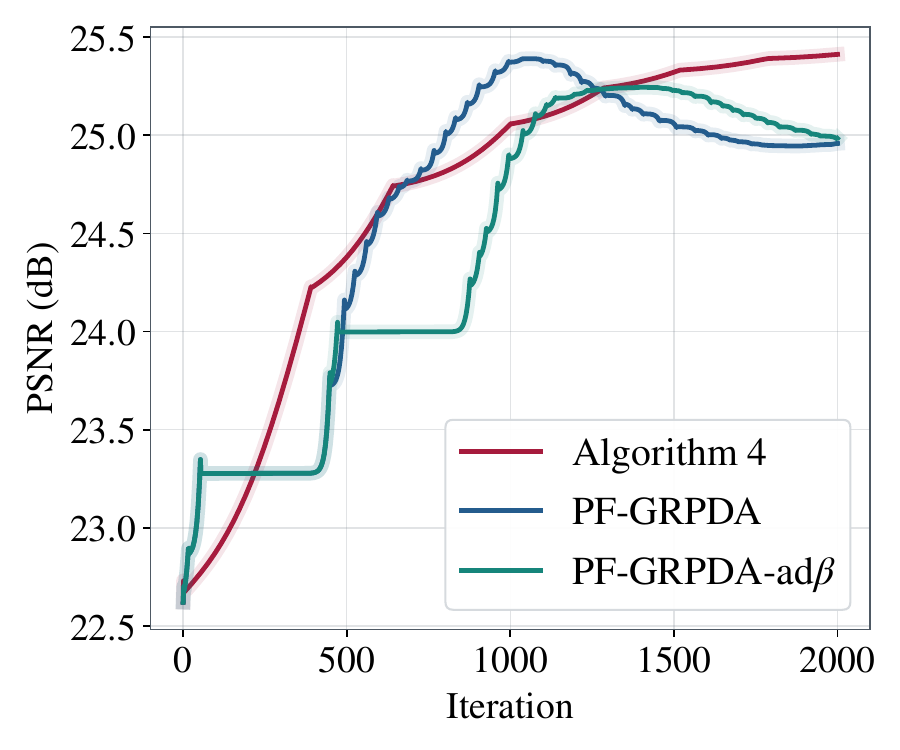}
}

\vspace{0.8em}

\subfloat[Objective residual]{
  \includegraphics[width=0.27\linewidth]
  {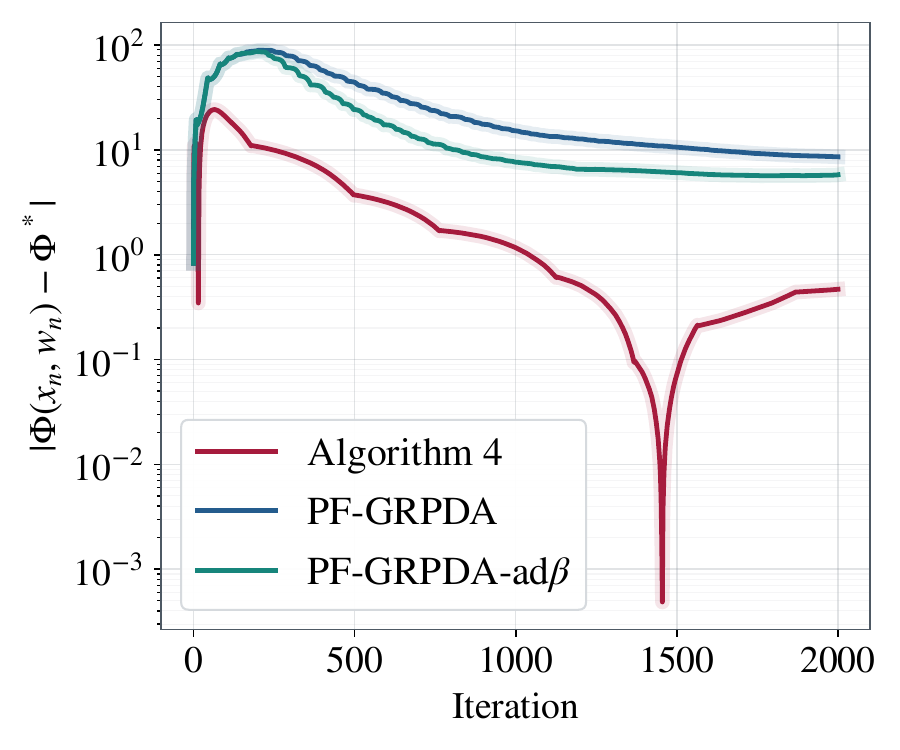}
}\hfill
\subfloat[Feasibility residual]{
  \includegraphics[width=0.27\linewidth]
  {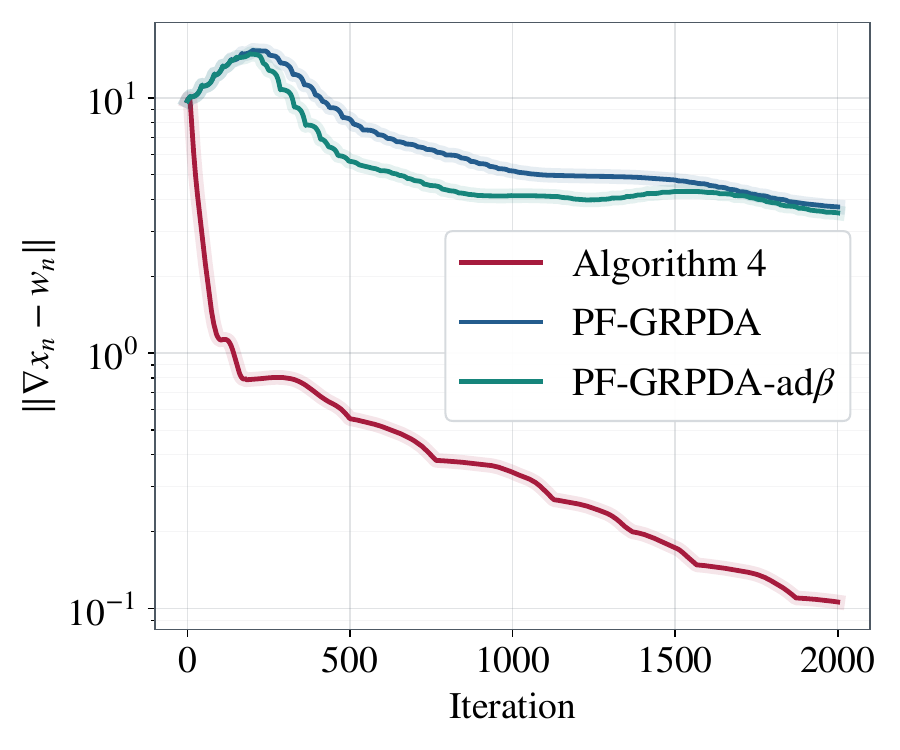}
}\hfill
\subfloat[PSNR]{
  \includegraphics[width=0.27\linewidth]
  {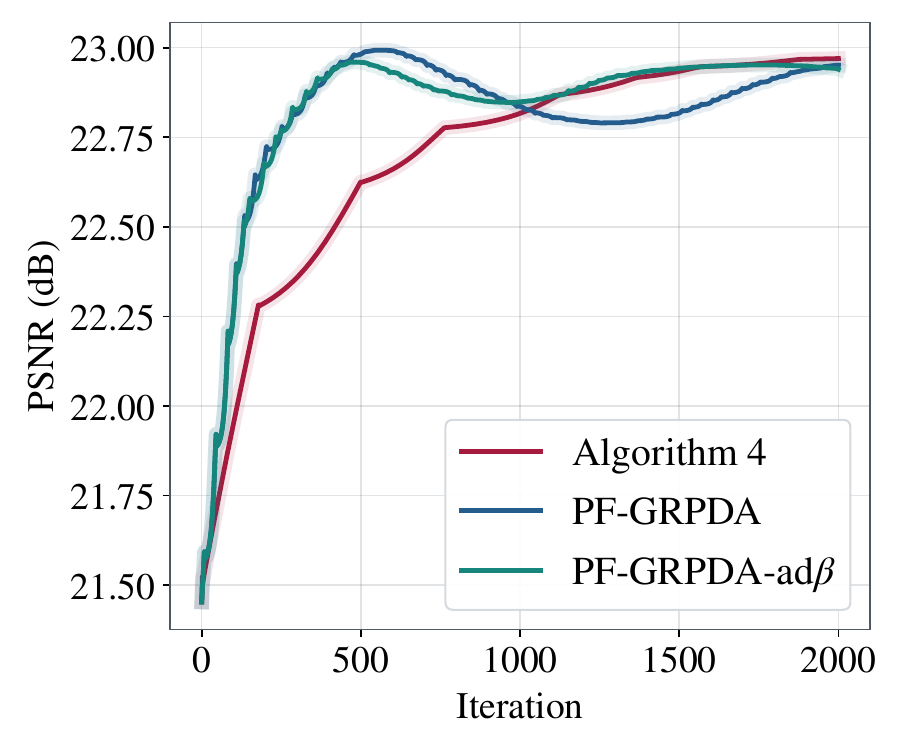}
}

\caption{Convergence results for \textbf{Setting~3} using the \texttt{columbia} image degraded
by motion blur (top row) and the \texttt{boats} image degraded by defocus blur
(bottom row).}
\label{fig:convergence_results_setting3}
\end{figure}
\section{Conclusion}\label{sec:conclusion}
In this paper, we revisited the extended adaptive golden-ratio primal--dual
algorithm (aEGRPDA) \cite{Soe2026} for the convex minimisation problem
\eqref{main_prob}, where the differentiable term $h$ is assumed to be only
locally smooth. We showed that the artificial upper bound imposed on the primal
step-size in the aEGRPDA is not needed, as the adaptive rule
itself provided the upper bound together with
$\|K\|$. As a consequence, we obtained ergodic
$\mathcal{O}(1/N)$ estimates for both the objective residual and the feasibility
violation without relying on an external step-size cap. Under additional curvature
assumptions, namely when either $f$ or $h$ is strongly convex, we also
proposed accelerated variants and established ergodic $\mathcal{O}(1/N^2)$
convergence rates. The numerical experiments on Poisson image reconstruction
illustrated the practical advantages of the proposed accelerated methods.

Several directions remain open and deserve further investigation.

\begin{itemize}
    \item \textbf{Bregman accelerated variants.}
    A natural direction is to develop Bregman versions of the proposed
    algorithms. Such an extension would allow the Euclidean proximal
geometry to be replaced by a geometry that is better adapted to the structure
of the problem. More precisely,
    for a suitable kernel $\omega$, and  $x,y\in\mathbb H_1$, one may use the Bregman distance
    \[
        D_{\omega}(x,y)
        :=
        \omega(x)-\omega(y)-\langle \nabla\omega(y),x-y\rangle .
    \]
    Such a distance can be computationally advantageous when an appropriate Bregman kernel $\omega$ makes the Bregman proximal operator cheaper
    than the Euclidean one. For instance, the projection onto the probability simplex requires $\mathcal{O}(n)$ times, whereas the Euclidean projection takes $\mathcal{O}(n\log n)$ time. One can follow the works of \cite{Soe2026,tam2023bregman}, which provide a useful starting point for developing such Bregman-accelerated methods.

\medskip
    \item \textbf{Adaptive local curvature estimates.}
    In Algorithm~\ref{algorithm:hstrong_acc}, the acceleration relies on the local
    smoothness and global strong convexity of $h$. In particular, we used 
    the following inequality in the analysis of Algorithm \ref{algorithm:hstrong_acc},
    \[
        \frac{\mu}{2}\|x-y\|^2
        \le
        D_h(x,y)
        \le
        \frac{L}{2}\|x-y\|^2,
        \qquad x,y\in \mathbb{H}_1,
    \]
  In the present setting of Algorithm \ref{algorithm:hstrong_acc}, the
    Lipschitz constant of $\nabla h$, is estimated along the iterates. A more
    delicate question is whether the strong convexity constant $\mu$ can also
    be estimated adaptively. Such an extension would be useful, but it would require care because the local estimate of $\mu$ may directly affect the acceleration parameter and the descent arguments.

\medskip

\item \textbf{Separable convex optimisation.}
Another interesting direction is to revisit the separable convex optimisation framework associated with
\eqref{eq:constrained_form} in the special case $h=0$. In \cite{soe2026golden}, two golden-ratio proximal ADMM-type algorithms have been proposed for locally estimating the operator norm in this setting. However, the resulting step-size rules are monotone in nature. Since non-monotone step-sizes can often exploit the local behaviour of the problem more effectively,
it would be worthwhile to investigate whether the ideas developed in \cite{Soe2026} and in the present
work can be adapted to this separable setting.

\medskip
    \item \textbf{Distributed optimisation.}
    Another promising direction is to extend these ideas to distributed composite optimisation problems. The work \cite{yin2024golden} provides a golden-ratio framework for distributed problems in which agents have access to local objective functions. It would be interesting to investigate whether the local-smoothness-based step-size rule and the acceleration mechanisms developed in this paper can be incorporated into such distributed algorithms, especially when some of the local functions are only locally smooth or only locally strongly convex.

\end{itemize}

\backmatter

\section*{Acknowledgements}

The authors gratefully acknowledge A/Prof. Matthew K. Tam for the comments on this paper. Santanu Soe expresses his gratitude to A/Prof. Matthew K. Tam for his encouragement, support, and guidance throughout his PhD.

\subsection*{Funding}

The research of Santanu Soe was supported by the Prime Minister's Research Fellowship program (Project number SB23242132MAPMRF005015), the Ministry of Education, Government of India, and the Melbourne Research Scholarship.

\subsection*{Data availability}

The test images used in the numerical experiments are publicly accessible
from \href{https://www.dip.ee.uct.ac.za/imageproc/stdimages/greyscale/}
{https://www.dip.ee.uct.ac.za/imageproc/stdimages/greyscale/} and \href{https://sipi.usc.edu/database/database.php?volume=misc}
{https://sipi.usc.edu/database/database.php?volume=misc}. The Python code used in this paper to reproduce the numerical results is available from the corresponding author upon reasonable request.

\section*{Ethics Declarations}

\subsection*{Competing interests}

The authors declare that there are no conflicts of interest in this paper.

\subsection*{Author contributions}

All authors contributed to the conception, analysis, and writing of the manuscript. All authors read and approved the final manuscript.

\subsection*{Ethical approval}

This article does not contain any studies involving human participants or animals performed by any of the authors.

\bibliography{my_bib}

\end{document}